\documentclass[11pt]{amsart}
\usepackage{amsmath,amssymb,verbatim,eucal, amscd,color}
\usepackage{graphicx}
\usepackage{colordvi}
\usepackage{mathdots}
\usepackage{xypic}
\usepackage[colorlinks=true, pdfstartview=FitV,
linkcolor=blue, citecolor=blue, urlcolor=red]{hyperref}
\usepackage{amsfonts,latexsym}

\usepackage{geometry,tikz,amsmath}

\usetikzlibrary{fit,matrix}
 
\newcommand*{\Scale}[2][4]{\scalebox{#1}{$#2$}}%

\usepackage{yhmath}
\usepackage{accents}

\input xy

\numberwithin{equation}{section}

\theoremstyle{plain}
				
\newtheorem{theorem}{Theorem}[section]				
\newtheorem{proposition}[theorem]{Proposition}		
\newtheorem{corollary}[theorem]{Corollary}
\newtheorem{lemma}[theorem]{Lemma}

\newtheorem{question}[theorem]{Question}

\theoremstyle{definition}

\newtheorem{definition}[theorem]{Definition}
\newtheorem{remark}[theorem]{Remark}

\newcommand\frg{\mathfrak{g}}

\newcommand{\CBbb}{\mathbb C}

\newcommand{\HBbb}{\mathbb H}

\newcommand{\PBbb}{\mathbb P}

\newcommand{\ZBbb}{\mathbb Z}

\newcommand{\Acal}{\mathcal A}

\newcommand{\Ccal}{\mathcal C}
\newcommand{\Dcal}{\mathcal D}
\newcommand{\Ecal}{\mathcal E}

\newcommand{\Gcal}{\mathcal G}
\newcommand{\Hcal}{\mathcal H}

\newcommand{\Lcal}{\mathcal L}
\newcommand{\Mcal}{\mathcal M}

\newcommand{\Ocal}{\mathcal O}
\newcommand{\Pcal}{\mathcal P}
\newcommand{\Qcal}{\mathcal Q}

\newcommand{\Scal}{\mathcal S}
\newcommand{\Tcal}{\mathcal T}

\newcommand{\Vcal}{\mathcal V}

\newcommand{\Xcal}{\mathcal X}
\newcommand{\Ycal}{\mathcal Y}

\newcommand{\Xfrak}{\mathfrak X}

\newcommand{\gfrak}{\mathfrak g}

\newcommand{\sfrak}{\mathfrak s}

\newcommand{\SL}{\mathsf{SL}}

\newcommand{\GL}{\mathsf{GL}}

\newcommand{\Psf}{\mathsf{P}}
\newcommand{\SO}{\mathsf{SO}}

\newcommand{\Spin}{\mathsf{Spin}}

\newcommand{\slfrak}{\mathfrak{sl}}

\newcommand{\sofrak}{\mathfrak{so}}

\newcommand{\SC}{\mathsf{SC}}

\newcommand{\OG}{\mathsf{OG}}

\DeclareMathOperator{\End}{End}

\DeclareMathOperator{\id}{id}

\DeclareMathOperator{\rank}{rank}

\DeclareMathOperator{\Ad}{Ad}

\DeclareMathOperator{\Div}{div}

\DeclareMathOperator{\Nm}{Nm}
\DeclareMathOperator{\Det}{Det}
\DeclareMathOperator{\pr}{pr}

\newcommand{\lra}{\longrightarrow}

\newcommand{\Pic}{{\rm Pic}}

\newcommand{\Th}{{\rm Th}}

\newcommand{\isorightarrow}{\xrightarrow{
   \,\smash{\raisebox{-0.5ex}{\ensuremath{\sim}}}\,}}

\newcommand{\llrrbracket}[1]{
  \left[\mkern-2mu\left[#1\right]\mkern-2mu\right]}
\newcommand{\llrrparen}[1]{
  \left(\mkern-3mu\left(#1\right)\mkern-3mu\right)}

\makeatletter
\newcommand*\bigcdot{\mathpalette\bigcdot@{.5}}
\newcommand*\bigcdot@[2]{\mathbin{\vcenter{\hbox{\scalebox{#2}{$\m@th#1\bullet$}}}}}
\makeatother

\makeatletter
\newcommand{\thickcolon}{\mathpalette\thick@colon\relax}
\newcommand{\thick@colon}[2]{%
  \mspace{1mu}%
  \vbox{%
    \hbox{$\m@th#1\bigcdot$}
    \nointerlineskip
    \kern.15ex
    \hbox{$\m@th#1\bigcdot$}
    \kern-.55ex
  }%
  \mspace{1mu}%
}
\makeatother

\newcommand{\normalorder}[1]{\thickcolon\mathop{#1}\thickcolon}

\newcommand\MSO {\mathcal{M}_{\SO(m)}}

\newcommand\Res{\operatorname{Res}}

\AtBeginDocument{%
  \def\MR#1{}
  }

\begin{document}
\title[Strange duality and odd orthogonal bundles on curves]{Generalized theta functions, strange duality,  and odd orthogonal bundles on curves}
\author{Swarnava Mukhopadhyay}
\author{Richard Wentworth}
 \thanks{S.M. was supported in part by a Simons Travel Grant and by  NSF grant DMS-1361159  (PI: Patrick Brosnan). R.W. was supported in part by NSF  grant DMS-1406513.
  The authors also acknowledge support from NSF grants DMS-1107452, -1107263, -1107367 ``RNMS: GEometric structures And Representation varieties'' (the GEAR Network).
  } 
\address{School of Mathematics, Tata Institute of Fundamental Research, Colaba, Mumbai 400005}
\email{swarnava@math.tifr.res.in}
\address{Department of Mathematics\\ University of Maryland
\\ College Park, MD 20742}

\email{raw@umd.edu}
\subjclass[2010]{Primary  14H60, Secondary 17B67,  32G34, 81T40}
\begin{abstract} This paper studies  spaces of generalized theta functions for odd orthogonal bundles with nontrivial Stiefel-Whitney class and  the associated space of twisted spin  bundles. In particular, we prove a Verlinde type formula and a dimension equality that was conjectured by Oxbury-Wilson.
Modifying Hitchin's argument, we also show that the bundle  of generalized theta functions for twisted spin bundles over the moduli space of curves admits a  flat projective connection. We furthermore address the issue of strange duality for odd orthogonal bundles, and we demonstrate that the naive conjecture fails in general. A consequence of this is the reducibility of the projective representations of spin  mapping class groups arising from the Hitchin connection for these moduli spaces. Finally, we answer a question of Nakanishi-Tsuchiya about rank-level duality for conformal blocks on the pointed projective line with spin weights. 
\end{abstract}

\maketitle

\allowdisplaybreaks


\section{Introduction}
Let $C$ be a smooth projective curve of genus $g\geq 2$,  and choose integers $n\geq 2$, $\ell\geq 1$.
Let $M_{\SL(n)}$ denote the coarse moduli space of semistable vector bundles of rank $n$ and trivial determinant on $C$, and let $\mathcal{L}$ be the ample generator of the Picard group  $\Pic(M_{\SL(n)})\simeq\ZBbb$. Similarly,  let $M_{\GL(\ell)}$ denote the moduli space of semistable vector bundles of rank $\ell$ and degree $\ell(g-1)$, and consider the locus $\Theta_\ell \subset M_{\GL(\ell)}$ of points $[\Ecal]\in M_{\GL(\ell)}$ such that $H^0(C,\Ecal)\neq 0$. It turns out that $\Theta_\ell$ is a Cartier divisor in $M_{\GL(\ell)}$, and we use the same notation for the associated line bundle. Tensor product defines a map:
$$s:M_{\SL(n)} \times M_{\GL(\ell)} \longrightarrow M_{\GL(n\ell)}\ , $$ and by the ``see-saw" principle it is easy to see that $s^*\Theta_{n\ell}\simeq \mathcal{L}^{\otimes \ell}\boxtimes \Theta_\ell^{\otimes n}$. The pull-back of the defining section of $\Theta_{n\ell}$ gives a map, well-defined up to a multiplicative constant,
$$s_{n\ell}:H^0(M_{\SL(n)},\mathcal{L}^{\otimes \ell})^*\longrightarrow H^0(M_{\GL(\ell)},\Theta_\ell^{\otimes n})\ ,$$ 
 known as the \emph{strange duality map}. It was conjectured to be an isomorphism (cf.\ Donagi-Tu \cite{DonagiTu:94} and Beauville \cite{Beauville:95}), and this conjecture was confirmed independently by  Belkale \cite{Belkale:08} and by Marian-Oprea \cite{MarianOprea:07} (cf.\ Beauville-Narasimhan-Ramanan \cite{BNR} for $\ell=1$). The analogous strange duality for symplectic bundles was conjectured by Beauville \cite{BORTHO} and proven by  Abe \cite{Abe:08} (see also result of Belkale \cite{Belkale:12}). Strange duality for maximal subgroups of ${\mathsf E}_8$ has been considered independently by Boysal-Pauly \cite{BoysalPauly:10} and by the first author \cite{BoysalPauly:10, Mukhopadhyay:15}. 
However, a conjectural description of  strange duality for other dual pairs, e.g.\ orthogonal bundles, has as yet not been formulated in the literature. 

An approach to strange duality questions, and in fact the original motivation,  comes from the study of the space $\mathcal{V}^{*}_{\vec\lambda}(\Xfrak,\gfrak,\ell)$ of  \emph{conformal blocks} 
(cf.\ Tsuchiya-Ueno-Yamada \cite{TUY:89} and  Definition \ref{definitionconformalblocks} below). These are dual spaces to  quotients of tensor products of level $\ell$ integrable highest weight modules of the affine Kac-Moody algebra $\widehat\gfrak$ associated to a simple Lie algebra $\gfrak$, and with weights $\vec\lambda=(\lambda_1,\ldots, \lambda_n)$ attached to the curve $\Xfrak= (C, p_1,\ldots, p_n)$  with marked points $p_i$. Isomorphisms between spaces of conformal blocks can sometimes arise from \emph{conformal embeddings} of affine Lie algebras (cf.\ Kac--Wakimoto \cite{KacWakimoto:88} and Definition \ref{def:conformal-embedding} below), and this phenomenon is known in the conformal field theory literature as \emph{rank-level duality}
  (cf.\ Naculich--Schnitzer \cite{NaculichSchnitzer:90} and Nakanishi--Tsuchiya \cite{ NakanishiTsuchiya:92}). By a factorization or sewing procedure (see Sections \ref{sec:conformal-block} and \ref{factorizationlemma}), one can often reduce strange duality questions for curves of positive genus to rank-level duality on $\mathbb{P}^1$ with marked points. Indeed, all known strange dualities can be proved using this approach. In \cite{Mukhopadhyay:12}, the first author proved a rank-level duality for $\gfrak=\sofrak(2r+1)$ conformal blocks on $\mathbb{P}^1$ with marked points and weights associated to representations of the group $\SO(2r+1)$. One would naturally like to investigate whether the result can be generalized to curves of positive genus to give a strange duality for orthogonal bundles. This question forms the starting point of the present work.

As we shall see below, any generalization of rank-level or strange dualities for orthogonal groups is complicated by the existence of spin representations (in the former case) and the fundamental group (in the latter).   Spin weights cause difficulty in the branching rules for highest weight representations under embeddings.
This issue was already raised in the discussion in \cite{NakanishiTsuchiya:92}, and for this reason   only vector representations were considered in \cite{Mukhopadhyay:12}. On the geometric side, since $\SO(m)$ is not simply connected the moduli spaces for orthogonal groups will be disconnected, and any reasonable approach to strange duality  must  take into account all components.  It was this observation that led to the conjectural Verlinde type formula of Oxbury-Wilson \cite{OxburyWilson:96}, which is proved below.

In this paper, we discuss these issues for the conformal embeddings of the odd orthogonal algebras $\sofrak(2r+1)$.
The next subsections summarize the  results we have obtained.

\subsection{Twisted moduli spaces and uniformization} For a complex reductive group $G$, let
 $\mathcal{M}_G$ denote the moduli stack  of principal $G$-bundles on $C$. 
Consider the natural map $\Spin(m) \times \Spin(n)\rightarrow \Spin(mn)$ induced by tensor product of vector spaces of dimensions $m$ and $n$, each endowed with a symmetric nondegenerate bilinear form. This  map induces one between the corresponding moduli stacks $\mathcal{M}_{\Spin(m)}\times \mathcal{M}_{\Spin(n)}\rightarrow \mathcal{M}_{\Spin(mn)}$. 
If we pull back any section of $H^0(\mathcal{M}_{\Spin(mn)},\mathcal{P})$, we get a map 
$$H^0(\mathcal{M}_{\Spin(m)}, \mathcal{P}_1^{\otimes n})^*\longrightarrow H^0(\mathcal{M}_{\Spin(n)},\mathcal{P}_2^{\otimes m})\ ,$$
 where $\mathcal{P}$, $\mathcal{P}_1$ and $\mathcal{P}_2$ are the ample generators of the respective Picard groups of the moduli stacks, which are given by Pfaffian line bundles. By the Verlinde formula (cf.\ \cite[Cor.\ 9.8]{Beauville:93}), it is easy to find $m$ and $n$ for which
 \begin{equation} \label{eqn:spin-dim}
 \dim_\CBbb H^0(\mathcal{M}_{\Spin({m})}, \mathcal{P}_1^{\otimes n})\neq \dim_\CBbb H^0(\mathcal{M}_{\Spin({n})}, \mathcal{P}_2^{\otimes m})\ ,
 \end{equation}
 and hence there can be no obvious strange duality for spin bundles. Nevertheless, following suggestions of Oxbury-Wilson \cite{OxburyWilson:96}, we can attempt to rectify this situation by considering orthogonal bundles that do not lift to spin.

Fix $p\in C$,  and let  $\mathcal{M}_{\Spin(m)}^{-}$, $m\geq 5$, denote the moduli stack of special Clifford bundles whose spinor norm is $\mathcal{O}_C(p)$ (cf.\ Section \ref{twistedmoduli} and the discussion around eq.\ \eqref{eqn:projection}). We refer to these objects as \emph{twisted spin bundles}: their associated orthogonal bundles have nontrivial Stiefel-Whitney class.
A uniformization theorem for these moduli stacks was proved in Beauville-Laszlo-Sorger \cite{BLS:98}, and there is again a Pfaffian line bundle $\Pcal\to \mathcal{M}_{\Spin(m)}^{-}$  which generates the Picard group. 
Now if $G$ is simply connected and $\Lcal\to \Mcal_G$ is the ample generator of $\Pic(\Mcal_G)$, then $H^0(\mathcal{M}_G,\mathcal{L}^{\otimes\ell})$ is canonically identified with the space of conformal blocks $\mathcal{V}^{*}_{\omega_0}(\mathfrak{X},\frg,\ell)$.
We prove the analog of this result in the twisted case.
\begin{theorem}\label{twistedconformal}
The space $H^0(\mathcal{M}^{-}_{\Spin(m)},\mathcal{P}^{\otimes \ell})$ is naturally isomorphic to the space of conformal blocks $\mathcal{V}^{*}_{\ell\omega_1}(\mathfrak{X},\mathfrak{so}(m),\ell)$. 
\end{theorem}
\noindent In particular, from the Verlinde formula and results in \cite{Mukhopadhyay:12}, we obtain an expression for the dimension  of $H^0(\mathcal{M}^{-}_{\Spin(m)},\mathcal{P}^{\otimes\ell})$ that was first conjectured to hold in \cite{OxburyWilson:96} (see Theorem \ref{oxbury} below).

 Next, we observe the following. Let 
\begin{equation} \label{eqn:mr}
\mathcal{M}_{2r+1}=\mathcal{M}_{\Spin(2r+1)}\sqcup \mathcal{M}^{-}_{\Spin(2r+1)}\ ,
\end{equation}
and denote also by $\Pcal$ the bundle which restricts to the Pfaffian on each component.
Then we prove the following equality.
\begin{corollary} \label{cor:spin-dimension}
$\dim_{\mathbb{C}}H^0(\mathcal{M}_{2r+1},\mathcal{P}^{\otimes (2s+1)})=\dim_{\mathbb{C}}H^0(\mathcal{M}_{2s+1},\mathcal{P}^{\otimes(2r+1)})$.
\end{corollary}

\subsection{Hecke transformations} Let $\mathcal{M}^{par}_{\Spin(m)}$ be the moduli stack of pairs $(S,\Psf)$, where $S\to C$ is a $\Spin(m)$ bundle and $\Psf$ is a maximal parabolic subgroup of the fiber $S_{p}$ preserving an isotropic line in the associated orthogonal bundle. 
A theorem of Laszlo-Sorger \cite{LaszloSorger:97} states that $H^0(\mathcal{M}^{par}_{\Spin(m)},\mathcal{P}(\ell))$ is naturally isomorphic to $\mathcal{V}^{*}_{\ell\omega_1}(\mathfrak{X},\mathfrak{so}(m),\ell)$, for a suitable choice of line bundle $\Pcal(\ell)\to \mathcal{M}^{par}_{\Spin(m)}$.
Theorem \ref{twistedconformal}  raises the question of whether  $\mathcal{M}_{\Spin(m)}$ and $\mathcal{M}^{-}_{\Spin(m)}$ are related by a Hecke type elementary transformation.

Recall that an \emph{oriented orthogonal bundle} on $C$ is a pair $(E,q)$,
where $E\to C$ is a vector bundle with trivial determinant and a nondegenerate quadratic form $q: E\otimes E\to \Ocal_C$ with obvious compatibility of trivialization with $\det E$ and $\det q$. 
In \cite{Abe:13}, T. Abe defines a transformation yielding a new orthogonal bundle  $E^\iota$ from an orthogonal bundle $E$ equipped with an isotropic line in the fiber $E_p$. Below we observe that the bundles $E^\iota$ and $E$ have opposite Stiefel-Whitney classes, meaning that the $\iota$-transform switches components of $\mathcal{M}_{\SO(m)}$. We then extend the $\iota$-transform to a Hecke type elementary transformation on Clifford bundles (see \eqref{eqn:hecke}). This enables us to give an alternative proof of Theorem \ref{twistedconformal}. 
The advantage of this identification will be seen in Theorem \ref{projflat} below.
The details of this construction are contained in Section \ref{sec:parabolic}.

\subsection{Hitchin connection} 
The locally free sheaf of conformal blocks associated to a family of smooth projective curves $\pi: \mathcal{C}\rightarrow B$ and any simple Lie algebra $\frg$ carries a flat projective connection known as the TUY connection (or the KZ connection in genus zero). The identification in Theorem \ref{twistedconformal} motivates a geometric description of this connection for twisted spin bundles.
Indeed, Hitchin \cite{Hitchin:90}  introduced a flat projective connection on  spaces of generalized theta functions  as the underlying curve $C$ varies over the Teichm\"uller  space  of  Riemann surfaces (see also \cite{Axelrod:91,VanGeemenDeJong:98,Andersen:12}). 
In  \cite{Laszlo}, Laszlo showed that with this  identification, and over the pointed Teichm\"uller space $\Tcal_{g,1}$,  the Hitchin connection coincides with the TUY connection on the space of conformal blocks. This statement also generalizes to the case of twisted spin bundles. More precisely, we prove the following.
\begin{theorem} \label{hitchinlaszlo}
As the pointed curve $(C,p)$ varies in $\Tcal_{g,1}$, the vector bundle with fiber\break  $H^0(\mathcal{M}^{-}_{\Spin(m)},\mathcal{P}^{\otimes\ell})$ is endowed with a flat projective connection which we also call the Hitchin connection. Under the identification of $H^0(\mathcal{M}^{-}_{\Spin(m)},\mathcal{P}^{\otimes\ell})$ with $\mathcal{V}^{*}_{\ell\omega_1}(\mathfrak{X},\frg,\ell)$, the Hitchin connection coincides with the TUY connection. 
\end{theorem}

Let $M^{-, reg}_{\Spin(m)}$ denote the moduli space of regularly stable twisted spin bundles (see Section \ref{sec:uniformization}). Then the Pfaffian line bundle descends to $M^{-, reg}_{\Spin(m)}$, and 
\begin{equation} \label{eqn:stack-moduli}
H^0(M^{-, reg}_{\Spin(m)},\mathcal{P}^{\otimes\ell})\simeq H^0(\mathcal{M}^{-}_{\Spin(m)},\mathcal{P}^{\otimes\ell})\ .
\end{equation}We refer the reader to Proposition \ref{prop:stackisspace} for more details. 
Now the essential strategy in the  proof of Theorem \ref{hitchinlaszlo} is the same as in \cite{Hitchin:90}, but there are two key differences. 
These are as follows:
\begin{itemize}
\item The connectivity of the fibers of the Hitchin map from the moduli space $M^\theta_G$ of $G$-Higgs bundles  to the Hitchin base
 is an essential ingredient in Hitchin's proof. 
In the untwisted case,
the connectivity follows, for example,  from a 
description of the fibers  in terms of  spectral data.
It seems not to be known if the fiber of the Hitchin map for twisted Higgs bundles is connected in general.
 We circumvent this issue by reducing to the  $\SO(m)$ moduli space, and then using results of Donagi-Pantev \cite{DonagiPantev:12}. 
\item 
The condition $H^1(M^{-, reg}_{\Spin(m)},\mathcal{P}^{\otimes\ell})=\{0\}$, 
is sufficient to show that the symbol map of the projective heat operator is injective. In the untwisted case,
one can again use Higgs bundles to establish this vanishing
 \cite{Hitchin:90, Laszlo}. For the same reason as above, this argument is unavailable in the twisted case. However, 
 Kumar-Narasimhan proved 
  such vanishing results  directly  without using Higgs bundles. In the present paper, we generalize the proof in \cite{KumarNarasimhan:97} to the twisted setting. 
\end{itemize}
The proof of the second statement in Theorem \ref{hitchinlaszlo} is analogous to that in \cite{Laszlo}. We refer the reader to Section \ref{twistedHiggs} for more details.


\subsection{Level one sections}
Since the strange duality map arises by pulling back a level one section, 
we study these sections   in detail. 
Let 
\begin{equation} \label{eqn:theta}
\Th(C):=\{ \kappa\in \Pic_{g-1}(C)\mid \kappa^{\otimes 2}=\omega_C\}
\end{equation}
denote the set of \emph{theta characteristics} of $C$.
Furthermore, denote by $\Th^+(C)\subset \Th(C)$ (resp.\ $\Th^-(C)\subset \Th(C)$)  the set of \emph{even} (resp.\ \emph{odd})
theta characteristics, i.e.\ those for which $h^0(C,\kappa)$ is even (resp.\ odd). 
We shall prove the following analog  of a  theorem of Belkale  \cite{Belkale:12} and Pauly-Ramanan \cite{PaulyRamanan}.
\begin{theorem}\label{basis} Let $s_{\kappa}$ denote the canonical $($up to scale$)$ section of the Pfaffian line bundle for a theta characteristic $\kappa$ $($see Definition \ref{def:pfaffian-section}$)$. Then the collection $\{s_{\kappa}\mid \kappa\in \Th^-(C)\}$ forms a basis of the space of the level one generalized theta functions $H^0(\mathcal{M}^{-}_{\Spin(2r+1)},\mathcal{P})$. 
\end{theorem} 

\begin{remark}
In \cite{AndersenMasbaum}, using TQFT methods, Andersen-Masbaum give a  ``brick decomposition'' of the  $\SL(m)$-conformal block bundles under the action of the Heisenberg group. The invariant Pfaffian sections and the decomposition of $H^0(\mathcal{M}_{\Spin(2r+1)}^{-},\mathcal{P})$ (as well as   $H^0(\mathcal{M}_{\Spin(2r+1)},\mathcal{P})$) into Pfaffian sections should be considered as an analog of brick decompositions for these spaces.
\end{remark}

By passing to a local \'etale cover, we can assume the torsor of theta characteristics is trivialized on $\mathcal{C}\rightarrow B$. 
We show the following.
\begin{theorem}\label{projflat}
For each $\kappa\in\Th^-(C)$, the Pfaffian section $s_{\kappa}\in H^0(\Mcal^-_{\Spin(2r+1)})$ is projectively flat with respect to the Hitchin/TUY connection of Theorem \ref{hitchinlaszlo}.
\end{theorem}
\noindent In the untwisted case this result appears in \cite{Belkale:12}. The proof of Theorem \ref{projflat} uses the fact that the projective heat operator is invariant under the action of the group of two torsion points of the the Jacobian. Once the existence of the Hitchin connection is established, the rest of the proof is same as that in \cite{Belkale:12}. 

\subsection{Rank-level duality for genus zero} For $r,s\geq 2$, let $d=2rs+r+s$ (this notation will be used throughout the paper). 
The embedding 
\begin{equation} \label{eqn:so-embedding}
\sofrak(2r+1)\oplus \sofrak(2s+1)\lra \sofrak(2d+1)
\end{equation}
extends to an embedding of affine Lie algebras. 
For integrable weights $\vec\lambda$, $\vec\mu$, and $\vec\Lambda$ of $\widehat \sofrak(2r+1)$ at level $2s+1$, 
$\widehat \sofrak(2s+1)$ at level $2r+1$, and $\widehat \sofrak(2d+1)$ at level $1$, respectively, 
suppose that the pair $(\vec\lambda, \vec\mu)$ appears in the affine branching of $\vec\Lambda$.
This in turn gives rise to maps on dual conformal blocks
$$
\mathcal{V}_{\vec{\lambda}}(\mathfrak{X},\mathfrak{so}(2r+1),2s+1))\rightarrow \mathcal{V}^{*}_{\vec{\mu}}(\mathfrak{X},\mathfrak{so}(2s+1),2r+1)\otimes \mathcal{V}_{\vec{\Lambda}}(\mathfrak{X},\mathfrak{so}(2d+1),1)\ .
$$
We note that in case $\vec\Lambda=(\omega_{\varepsilon_1}, \ldots, \omega_{\varepsilon_{n-2}}, \omega_d,\omega_d)$, with $\varepsilon_i\in\{0,1\}$, then 
$$\dim_\CBbb \mathcal{V}_{\vec{\Lambda}}(\mathfrak{X},\mathfrak{so}(2d+1),1)=1\ , $$
and we have a rank-level duality map, 
\begin{equation}\label{eqn:injectivity}
\mathcal{V}_{\vec{\lambda}}(\mathfrak{X},\mathfrak{so}(2r+1),2s+1)\lra \mathcal{V}^{*}_{\vec{\mu}}(\mathfrak{X},\mathfrak{so}(2s+1),2r+1)\ ,
\end{equation}
which is well-defined up to a nonzero multiplicative constant.   Recall that  $\widehat\sofrak(2r+1)$ has a \emph{diagram automorphism}  $\sigma$ which interchanges the nodes of the extended Dynkin diagram associated to the weights $\omega_0$ and $\omega_1$ (cf.\ \eqref{daiso}).
In Section \ref{ranklevelinjectivity}, we prove the following.
\begin{theorem}\label{mainjune16}
Let $C=\PBbb^1$.
Let $\vec\lambda=(\lambda_1,\ldots, \lambda_{n-2}, \lambda_{n-1}, \lambda_{n})$, where $\lambda_i$ is the highest weight of a representation of the group $\SO(2r+1)$ for $i\leq n-3$,  $\lambda_{n-1}, \lambda_{n}$ are spin representations that are not fixed by the diagram automorphism $\sigma$, and  $\vec\mu, \vec\Lambda$ are as above. Then
the rank-level duality map defined in \eqref{eqn:injectivity} is injective.
\end{theorem}
\noindent
This answers a question of Nakanishi and Tsuchiya (cf.\  \cite[Sec.\ 6]{NakanishiTsuchiya:92}).
It is important to note that the dimensions of the left and right hand sides of \eqref{eqn:injectivity} are not  equal in general: 
some explicit examples are given in Section \ref{failure} below.
This fact is in stark contrast with the case of $\slfrak(m)$ conformal blocks and demonstrates the subtlety of rank-level duality.

\begin{remark}
If $\lambda  \in P_{2s+1}(\mathfrak{so}(2r+1))$, $ \mu \in P_{2r+1}(\mathfrak{so}(2s+1))$, are such that $\sigma(\lambda)\neq \lambda$ and $(\lambda, \mu)$ appears in the branching of $\omega_d$, then $\sigma (\mu)=\mu$. 
\end{remark}

Let $X_n=\{(z_1,\dots,z_n)\mid z_i\in \PBbb^1\ ,\  z_i\neq z_j\}$ denote the configuration space of points on $\PBbb^1$, and let $P_n=\pi_1(X_n)$. The conformal blocks form a vector bundle over $X_n$ with a flat connection $\nabla_{KZ}$, and one can define the rank-level duality map as a map of vector bundles over $X_n$. Moreover, the rank-level duality map commutes with $\nabla_{KZ}$. 

As a corollary of Theorem \ref{mainjune16}, by factorizing two Spin weights at a time, we also obtain a result asserted in \cite{NakanishiTsuchiya:92}.
\begin{corollary} Let $C=\PBbb^1$ and $m$ be a positive integer.
	The  representations of the pure braid group $P_n$ associated to the conformal block bundles $\mathcal{V}_{\vec{\lambda}}(\mathfrak{X},\mathfrak{so}(2r+1),2s+1)$ with spin weights are reducible in general. More precisely, this occurs if $\vec{\lambda}$ is of the form $(\lambda_1,\ldots, \lambda_{n-2m},\mu_{1},\dots,\mu_{2m})$, where $\lambda_1,\ldots,\lambda_{n-2}$ are $\SO$-weights and $\mu_{1},\dots, \mu_{2m}$ are weights of $\operatorname{Spin}(2r+1)$ that are fixed by the Dynkin automorphism $\sigma$. 
\end{corollary}

\subsection{Strange duality maps in higher genus} Let  $\mathcal{M}_{2r+1}$ 
be as in \eqref{eqn:mr}.  The equality of dimensions in Corollary \ref{cor:spin-dimension} suggests the possibility of a strange duality isomorphism.  To make this precise, note that 
we  have  the following map:
\begin{equation} \label{eqn:SD}
SD:
 H^0(\mathcal{M}_{2r+1},\mathcal{P}^{\otimes(2s+1)})^* \longrightarrow H^0(\mathcal{M}_{2s+1},\mathcal{P}^{\otimes(2r+1)})\otimes H^0(\mathcal{M}_{2d+1},\mathcal{P})^*\ .
\end{equation} 
Since $\dim_{\mathbb{C}}H^0(\mathcal{M}_{2r+1},\mathcal{P})=2^{2g}$,  and we know that the Pfaffian sections $\{s_{\kappa}\mid \kappa\in \Th(C)\}$ form a basis (Theorem \ref{basis}, \cite{Belkale:12}, \cite{PaulyRamanan}), it is natural to consider $s_\Delta=\sum_{\kappa} s_{\kappa}$, and investigate whether the induced {\em strange duality map} is an isomorphism. Denote this map by
\begin{equation}\label{eqn:SD1}
s_\Delta^\ast:H^0(\mathcal{M}_{2r+1},\mathcal{P}^{\otimes(2s+1)})^*\longrightarrow  H^0(\mathcal{M}_{2s+1},\mathcal{P}^{\otimes(2r+1)})\ .
\end{equation} 
It is easy to arrange that the map \eqref{eqn:SD} be equivariant with respect to  the action of $J_2(C)$ permuting the theta characteristics. 
By taking invariants, 
for every  $\kappa\in \Th(C)$ we get a  map induced by the Pfaffian section $s_{\kappa}$:
\begin{equation} \label{eqn:SK}
s^*_{\kappa}:H^0(\mathcal{M}_{\SO(2r+1)},\mathcal{P}_{\kappa}^{\otimes(2s+1)})^*\longrightarrow H^0(\mathcal{M}_{\SO(2s+1)},\mathcal{P}_{\kappa}^{\otimes(2r+1)})\ .
\end{equation}
A simple argument shows that  $s^*_\kappa$  is an isomorphism for every $\kappa$ if and only if the map $s_\Delta^\ast$  is an isomorphism. We refer the reader to Section \ref{strangedualityj2C} for more details. 
However, the fact, mentioned above, that the rank-level duality map for spin weights fails to be an isomorphism may be taken as an indication that the strange duality map \eqref{eqn:SD1} might not be an isomorphism either. We shall prove  that this is indeed the case. 
\begin{theorem} \label{thm:SD-false}
The strange duality map \eqref{eqn:SD1} $($resp.\ \eqref{eqn:SK}$)$ is not an isomorphism $($resp.\ is not an isomorphism for every $\kappa$$)$. 
\end{theorem}

The analysis passes through the sewing construction and
detailed calculations involving the rank-level maps discussed above. Since the Pfaffian sections are projectively flat, there is a  consequence  for the holonomy representations of spin mapping class groups.

\begin{corollary} \label{thm:MCG}
For some theta characteristic $\kappa$ and any $r,s\geq 2$, 
the Hitchin connection in Theorem \ref{hitchinlaszlo} has reducible holonomy representation.
%
\end{corollary}

\begin{remark}
The holonomy representations of the Hitchin connection for $\mathcal{M}_{\Spin(2r+1)}$ and $\mathcal{M}_{\Spin(2r+1)}^{-}$ are easily seen to be reducible by noting the difference of dimensions of the Verlinde spaces for $\Spin(2r+1)$ and $\Spin(2s+1)$
(cf.\ \eqref{eqn:spin-dim}).
  However, for the $\SO$ moduli spaces and powers of the Pfaffian line bundle there is no known Verlinde type formula. Hence, simple arguments based on dimension do not work. Questions about irreducibility of mapping class group representations for $\SL(n)$ have been considered in \cite{AndersenFeltsjad}.
\end{remark}


\subsection{Acknowledgments} The authors are grateful to P. Belkale, I. Biswas, P. Brosnan and T. Pantev for useful discussions and suggestions. Additional thanks to  J. Andersen,   S. Bradlow, J. Martens, and L. Schaposnik for their valuable input on aspects of this work. The referee made useful suggestions 
for improvements to the exposition and is gratefully acknowledged.


\section{Conformal blocks and basic properties}
Here we recall some definitions from ~\cite{TUY:89}. Let $\frg$ be a simple complex Lie algebra  with Cartan subalgebra $\mathfrak{h}$. Let $\Delta=\Delta_{+}\sqcup\Delta_{-}$ be a positive/negative decomposition of the set of roots, and
$\frg=\mathfrak{h} \oplus \sum_{\alpha \in \Delta}\frg_{\alpha},$
the decomposition into root spaces $\gfrak_\alpha$.
 Let $(\, ,\, )$ denote the Cartan-Killing form on $\frg$, normalized so that $(\theta, \theta)=2$ for a longest root  $\theta $. We often identify $\mathfrak{h}$ with $\mathfrak{h}^*$ using  $(\, ,\, )$. 

\subsection{Affine Lie algebras} \label{sec:affine}
 The \emph{affine Lie algebra} $\widehat{\frg}$ is defined as a central extension of the loop algebra $\frg\otimes \mathbb{C}\llrrparen{\xi}$. As a vector space 
$\widehat{\frg}:=   \frg\otimes \mathbb{C}\llrrparen{\xi} \oplus \mathbb{C}\cdot c,$ where $c$ is central, and the Lie bracket is determined by
$$[X\otimes f(\xi), Y\otimes g(\xi)]=[X,Y]\otimes f(\xi)g(\xi) + (X,Y)\Res_{\xi=0}(gdf)\cdot c\ ,$$ where $X,Y \in \frg$ and $f(\xi),g(\xi) \in \mathbb{C}\llrrparen{\xi}$. 
Set $X(n)=X\otimes \xi^n$ and $X=X(0)$ for any $X \in \frg $ and $n \in \mathbb{Z}$. 

 The theory of highest weight integrable irreducible modules for $\widehat{\frg}$ runs parallel to that of finite dimensional irreducible modules for $\frg$. Let us briefly recall the details for completeness. The finite dimensional irreducible $\gfrak$-modules  are parametrized by the set of dominant integral weights $P_{+}(\gfrak) \subset \mathfrak{h}^*$. For each $\lambda \in P_{+}(\frg)$, let $V_{\lambda}$ denote the irreducible $\gfrak$-module with highest weight $\lambda$. 
 Fix a positive integer $\ell$,  called the \emph{level}. The set of \emph{dominant integral weights of level $\ell$} is defined by:
\begin{equation*} \label{eqn:integral-weights}
P_{\ell}(\frg):=  \{ \lambda \in P_{+}(\gfrak) \mid (\lambda, \theta) \leq \ell\}\ .
\end{equation*}
For each $\lambda \in P_{\ell}(\frg)$, there is a unique irreducible {\em integrable  highest weight $\widehat{\frg}$-module} $\mathcal{H}_{\lambda}(\frg,\ell)$ which satisfies the following properties: 
\begin{enumerate}
\item  $\mathcal{H}_{\lambda}(\frg,\ell)$ is generated by $V_{\lambda}$ over $\widehat{\frg}$ (cf.\ \cite{KacWakimoto:88});
\item $\mathcal{H}_{\lambda}(\frg,\ell)$ are infinite dimensional;
\item $V_{\lambda} \subset \mathcal{H}_{\lambda}(\frg,\ell)$;
\item The central element $c$ of $\widehat{\frg}$ acts by the scalar $\ell$. 
\end{enumerate}
When there are implicitly understood, we sometimes omit the notation $\frg$ or $\ell$ from  $\mathcal{H}_{\lambda}(\frg,\ell)$. 

We will also need the following quantity.
 For any $\lambda \in P_{\ell}(\gfrak)$,  define the {\em trace anomaly} 
 \begin{equation} \label{eqn:trace-anomaly}
 \Delta_{\lambda}(\gfrak,\ell):=\frac{(\lambda, \lambda+2\rho)}{2(g^\vee+\ell)}\ ,
 \end{equation}
  where $\rho$ is the half sum of positive roots, and $g^\vee$ is the dual Coxeter number of $\gfrak$. 
  
 Let  $\widehat{\mathcal{L}}(\mathfrak{g})$ denote the untwisted affine Kac-Moody Lie algebra of $\mathfrak{g}$ 
(see  \cite[Sec.\ 7.2]{KacBook}). Explicitly, it can be defined as 
$\widehat{\mathcal{L}}\frg=\widehat{\frg}\oplus\mathbb{C}d,$ where $d$ is a derivation that commutes with $c$ and acts on $\mathfrak{g}\otimes\mathbb{C}((\xi))$ by the formula $\displaystyle d=\xi\frac{d}{d\xi}$. 
Clearly, $\widehat{\frg}$ is a Lie subalgebra of $\widehat{\mathcal{L}}\frg$. Let $\Lambda_0,\Lambda_1,\dots,\Lambda_m$ (resp.\  $\omega_1,\dots, \omega_m$) denote the affine fundamental weights of $\widehat{\mathcal{L}}\frg$ (resp.\ $\frg$), where $m=\rank\gfrak$.  We observe the following:
\begin{itemize}
	\item $\Lambda_i=\omega_i+a_i^{\vee}\Lambda_0$ for $1\leq i\leq m$, where $a_i^{\vee}$ are the dual Coxeter labels (\cite{KacBook}). 
	\item Any $\lambda \in P_{\ell}(\frg)$ corresponds to the weight $\lambda+ \ell \Lambda_0$ of $\widehat{\mathcal{L}}\frg$.
	\item Any highest weight integrable irreducible representation of $\widehat{\mathcal{L}}\frg$ is also irreducible as $\widehat{\frg}$-module.
	\item $\Lambda_0$ restricted to the Cartan subalgebra of $\widehat{\frg}$ is zero. 
\end{itemize} 
Since we are working with $\widehat{\frg}$ in this paper, we will denote by $\omega_i$ the fundamental weight for both $\frg$ and $\widehat{\mathcal{L}}\frg$. For uniformity of notation, we will denote by $\omega_0$-the zero-th fundamental weight $\Lambda_0$ of $\widehat{\mathcal{L}}\frg$. 
  
\subsection{Conformal embeddings}   \label{sec:conformal-embedding}
Let  $\phi : \mathfrak{s} \rightarrow \mathfrak{g}$ an embedding of simple Lie algebras, and let $(\, ,\, )_{\mathfrak{s}}$ and $(\, ,\, )_{\frg}$ be the  Cartan-Killing forms, normalized as above. 
Then the \emph{Dynkin index} of $\phi$ is the unique integer $d_{\phi}$ satisfying 
$(\phi(x), \phi(y))_{\frg}=d_{\phi}\cdot(x,y)_{\mathfrak{s}}$, for all $x, y \in \mathfrak{s}$. More generally, when $\mathfrak{s}=\mathfrak{g}_1\oplus \mathfrak{g}_2$, $\gfrak_i$ simple, we define the \emph{Dynkin multi-index of} $\phi=\phi_1\oplus \phi_2:\mathfrak{g}_1\oplus\mathfrak{g}_2 \rightarrow \mathfrak{g}$ to be $d_{\phi}=(d_{\phi_1}, d_{\phi_2})$.
\begin{definition}\label{def:conformal-embedding}
Let $\phi=(\phi_1,\phi_2): \mathfrak{s}=\mathfrak{g}_1\oplus \mathfrak{g}_2 \rightarrow \frg$ be an embedding of Lie algebras with Dynkin multi-index $d_\phi=(d_{\phi_1},d_{\phi_2})$. Then $\phi$ is said  to be a \emph{conformal embedding}  if 
\begin{equation*}\label{eqn:conformal-embedding}
\frac{d_{\phi_1}\dim{\frg_1}}{g_1^\vee+d_{\phi_1}}+\frac{d_{\phi_2}\dim{\frg_2}}{g_2^\vee+d_{\phi_2}}=\frac{\dim{\frg}}{g^\vee+1}\ ,\end{equation*}
where $g_1^\vee$, $g_2^\vee$, and $g^\vee$ are the dual Coxeter numbers of $\frg_1$, $\frg_2$, and $\frg$, respectively.
\end{definition}

Many familiar and important embeddings  are conformal:  \eqref{eqn:so-embedding} is one family of such examples. For a complete list, see  ~\cite{BaisBouwknegt:87}. 
For the purposes of this paper, the key property of conformal embeddings that we need is the following:
an embedding 
$\phi: \sfrak=\frg_1\oplus \frg_2\rightarrow  \frg$ is  conformal if and only if any irreducible $\widehat{\frg}$-module
$\mathcal{H}_{\Lambda}(\frg, 1)$, $\Lambda\in P_1(\gfrak)$, decomposes into a finite sum of irreducible $\widehat{\mathfrak{s}}$-modules
of the form $\Hcal_{\lambda_1}(\gfrak_1, \ell_1)\otimes \Hcal_{\lambda_2}(\gfrak_2, \ell_2)$, where $\lambda_i\in P_{\ell_i}(\gfrak_i)$, $i=1,2$, and $(\ell_1,\ell_2)=d_\phi$, the Dynkin multi-index.
 See ~\cite{KacWakimoto:88}.

\subsection{Conformal blocks} \label{sec:conformal-block}
 Let  $C$ be  a smooth projective curve with marked points $\vec{p}=(p_1,\ldots, p_n)$
such that $(C, \vec p)$ satisfies the Deligne-Mumford stability conditions. We furthermore assume a choice coordinates and
  formal neighborhoods around the $p_i$, 
which give isomorphisms $\widehat{\mathcal{O}}_{C,P_i}\isorightarrow \mathbb{C}\llrrbracket{\xi_i}$. 
We will use the notation  $\mathfrak{X}=(C;\vec{p})$ to denote this data. 
The \emph{current algebra} is defined to be $\frg(\mathfrak{X}):=  \frg\otimes H^0(C, \mathcal{O}_{C}(*(p_1,\ldots,p_n)))$. By local expansion of functions using the chosen coordinates $\xi_i$, we get an embedding:
$$\frg(\mathfrak{X}) \hookrightarrow \widehat{\frg}_{n}:=  \bigoplus_{i=1}^n\frg\otimes_{\mathbb{C}}\mathbb{C}\llrrparen{\xi_i} \oplus \mathbb{C}\cdot c\ .$$
Consider an $n$-tuple of weights $\vec{\lambda}=(\lambda_1,\ldots, \lambda_n) \in P_{\ell}^n(\frg)$, and set 
$$\mathcal{H}_{\vec{\lambda}}(\gfrak,\ell)=\mathcal{H}_{\lambda_1}(\frg,\ell)\otimes \cdots \otimes \mathcal{H}_{\lambda_n}(\frg,\ell)\ .$$
 The algebra $\widehat{\frg}_n$ (and hence also the current algebra $\frg(\mathfrak{X})$) acts on $\mathcal{H}_{\vec{\lambda}}(\gfrak,\ell)$ componentwise using the embedding above. 
\begin{definition}\label{definitionconformalblocks}
The space of \emph{conformal blocks} is
$$\mathcal{V}^{*}_{\vec{\lambda}}(\mathfrak{X}, \frg,\ell):=  \operatorname{Hom}_{\mathbb{C}}(\mathcal{H}_{\vec{\lambda}}(\gfrak,\ell)/\frg(\mathfrak{X})\mathcal{H}_{\vec{\lambda}}(\gfrak,\ell), \mathbb{C})\ .$$ 
The space of \emph{dual conformal blocks} is 
$\mathcal{V}_{\vec{\lambda}}(\mathfrak{X}, \frg,\ell)=\mathcal{H}_{\vec{\lambda}}(\gfrak,\ell)/\frg(\mathfrak{X})\mathcal{H}_{\vec{\lambda}}(\gfrak,\ell)\ .$
\end{definition}

Conformal blocks are finite dimensional vector spaces, and their dimensions  are given by the \emph{Verlinde formula} \cite{Faltings:94, Teleman, TUY:89}. We now discuss some important properties of the spaces of conformal blocks. 
\begin{itemize}\label{propertiesofconformalblocks}

\item ({\sc Flat projective connection}) Consider a family 
$$\mathcal{F}=(\pi:\mathcal{C}\rightarrow B; \sigma_1,\ldots, \sigma_n; \xi_1,\ldots,\xi_n)$$ of nodal curves on a base $B$ with sections $\sigma_i$ and formal coordinates $\xi_i$. In ~\cite{TUY:89}, a locally free sheaf $\mathcal{V}^{*}_{\vec{\lambda}}(\mathcal{F},\frg,\ell)$ known as the sheaf of conformal blocks is constructed over the base $B$. 
Moreover, if $\mathcal{F}$ is a family of smooth projective curves, then the sheaf $\mathcal{V}^{*}_{\vec{\lambda}}(\mathcal{F},\frg,\ell)$ carries a  flat projective connection known as the \emph{TUY connection}. We refer the reader to ~\cite{TUY:89} for more details. In genus zero, the TUY connection is a flat connection and is also known as {\em KZ connection}.

\item ({\sc Propagation of vacua}) Let $C$ be any curve with $n$-marked points satisfying the Deligne-Mumford stability conditions and $C$ be the same curve with $n+1$ marked points. Assume that the weights attached to the $n$ marked points are $\vec{\lambda}=(\lambda_1,\dots,\lambda_n)$ and we associate the vacuum representation ($\mathcal{H}_{\omega_0}$) at the $(n+1)$-st point. Then there is a canonical isomorphism 
$\mathcal{V}_{\vec{\lambda}}(\mathfrak{X}, \frg,\ell)\simeq \mathcal{V}_{\vec{\lambda}\cup \omega_0}(\mathfrak{X}',\frg,\ell)$, where $\mathfrak{X}$ (resp.\  $\mathfrak{X}'$) denote the data associated to the $n$ (resp.\  $n+1$) pointed curve $C$.

\item ({\sc Gauge symmetry}) Let $f \in H^0(C, \mathcal{O}_C(*(p_1,\dots,p_n)))$ and $\langle \Psi \mid  \in \mathcal{V}_{\vec{\lambda}}(\mathfrak{X},\frg,\ell)$, then $\langle \Psi\mid (X\otimes f)=0$.  More precisely, for any $\mid \phi_1 \otimes \dots \otimes \phi_n\rangle \in \mathcal{H}_{\vec{\lambda}}(\frg,\ell)$, 
$$\sum_{i=1}^n \langle\Psi \mid \phi_1\otimes \dots \otimes (X\otimes f(\xi_i))\phi_i\otimes \dots \otimes \phi_n\rangle=0\ .$$
\end{itemize}

Let $\mathcal{X}\to \operatorname{Spec}\mathbb{C}\llrrbracket{t}$ be a family of curves of genus $g$ with $n$ marked points with chosen coordinates such that the special fiber $\mathcal{X}_0$ is a curve $X_0$ over $\mathbb{C}$ with exactly one node,  and the generic fiber $\mathcal{X}_t$ is a smooth curve. Let $\widetilde{X}_0$ be the normalization of $X_0$. For $\lambda \in P_{\ell}(\frg)$, the following isomorphism is constructed in \cite{TUY:89}:

$$\oplus \iota_{\lambda}: \bigoplus_{\lambda \in P_{\ell}(\frg)} \mathcal{V}_{\vec{\lambda},\lambda,\lambda^{\dagger}}^{\ast}(\widetilde{\mathfrak{X}}_0, \frg,\ell)\rightarrow \mathcal{V}_{\vec{\lambda}}^{\ast}(\mathcal{X}_0, \frg,\ell)\ ,$$
 where $\widetilde{\mathfrak{X}}_0$ is the data associated to the $(n+2)$ points of the smooth pointed curve $\widetilde{X}_0$ with chosen coordinates and $\lambda^{\dagger}$ is the highest weight of the contragredient representation of $V_{\lambda}$. This is commonly referred to as {\em factorization of conformal blocks}. 

In the same paper \cite{TUY:89}, a sheaf theoretic version of the above isomorphism was also constructed which is commonly referred to as the {\em sewing construction}. This provides for each $\lambda \in P_{\ell}(\frg)$, a map of $\mathbb{C}\llrrbracket{t}$-modules: 
$s_{\lambda}(t) : \mathcal{V}_{ \vec{\lambda},\lambda,\lambda^{\dagger}}^{*}(\widetilde{\mathfrak{X}}_0, \frg,\ell)\otimes \mathbb{C}\llrrbracket{t}\rightarrow \mathcal{V}_{\vec{\lambda}}^{*}(\mathcal{X}, \frg,\ell)$. 
Then $s_{\lambda}(t)$ extends the map $\iota_{\lambda}$ in families such that $\oplus_{\lambda \in P_{\ell}(\frg)} s_{\lambda}(t)$, is an isomorphism of locally free sheaves over $\operatorname{Spec}\mathbb{C}[[t]]$. We refer the reader to \cite{Mukhopadhyay:12, TUY:89} for exact details. 

\section{Twisted moduli stacks}\label{twistedmoduli}

\subsection{Uniformization} \label{sec:uniformization}
In this section we recall the construction of the twisted moduli stacks for 
spin bundles as in \cite{BLS:98} (see also  \cite{Oxbury:99, OxburyWilson:96}).
First, let us fix some notation.

\begin{definition} \label{def:stacks}
Let $G$ be a connected complex reductive Lie group.  Then
\begin{enumerate}
\item  $\Mcal_G:=  $ the moduli stack of $G$-bundles on $C$;
 \item $M_G:=  $ the Ramanathan coarse moduli space of $S$-equivalence classes of semistable $G$-bundles on $C$;
 \item a $G$-bundle is \emph{regularly stable} if it is stable and its automorphism group is equal to  the center $Z(G)$. We denote by $M_G^{reg}\subset M_G$ the moduli space of regularly stable bundles.
 \end{enumerate}
\end{definition}

 Recall the exact sequence
$
1\to \ZBbb/2\to \Spin(m) \to \SO(m)\to 1
$.
Identify $\ZBbb/2$ with the subgroup  $\{\pm 1\}\subset \CBbb^\times$,
and define the \emph{special Clifford group}
\begin{equation} \label{eqn:spinc}
\SC(m) :=   \Spin(m)\times_{\ZBbb/2} \CBbb^\times\ .
\end{equation}
The \emph{spinor norm} is the group homomorphism
$\Nm : \SC(m)\lra \CBbb^\times\ ,$ 
which induces a morphism of stacks $\mathcal{M}_{\SC(m)}\rightarrow \mathcal{M}_{\CBbb^\times}$. We will denote this stack morphism also by $\Nm$.

 Let $p$ be a fixed point of the curve $C$. Throughout the paper we will denote the punctured curve by $C^\ast=C-\{p\}$. Consider bundles $\mathcal{O}_C(dp)$, where $d\in \ZBbb$. 
 Then the preimage by $\Nm$ of the class of
 $[\Ocal_C(dp)]\in \mathcal{M}_{\CBbb^\times} $ depends only on the parity of $d$ (cf.\ \cite[Prop.\ 3.4]{Oxbury:99}).
 We will denote by $\Mcal^\pm_{\SC(m)}$ the inverse images of the Jacobian $J(X)$ and $\Pic_1(C)$, respectively.
 Let
 $\Mcal^\pm_{\Spin(m)}$ be  the inverse images of the points $\mathcal{O}_C(dp)$, for $d=0,1$, respectively. 
Therefore, while by definition $\Mcal^+_{\Spin(m)}=\Mcal_{\Spin(m)}$, the space $\Mcal^-_{\Spin(m)}$ is a ``twisted'' component that does not correspond to a stack of $G$-bundles for any complex reductive $G$.

The components  $\Mcal^\pm_{\SO(m)}$ of $\Mcal_{\SO(m)}$ are labeled by $\delta\in \pi_1(\SO(m))\simeq \ZBbb/2$ (cf.\ \cite[Prop.\ 1.3]{BLS:98}).
The map $\SC(m) \rightarrow \SO(m)$, coming from the projection of \eqref{eqn:spinc} on the first factor, induces a morphism of stacks 
\begin{equation} \label{eqn:projection}
p:\Mcal^\pm_{\Spin(m)} \lra \Mcal^\pm_{\SO(m)}\ .
\end{equation}


\begin{definition} \label{def:loop-group}
For $G$ as in Definition \ref{def:stacks}, let
\begin{enumerate}
\item $LG:=  G\llrrparen{\xi}$ be the algebraic loop group of $G$;
\item $L^+G:=  G\llrrbracket{\xi}$ be the group of positive loops;
\item $\Qcal_G :=  LG/L^+G$ be the \emph{affine Grassmannian};
\item $L_CG:=  G(\Ocal_{C^\ast})\hookrightarrow LG$.
\end{enumerate}
\end{definition}

The following result, proved in \cite{BLS:98},  gives a uniformization for the twisted moduli stacks and determines their Picard groups. We only state it in the case  $G=\Spin(m)$.

\begin{proposition}\label{uniformlyBLS}
Let $\delta \in\{\pm 1\}=\pi_1(\SO(m))$ and $ \zeta \in (L\SO(m))^{\delta}(\mathbb{C})$. Then 
$$\Mcal^\delta_{\Spin(m)}=(\zeta^{-1}\cdot L_C(\Spin(m))\cdot\zeta)\backslash\Qcal_{\Spin(m)}\ ,$$  where $\Qcal_{\Spin(m)}$ is the affine Grassmannian of  $\Spin(m)$. The torsion subgroup of $\Pic(\Mcal^\pm_{\Spin(m)})$ is trivial, and in fact,  
$\Pic(\Mcal^\pm_{\Spin(m)})\simeq\ZBbb$.
\end{proposition}

As we have done with stacks, we may also define the  coarse moduli spaces $M_{\SC(m)}^-$ and $M_{\Spin(m)}^-$ of semistable twisted bundles on $C$.

\subsection{Pfaffian divisors} \label{sec:divisor} 
The set $\Th(C)$ of theta characteristics forms a torsor for the 2-torsion points $J_2(C)$ of the Jacobian of $C$. Note the  cardinalities \cite[Sec.\ 4]{mumford}: $|J_2(C)|=|\Th(C)|=2^{2g}$, $|\Th^\pm(C)|=2^{g-1}(2^g\pm 1)$. 
Recall from the introduction that by an oriented orthogonal bundle on $C$  we mean a pair $(E, q)$ consisting of a bundle $E\to C$ with trivial determinant, and a nondegenerate quadratic form $q:E\otimes E\to \Ocal_C$. Then $q$ induces a nondegenerate quadratic form $\hat q: (E\otimes \kappa)\otimes (E\otimes \kappa)\to \omega_C$.
We recall the following from \cite{LaszloSorger:97}.
\begin{proposition}\label{LSpfaff}
Let $B$ be a locally noetherian scheme, $\pi: C\times B\to B$, $\pr: C\times B\to C$, the projections,
 and $(\Ecal,\hat q)\to C\times B$  a vector bundle equipped with an $\omega_C$-valued nondegenerate quadratic form $\hat q$. Then the choice of a theta characteristic $\kappa\to C$ gives a canonical square root $\mathcal{P}_{\mathcal{E},\hat q, \kappa}$ of the determinant of cohomology $\mathcal{D}_{\Ecal}=\left[{\rm Det}R\pi_\ast (\Ecal\otimes \pr^\ast\kappa)\right]^\ast$. Moreover, if $f:B'\rightarrow B$ is a morphism of locally noetherian schemes, then the Pfaffian functor commutes with base change, i.e.\
$f^*{\mathcal{P}_{\mathcal{E},\hat q}}=\mathcal{P}_{f^*{\mathcal{E}},f^*{\hat q}}$.
\end{proposition}


Next,
we recall the definition of the \emph{Pfaffian divisor}, following \cite{BLS:98, LaszloSorger:97}. Let  $(\mathcal{E}, q)\to C\times \MSO$ be  the universal quadratic bundle.  For $\kappa\in \Th(C)$, consider the substack defined by:
$\Theta_{\kappa}:=  \Div(R\pi_\ast(\mathcal{E}\otimes pr^*{\kappa}))$.
 It is shown in \cite[eq.\ (7.10)]{LaszloSorger:97} that $\Theta_{\kappa}$ is a divisor on $\MSO^+$ if and only if either $m$ or $\kappa$ is even. We postpone the proof of the following proposition to Section \ref{sec:det-pfaffian}.
\begin{proposition} \label{prop:pfaffian-divisor}
The substack $\Theta_{\kappa}$ is a divisor on $\MSO^{-}$ if and only if both $m$ and $\kappa$ are odd.
\end{proposition}

\begin{definition} \label{def:pfaffian-section}
It follows from the above that there is a nonzero section  $s_{\kappa}$ (canonical up to scale)  of $\mathcal{P}_{\kappa}\to\Mcal_{\SO(2r+1)}$, supported on $\Mcal_{\SO(2r+1)}^+$ (resp.\ $\Mcal_{\SO(2r+1)}^-$) if $\kappa$ is even (resp.\ odd). We call  $s_{\kappa}$ the 
 \emph{Pfaffian section}.
 \end{definition} 

Recall the projection \eqref{eqn:projection}.
For  $\kappa,\kappa'\in\Th(C)$, the line bundles $p^*\mathcal{P}_{\kappa}$, $p^{*}\mathcal{P}_{\kappa'}$ are isomorphic. 
We therefore set $\mathcal{P}=p^*\mathcal{P}_{\kappa}$, which is well-defined up to this isomorphism. On each component $\mathcal{M}^\pm_{\Spin(m)}$,  $\mathcal{P}$ is the ample generator of  $\Pic(\Mcal^\pm_{\Spin(m)})$ \cite{BLS:98}. 

Let $A$ be the group of principal $\ZBbb/2$-bundles on $C$, where $\ZBbb/2$ is identified with the kernel of the map $\Spin(m)\rightarrow \SO(m)$. Then $A\simeq J_2(X)$. Let $\widehat{A}$ denote the set of characters of $A$. Let $Y=M_{\SO(m)}^{-,reg}$ (the notion of regularly stable extends directly to the twisted setting), and $\widetilde{Y}=p^{-1}(Y)$. Here $p: {M}^{-}_{\Spin(m)}\rightarrow {M}^{-}_{\SO(m)}$ is the projection map. 
By \cite[Prop.\ 13.5]{BLS:98}, the Galois covering $p$ is \'etale over $M_{\SO(m)}^{-,reg}$.   Since $M_{\SO(m)}^{-}-Y$ has codimension $\geq 2$ (see \cite[Thm.\  II.6]{Faltings:93}, \cite[Appendix]{Laszlo} and recall $m\geq 5$), and $p$ is finite and dominant, we conclude that  $M_{\Spin(m)}^{-}-\widetilde{Y}$ has codimension  $\geq 2$ as well. The moduli spaces $M^{-}_{\SO(m)}$ and $M^{-}_{\Spin(m)}$ can be constructed as GIT quotients of a smooth scheme by a reductive group (\cite[Lemma 7.3]{BLS:98}), hence they are normal. Therefore, by normality of the moduli spaces $M^{-}_{\SO(m)}$ and $M^{-}_{\Spin(m)}$, we get
$H^0(Y,\mathcal{O}_{Y})=H^0(\widetilde{Y},\mathcal{O}_{\widetilde{Y}})=\mathbb{C}$.
There is a decomposition of sheaves $p_{*}\mathcal{O}_{\widetilde{Y}}=\oplus_{\chi \in \widehat{A}} L_{\chi},$ where as a presheaf  
$L_{\chi}(U)=\{s \in \Ocal_{\widetilde Y}(p^{-1}(U))\mid gs=\chi(g)s, \forall g \in A\}$.
\begin{proposition} We have the following properties: 
\begin{enumerate}
\item $H^0(Y,L_{\chi})=\begin{cases}  \CBbb & \chi=1\\ 0 &\chi\neq 1\end{cases}$;
\item  for any $\chi$, $p^*L_{\chi}=\Ocal_{\widetilde{Y}}$;
\item $L_{\chi}\otimes L_{\chi'}=L_{\chi \chi'}$;
\item $L_{\chi}\simeq L_{\chi'} \iff \chi=\chi'$. 
\end{enumerate}
\end{proposition}

It is well-known \cite{BLS:98} that $Y$ is smooth, and since the map $p:\widetilde{Y} \rightarrow Y$ is Galois and \'etale, this implies that $\widetilde{Y}$ is also smooth and is contained in $M^{-,reg}_{\Spin(m)}$. 
  We will also need the following fact.

\begin{lemma} \label{lem:simplyconnected}
 $\pi_1(\widetilde{Y})=\{1\}$.
\end{lemma}

\begin{proof} The proof is essentially the same as in Atiyah-Bott \cite[Thm.\ 9.12]{AtiyahBott:82}.
Let $K\subset \SC(m)$ be a maximal compact subgroup.
Fix a topologically nontrivial smooth principal $\SC(m)$-bundle $P\to C$, and let $P_K$ be a  reduction to $K$.  Let $\Acal(P_K)$ be the space of connections $P_K$. Then $\Acal(P_K)$ can be identified with the space of holomorphic structures on $P$, i.e.\ holomorphic principal $\SC(m)$-bundles. 
Let $\Gcal(P)$ denote the group of $\SC(m)$ gauge transformations, and $\overline\Gcal(P)$ the quotient of $\Gcal(P)$ by the constant central gauge transformations (recall that  $Z(\SC(m))=\CBbb^\times$ for $m$ odd, and $Z(\SC(m))=\CBbb^\times\times \ZBbb/2$ for $m$ even).  By a standard argument, 
$
\pi_0(\Gcal(P))\simeq 
H^1(C, \pi_1(\SC(m)))
$.
Since $\pi_1(\SC(m))=\ZBbb$, 
we conclude that
$
\pi_0(\Gcal(P))\simeq H^1(C, \ZBbb)
$.
From the  fibration 
$Z(\SC(m))\to \Gcal(P)\to \overline\Gcal(P)$,
 we find $
\pi_0(\overline\Gcal(P))\simeq H^1(C, \ZBbb)
$, as well.
From \cite[Sec.\ 10]{AtiyahBott:82}, the complement
of the  stable points  $\Acal^{s}(P_K)\subset\Acal(P_K)$ has
 complex codimension at least $2$. 
As noted by Faltings (see the comment in the proof of part (ii) of Theorem II.6 in \cite{Faltings:93}), the same proof applies to show that the set of regularly stable points $\Acal^{reg}(P_K)\subset\Acal(P_K)$ has
 complex codimension at least $2$. 
 Since $\Acal(P_K)$ is smooth and  contractible, this implies in particular  that $\Acal^{reg}(P_K)$ is simply connected. Note also that  $\overline\Gcal(P)$ acts freely on $\Acal^{reg}(P_K)$ with quotient $M^{-, reg}_{\SC(m)}$.
 It follows that
\begin{equation}\label{eqn:pi1}
\pi_1\bigl(M^{-, reg}_{\SC(m)}\bigr) 
= \pi_1(\Acal^{reg}(P_K)/\overline \Gcal(P)) \simeq\pi_0(\overline \Gcal(P))
\simeq H^1(C, \ZBbb)\ .
\end{equation}
Now consider the fibration:
\begin{equation} 
\begin{split} \label{eqn:fixed-norm}
\Scale[.9]{
\xymatrix{
M^{-, reg}_{\Spin(m)} \ar[r] &M^{-, reg}_{\SC(m)} \ar[d]^{\Nm} \\
& \Pic_1(C) 
}}
\end{split}
\end{equation}
By the associated exact sequence of fundamental groups,
$$
1\lra \pi_1\bigl(M^{-, reg}_{\Spin(m)}\bigr)\lra \pi_1\bigl(M^{-, reg}_{\SC(m)}\bigr)\lra \pi_1(\Pic_1(C))\lra 1\ ,
$$
and \eqref{eqn:pi1}, 
we see immediately that $\pi_1\bigl(M^{-,reg}_{\Spin(m)}\bigr)=\{1\}$.
Now both $\widetilde Y$ and $ M^{-, reg}_{\Spin(m)}$ are smooth with complement of codimension $\geq 2$. Therefore,
 $\pi_1(\widetilde Y)\simeq \pi_1\bigl(M^{-,reg}_{\Spin(m)}\bigr)=\{1\}$.
\end{proof}

\begin{proposition}
Given $\kappa\in \Th(C)$ and $\alpha\in J_2(C)$,  then 
$\Pcal_{\kappa \otimes \alpha}\otimes \Pcal_{\kappa}^{\otimes(-1)}$ is isomorphic to a unique $L_{\chi}$, where $\chi \in \widehat{A}$. 

\end{proposition}

\begin{proof}
By the proof of \cite[Prop.\ 5.2]{BLS:98}, there is an injective homomorphism $\lambda:\widehat{A}\rightarrow \Pic(\mathcal{M}_{\SO(m)}^{\delta})$, and $\mathcal{P}_{\kappa\otimes \alpha}\otimes \mathcal{P}_{\kappa}^{\otimes(-1)}$ equals $\lambda(W(\alpha))$,
where $W$ is the Weil pairing on $J_2(C)\otimes J_2(C) \rightarrow \mu_2=\{1,-1\}$.
Now if $\alpha \neq \alpha'$, we get $\lambda(W(\alpha))\neq \lambda(W(\alpha'))$. This proves the uniqueness. 
By Lemma \ref{lem:simplyconnected} we get that $\pi_1(Y)$ is isomorphic to $J_2$ and all torsion line bundles on $Y$ are of the form $L_{\chi}$ for some $\chi \in \widehat{A}$. We know that $\Pcal_{\kappa\otimes \alpha}\otimes \Pcal_{\kappa}^{\otimes(-1)}$ is torsion and hence $\Pcal_{\kappa\otimes \alpha}\otimes \Pcal_{\kappa}^{\otimes(-1)}$ is isomorphic to some $L_{\chi}$.

\end{proof}

Using the above, we have the following decomposition of $A$-modules:
\begin{equation} \label{eqn:decomposition}
H^0(\widetilde{Y},\Pcal)=\bigoplus_{\chi \in \widehat{A}} H^0(Y,\Pcal_{\kappa}\otimes L_{\chi})\ .
\end{equation}

In the next section we will prove the following.
\begin{proposition}\label{prop:belkale} Suppose $m$ is odd. Then, 
\begin{enumerate}
\item $\dim_\CBbb H^0(\widetilde{Y},\Pcal)=2^{g-1}(2^g-1)$;
\item each $ H^0(Y,\Pcal_{\kappa})$, $\kappa$ odd,  is $1$-dimensional and is spanned by the Pfaffian section $s_\kappa$;
\item the set $\{s_{\kappa}\mid \kappa\in\Th^-(C) \}$, is a basis for $H^0(\widetilde{Y},\Pcal)$.
\end{enumerate}
\end{proposition}
\noindent This result should be compared with \cite[Props.\ 2.3 and 2.4]{Belkale:12} in  the even case.

%

\section{Uniformization}

\subsection{Conformal blocks via uniformization}\label{sec:confuniform}
The main result in this section is the identification of generalized theta functions on $\mathcal{M}_{\Spin(m)}$ at any level with the space of conformal blocks. Let ${V}$ be a countable dimensional complex vector space. We define the tautological line bundle ${L}_{{V}}$ on $\mathbb{P}(V)$ as the Zariski closed subset: $L_V:=\{(x,v)\in \mathbb{P}(V)\times V \mid v\in x\}$. The space $\mathbb{P}(V)\times V$ has a natural ind-scheme structure. 
We induce a ind-scheme structure on $L_{V}$ such that $L_V\rightarrow \mathbb{P}(V)\times V$ is a closed immersion. Hence by projecting onto the first factor gives an algebraic line bundle $L_V$ on $\mathbb{P}(V)$. We denote its dual by $L^*_V$.

For a simply connected complex Lie group $G$, we denote by $\mathcal{Q}_G$ the affine Grassmannian associated to $G$ with Lie algebra $\mathfrak{g}$. The affine Grassmannian is a ind-projective scheme \cite{LaszloSorger:97, kumarbook, Mathieu}. Let $\omega_0$ be the affine fundamental weight of affine Kac-Moody algebra associated to $G$ and for any positive integer $\ell$ consider $\mathcal{H}_{\omega_0}(\mathfrak{g},\ell)$ be the corresponding irreducible, integrable representation with highest weight $\omega_0$ at level $\ell$. Let $v_0(\ell)$ be a highest weight vector of $\mathcal{H}_{\omega_0}(\mathfrak{g},\ell)$. The natural map $i: \mathcal{Q}_{G}\rightarrow \mathbb{P}(\mathcal{H}_{\omega_0}(\mathfrak{g},\ell)); \ g[L^+G]\rightarrow gv_0(\ell)$ is an embedding of ind-schemes (see Appendix C in \cite{KumVBACnew} or \cite{kumarbook}).  For any $\ell >0$, we define $\mathcal{L}(\ell{\chi})$ to be the pull back   $L^*_{\mathcal{H}_{\omega_0}(\mathfrak{g},\ell)}$ under the map $i$. We can extend the definition to non-negative integers by taking the dual bundles. By \cite[Sec.\ 2.7]{sl}, it is easy to see that $\mathcal{L}_{\ell \chi}=\mathcal{L}_{\chi}^{\otimes \ell}$. The line bundle $\mathcal{L}_{\chi}$ can also be realized in a {\em line bundle associated to a character} as follows (see \cite[Lemma 4.1]{LaszloSorger:97}):
Consider the representation $\mathcal{H}_{\omega_0}(\mathfrak{g},1)$ of the affine Lie algebra $\widehat{\frg}$. This representation can be (\cite{Faltings:94}) integrated to a projective representation of the loop group $\operatorname{L}G$. We pullback the exact sequence to $\operatorname{L}G$
$$0\longrightarrow \mathbb{G}_m\longrightarrow  GL(\mathcal{H}_{\omega_0}(\mathfrak{g},1))\longrightarrow PGL(\mathcal{H}_{\omega_0}(\mathfrak{g},1))\longrightarrow 1,$$ defines a central extension which splits (\cite[Lemma 4.9]{LaszloSorger:97}) over $\operatorname{L}^{+}G$:  
$$0\longrightarrow \mathbb{G}_m \longrightarrow \widehat{\operatorname{L}G }\stackrel{p}{\longrightarrow} \operatorname{L}G\longrightarrow 1.$$ Let $\widehat{\operatorname{L}^{+}G}$ be the inverse image of $\operatorname{L}^{+}G$ via the map $p$, then $\widehat{\operatorname{L}^{+}G}\simeq \operatorname{L}G\times \mathbb{G}_m$. The  projection onto the second factor defines a character $\chi_{0}$ of  $\widehat{\operatorname{L}^{+}G}$. Then $\mathcal{L}_{\chi}$ can be described (\cite[Lemma 4.11]{LaszloSorger:97}) as the line bundle on the $\mathcal{Q}_{G}= \widehat{\operatorname{L}G}/ \widehat{\operatorname{L}^{+}G}$ associated to the  character $\chi_{0}^{-1}$. We will henceforth refer to $\mathcal{L}_{\chi}$ as the line bundle associate to the character $\chi$ of the affine fundamental weight $\omega_0$. 

We have the following proposition.

\begin{proposition}\label{pullback}Let $\pi:\Qcal_{\Spin(m)}\rightarrow \mathcal{M}_{\Spin(m)}^{-}$ be the projection from Proposition \ref{uniformlyBLS}, and $\chi$ be the character corresponding to the affine fundamental weight $\omega_0$. Then we have:
$\pi^{*}\mathcal{P}=\mathcal{L}_{\chi}$.
\end{proposition}

\begin{proof}
Consider the map $\Spin(m)\rightarrow\SL(m)$ that comes from the standard embedding. This induces a map between the affine Grassmannians $\mathcal{Q}_{\Spin(m)}\rightarrow \mathcal{Q}_{\SL(m)}$. Let $\mathcal{L}_{\chi}^0$ denote the pull-back of the determinant of cohomology line bundle on $\mathcal{M}_{\SL(m)}$ to the affine Grassmannian $\mathcal{Q}_{\SL(m)}$. By a result in \cite{KNR:94}, we know that the pull-back of $\mathcal{L}^0_{\chi}$  is $\mathcal{L}_{2\chi}$, where $\chi$ is the character and $2$ is the Dynkin index of the embedding $\mathfrak{so}(m)\rightarrow \mathfrak{sl}(m)$. Now the pull-back of the determinant of cohomology   to $\mathcal{Q}_{\SL(m)}$ is $\mathcal{L}^0_{\chi}$.  Since the Picard group of $\mathcal{M}^{-}_{\Spin(m)}$ is torsion-free, we see that $\mathcal{P}$ pulls back to $\mathcal{L}_{\chi}$ on $\Qcal_{\Spin(m)}$. 
\end{proof}

Let $V$ be a vector space of dimension $2m$ (resp.\  $2m+1$) endowed with a symmetric nondegenerate bilinear form $(\, , )$. Let $e_1,\ldots ,e_{2m}$ (resp.\  $e_{2m+1}$) be a basis of $V$ such that $(e_{i},e_{2m+1-j})=\delta_{ij}$ (resp.\  $(e_i,e_{2m+2-j})=\delta_{ij}$. The elements $H_i=E_{i,i}-E_{2m-i,2m-i}$ (resp.\  $H_i=E_{i,i}-E_{2m+1-i,2m+1-i}$) span a basis of the Cartan subalgebra of $\mathfrak{so}(2m)$ (resp.\  $\mathfrak{so}(2m+1)$. The normalized Cartan-Killing form is given by $(A,B)=\frac{1}{2}\operatorname{Tr}(AB)$. Let $L_i$ be the dual of $H_i$  where $L_i(H_j)=\delta_{ij}$ and $\omega_i=\sum_{a=1}^iL_a$ for $1\leq i \leq m-1$ be the first $m-1$ fundamental weights of both $\mathfrak{so}(2m)$ and $\mathfrak{so}(2m+1)$. 

For $\zeta \in L\SO(m)$, following \cite{Faltings:94}, we define an automorphism $\operatorname{Ad}(\zeta)$ of $\widehat{\mathfrak{so}}(m)$ by the following formula. Let $A(z)$ be an element of $\widehat{\mathfrak{so}}(m)$.
\begin{equation}\label{adjointformula}
\operatorname{Ad}(\zeta)(A(z), s):=  (\operatorname{Ad}(\zeta)A(z), s + \operatorname{Res}_{z=0}\tfrac{1}{2}\operatorname{Tr}(\zeta^{-1}(d\zeta/dt)A(z))\ .
\end{equation}

Let
\begin{equation}\label{zeta}
\zeta=\Scale[0.9]{\left(\begin{matrix}
z&&&&\\
&1&&&\\
&&\ddots &&\\
&&&1&\\
  &&&&z^{-1}
\end{matrix}\right)}\ ,
\end{equation}
regarded as an element of $L\SO(m)$. 

\begin{lemma}\label{adjoint}
Let $\pi:\widehat{\mathfrak{so}}(m) \rightarrow \operatorname{End}(\mathcal{H}_{\omega_0}(\sofrak(m),\ell))$ be an integrable representation of $\widehat{\mathfrak{so}}(m)$ and $\operatorname{Ad}(\zeta):\widehat{\mathfrak{so}}(m)\rightarrow \widehat{\mathfrak{so}}(m)$ is the  automorphism defined by formula \ref{adjointformula}, for $\zeta$ as above. Then the representation 
$\tilde{\pi}:\widehat{\mathfrak{so}}(m)\rightarrow \operatorname{End}(\mathcal{H}_{\omega_0}(\sofrak(m),\ell))$ defined by $\pi\circ \operatorname{Ad}(\zeta)$ is isomorphic to $\mathcal{H}_{\ell \omega_1}(\sofrak(m),\ell)$. 
\end{lemma}

\begin{proof}Since $\mathcal{H}_{\omega_0}(\sofrak(m),\ell)$ is irreducible under the representation $\pi$, this implies that the representation $\tilde{\pi}$ is also irreducible. Let $A(z)=\sum A_i\otimes z^i$, then by a direct computation we can check that 
\begin{equation}\label{obvious}
\operatorname{Res}_{z=0}\tfrac{1}{2}\operatorname{Tr}(\zeta^{-1}(d\zeta/dt)A(z))=\omega_1(H_0)\ ,
\end{equation}
 where $H_0$ is the diagonal part of $A_0$. From a direct calculation, we can check that if $X_{\alpha}$ is a generator of the root space of $\alpha$, then 
$\operatorname{Ad}(\zeta)X_{\alpha}(n)=X_{\alpha}(n+ \omega_1(H_{\alpha}))$, where $H_{\alpha}$ is the coroot of $\alpha$. In particular, this shows that positive nilpotent part $\widehat{\mathfrak{n}}_{+}$ of $\widehat{\mathfrak{so}}(m)$ is preserved under the automorphism $\operatorname{Ad}(\zeta)$. This implies that if $v_0\in \mathcal{H}_{\omega_0}(\sofrak(m),\ell)$ is the highest weight vector for the representation $\pi$, then $v_0$ is also the highest weight vector for the representation $\tilde{\pi}$. Thus, it remains to determine the weight of the vector $v_0$ under the representation $\tilde{\pi}$. By \eqref{obvious}, we get for $H$ in the Cartan subalgebra of $\mathfrak{so}(m)$,
$\tilde{\pi}(H,s)v_0=\ell(s+ \omega_1(H))v_0$.
This completes the proof.
\end{proof}

\begin{theorem}\label{identification}
There is a canonical isomorphism:
$H^0(\mathcal{M}_{\Spin(m)}^{-},\mathcal{P}^{\otimes\ell})\simeq\mathcal{V}^{*}_{\ell\omega_1}(\mathfrak{X},\mathfrak{so}(m),\ell)$.
\end{theorem}
\begin{proof}
The essential idea of the proof is the same as in \cite[Thm.\ 9.1]{BeauvilleLaszlo:94}. Let $\delta$ be the generator of $\pi_1(\SO(m))$. It is easy to check that the element $\zeta$ defined in \eqref{zeta} lies in the component $L\SO(m)^{\delta}$. By Proposition \ref{uniformlyBLS}, we get 
$$\mathcal{M}^{-}_{\Spin(m)}=(\zeta^{-1}L_C(\Spin(m))\zeta)\backslash \mathcal{Q}_{\Spin(m)}\ .$$ By Proposition \ref{pullback}, the line bundle $\mathcal{L}_{\chi}$ has a $\zeta^{-1}L_C(\Spin(m))\zeta$ linearization. In particular, the map $A(\alpha)\rightarrow \zeta^{-1}A(\alpha)\zeta$
 extends to the Kac-Moody group $\widehat{L\Spin}(m)$. Now it is easy to check that this map is, on the level of the Lie algebra, given by the following:
$\operatorname{Ad}(\zeta^{-1})\circ \iota$, where $\iota$ is the canonical embedding of $\mathfrak{so}(m)$ into $\widehat{\mathfrak{so}}(m)$. 
By \cite[Prop.\ 7.4]{BeauvilleLaszlo:94} and \cite{KNR:94}, the space of global sections $H^0(\mathcal{M}_{\Spin(m)}^{-},\mathcal{P}^{\otimes\ell})$ is canonically isomorphic to the space of linear forms on $\mathcal{H}_{\omega_0}(\sofrak(m),\ell)$ that vanish on the image $\zeta^{-1}{L}_C(\Spin(m))\zeta$. By Lemma \ref{adjoint}, this is same as the ${L}_C(\Spin(m))$-invariant sections $\mathcal{H}_{\ell\omega_1}(\sofrak(m),\ell)$. This, by definition, is the space of conformal blocks $\mathcal{V}^{*}_{\ell\omega_1}(\mathfrak{X},\mathfrak{so}(m),\ell)$.
\end{proof}

The following proposition compares global section of the line bundles on the moduli stack to the stack of the corresponding regularly stable moduli space. 
\begin{proposition} \label{prop:stackisspace}Let $g\geq 2$, then there is a canonical isomorphism between $$H^0(M^{-,reg}_{\operatorname{Spin}(m)},\mathcal{P}^{\otimes \ell})\simeq \mathcal{V}_{\ell\omega_1}^{\ast}(\mathfrak{X},\mathfrak{so}(m),\ell)$$
\end{proposition}
\begin{proof}Applying Theorem \ref{identification}, we are reduced to show that there are canonical identifications  $H^0(M^{-,reg}_{\operatorname{Spin}(m)},\mathcal{P}^{\otimes \ell})\simeq H^0(\mathcal{M}^{-}_{\operatorname{Spin}(m)},\mathcal{P}^{\otimes \ell})$. The codimension of the complement of the stack $\mathcal{M}^{-,reg}_{\operatorname{Spin}(m)}$ in $\mathcal{M}^{-}_{\operatorname{Spin}(m)}$ is at least $2$ (\cite[Theorem II.6]{Faltings:93}). Hence by Hartog's theorem $H^0(\mathcal{M}^{-}_{\operatorname{Spin}(m)},\mathcal{P}^{\otimes \ell})\simeq H^0(\mathcal{M}^{-,reg}_{\operatorname{Spin}(m)},\mathcal{P}^{\otimes \ell})$. Now we observe that $\mathcal{M}^{-,reg}_{\operatorname{Spin}(m)}$ (respectively ${M}^{-,reg}_{\operatorname{Spin}(m)}$) has a presentation (\cite{BLS:98}) of the form as a quotient stack $[R/\Gamma]$( respectively GIT quotient $R/\Gamma$), where $R$ is a smooth scheme and $\Gamma$ is a reductive group.  Applying Proposition 4.1 in \cite{KNR:94}, we get that $H^0(\mathcal{M}^{-,reg}_{\operatorname{Spin}(m)},\mathcal{P}^{\otimes \ell})$ can be identified with $\Gamma$-variant section of $\mathcal{P}^{\otimes \ell}$ on the scheme $R$. The later is canonically identified with $H^0(\mathcal{M}^{-,reg}_{\operatorname{Spin}(m)},\mathcal{P}^{\otimes \ell})$. This completes the proof.
	\end{proof}

For a genus $g$ curve with marked points $\Xfrak$, 
let us denote:
 $N_g(\frg,\vec{\lambda},\ell):=  \dim_\CBbb \mathcal{V}_{\vec{\lambda}}(\mathfrak{X}, \frg,\ell)$.
 We sometimes omit the notation of the Lie algebra when  it is clear. In the following,  $\frg=\mathfrak{so}(2r+1)$, and we want to compute $N_g(\omega_1,1)$ and $N_g(\vec{\omega}_r,1)$, where $\vec{\omega}_r$ is an $n$-tuple of $\omega_r$'s for $n$  even. Let $\sigma$ denote  the diagram automorphism $\omega_0\leftrightarrow \omega_1$. 

\begin{lemma}\label{lem:diagram}
Let $\sigma$ denote the Dynkin diagram automorphism that switches the $0$-th node with the first node of the affine Dynkin diagram. Then 
$N_g(\vec{\sigma}\vec{\lambda},\ell)=N_g(\vec{\lambda},\ell)$, where 
$\vec{\sigma}\vec{\lambda}=(\sigma_1\lambda_1,\ldots,\sigma_n\lambda_n)$, $\sigma_i=$  either $\sigma$ or $1$,  and $\sigma_1\cdots \sigma_n=1$.

\begin{proof}
The proof of the above follows from  factorization (cf.\ Section \ref{factorizationlemma}), the genus $0$ \cite{FS} result, and the fact that $\sigma$ induces a permutation of $P_{\ell}(\sofrak(2r+1))$.
\end{proof}
\end{lemma}
\begin{proposition} \label{prop:level-one-dimension}
For $n>0$, let $\vec{\omega}_r^{(n)}$ denote a $2n$-tuple of $\omega_r$'s. Then
$N_g(\vec{\omega}_r^{(n)},1)=2^{2g+n-1}$.
\end{proposition}
\begin{proof}
If $g=0$, then the above is a result of N. Fakhruddin \cite{Fakhruddin:12}. We will prove this using factorization (cf.\ Section \ref{factorizationlemma}) and induction on $g$.  Therefore, suppose that the result holds for genus $g-1$ and all $n$.
Then since the level one weights are precisely, $\omega_0$, $\omega_1$, and $\omega_r$, and using
Lemma \ref{lem:diagram}, factorization and induction,
\begin{align*}
N_g(\vec{\omega}_r^{(n)},1)&=N_{g-1}(\omega_0,\omega_0;\vec{\omega}_r^{(n)},1)+N_{g-1}(\omega_1, \omega_1;\vec{\omega}_r^{(n)},1)
+N_{g-1}(\omega_r,\omega_r;\vec{\omega}_r^{(n)},1)\\
&=2N_{g-1}(\vec{\omega}_r^{(n)},1) + N_{g-1}(\vec{\omega}_r^{(n+1)},1) \ \ \mbox{(By Lemma \ref{lem:diagram})}\\
&=2\cdot 2^{2(g-1)+n-1} + 2^{2(g-1)+n} =2^{2g+n-1}\ .
\end{align*}
%
%
%
%
%
%
%
%
%
%
\end{proof}
 Now by factorization and Lemma \ref{lem:diagram} we get,
 $N_g(\omega_1,1)=2N_{g-1}(\omega_1,1)+N_{g-1}(\omega_r,\omega_r,1).$ By induction on $g$, the expression for $N_{g-1}(\omega_1,1)$, and the above calculation it follows:
\begin{corollary}
 $N_g(\omega_1,1)=2^{g-1}(2^g-1)$.
\end{corollary}
\subsubsection{Proof of Proposition \ref{prop:belkale}}
 Combining Theorem \ref{identification}, Proposition \ref{prop:stackisspace} and Proposition \ref{prop:level-one-dimension}, along with the decomposition \eqref{eqn:decomposition}, we obtain Proposition \ref{prop:belkale}.  Reformulated  in terms of the stack, we have the following.
\begin{theorem} For any $r\geq 1$, $\dim_\CBbb H^0(\mathcal{M}_{\Spin(2r+1)}^{-},\mathcal{P})=2^{g-1}(2^g-1)=|\Th^-(C)|$. Moreover, the Pfaffian sections $\{s_\kappa\mid \kappa \in \Th^-(C)\}$ give a basis.
\end{theorem}

\subsection{Oxbury-Wilson conjecture}
 Let $\mathcal{P}$ be the line bundle which restricts on each component of $\mathcal{M}_{2r+1}$ to the ample generator of the Picard group (cf.\  \eqref{eqn:mr}). 
 We now prove a Verlinde formula for powers of $\mathcal{P}$. 
\begin{theorem}\label{oxbury} 
Let
$$
N_g^0(\sofrak(2r+1),\ell):= 
\left(4(\ell+2r-1)^{{r}}\right)^{g-1}\sum_{\mu \in P_{\ell}(\SO(2r+1))}\prod_{\alpha >0}\biggl( 2 \sin \pi \frac{(\mu+\rho,\alpha)}{\ell+2r-1}\biggr)^{2-2g}\ .
$$
where  $P_{\ell}(\SO(2r+1))$ denotes the set of level $\ell$ weights of $\mathfrak{so}(2r+1)$ that exponentiate to a representation of the group $\SO(2r+1)$. 
Then
\begin{equation} \label{eqn:oxbury}
\dim_\CBbb H^0(\mathcal{M}_{2r+1},\mathcal{P}^{\otimes\ell})=2N_g^0(\sofrak(2r+1),\ell)\ .
\end{equation}
\end{theorem}
\begin{proof}
By Theorem \ref{identification}, 
$H^0(\mathcal{M}_{2r+1},\mathcal{P}^{\otimes\ell})\simeq\mathcal{V}^{*}_{\omega_0}(\mathfrak{X},\mathfrak{so}(2r+1),\ell)\oplus \mathcal{V}^{*}_{\ell\omega_1}(\mathfrak{X},\mathfrak{so}(2r+1),\ell)$. Now the Verlinde formula tells us the following:
\begin{eqnarray*}
&&\dim_{\mathbb{C}}(\mathcal{V}^{*}_{\omega_0}(\mathfrak{X},\mathfrak{so}(2r+1),\ell))=\\
&&\left(4(\ell+2r-1)^{{r}}\right)^{g-1}\bigl\{\sum_{\mu \in P_{\ell}(\SO(2r+1))}\prod_{\alpha >0}\bigl( 2 \sin \pi \frac{(\mu+\rho,\alpha)}{\ell+2r-1}\bigr)^{2-2g}\\
&& + \sum_{\mu \in P_{\ell}(\SO(2r+1))^c} \prod_{\alpha >0}\bigl( 2 \sin \pi \frac{(\mu+\rho,\alpha)}{\ell+2r-1}\bigr)^{2-2g}\bigr\}\ ,
\end{eqnarray*}
where $P_{\ell}(\SO(2r+1))^{c}:=  P_\ell(\sofrak(2r+1))\setminus P_\ell(\SO(2r+1))$ is the set of level $\ell$ weights that do not exponentiate to representations of $\SO(2r+1)$.
Similarly 
\begin{eqnarray*}
&&\dim_{\mathbb{C}}(\mathcal{V}_{\ell \omega_1}^{*}(\mathfrak{X},\mathfrak{so}(2r+1),\ell))=\left(4(\ell+2r-1)^{{r}}\right)^{g-1} \times\bigl\{ \\
&&\sum_{\mu \in P_{\ell}(\SO(2r+1))}\operatorname{Tr}_{V_{\ell\omega_1}}(\operatorname{exp}2\pi i\frac{\mu+\rho}{\ell+2r-1})\prod_{\alpha >0}\bigl( 2 \sin \pi \frac{(\mu+\rho,\alpha)}{\ell+2r-1}\bigr)^{2-2g} \\
&&+  \sum_{\mu \in P_{\ell}(\SO(2r+1))^c}\operatorname{Tr}_{V_{\ell\omega_1}}(\operatorname{exp}2\pi i\frac{\mu+\rho}{\ell+2r-1}) \prod_{\alpha >0}\bigl( 2 \sin \pi \frac{(\mu+\rho,\alpha)}{\ell+2r-1}\bigr)^{2-2g}\bigr\}\ .
\end{eqnarray*}
 It follows from \cite[Lemmas 10.6 and 10.7]{Mukhopadhyay:12} that 
	$$\operatorname{Tr}_{V_{\ell\omega_1}}\left(\operatorname{exp}\left(2\pi \sqrt{-1}\frac{\mu+\rho}{\ell+2r-1}\right)\right)
	=\begin{cases} 1 &
	\mu \in P_{\ell}({\SO_{2r+1}}) \\
	-1 & \text{otherwise.}
	\end{cases}
	$$
Using this, the proof follows by taking the sum of the expressions above. 
\end{proof}
\begin{remark}
The  formula \eqref{eqn:oxbury} was conjectured in Oxbury-Wilson \cite[Conj.\ 5.2]{OxburyWilson:96}. Theorem \ref{oxbury} resolves this conjecture.

\end{remark}
For any $r,s\geq 2$, the following result is proved in \cite{OxburyWilson:96}. 
\begin{lemma} \label{lem:oxbury}
$N_g^0(\sofrak(2r+1),2s+1)=N_g^0(\sofrak(2s+1),2r+1)$.
\end{lemma}

\begin{proof}[Proof of Corollary \ref{cor:spin-dimension}]
Combine Lemma \ref{lem:oxbury} and Theorem \ref{oxbury}.
\end{proof}

\begin{remark}\label{rem:thaddds}
The  equality of dimensions in Corollary \ref{cor:spin-dimension} also holds if either $r,s=1$. 
In this case, $\SC(3)=\GL(2)$, and so
the moduli stack $\mathcal{M}_{3}$ is the disjoint union of the moduli stacks of rank 2 vector bundles with fixed trivial determinant and  determinant $=\mathcal{O}_{C}(p)$. The Verlinde formula for these spaces is  due to Thaddeus \cite{Thaddeus:92}. Also in this case, the equality of dimensions in Lemma \ref{lem:oxbury} is mentioned in \cite[Prop.\ 4.16]{OxburyWilson:96}.
\end{remark}

\section{Hecke transformations for orthogonal bundles} \label{sec:abe}

\subsection{The $\iota$-transform on orthogonal bundles} \label{sec:parabolic}
In this section we review a Hecke type elementary transformation called the $\iota$-transform introduced by T. Abe \cite{Abe:13}.
This exchanges one orthogonal bundle with a choice of isotropic line at a point for another.
 As we shall see, this operation flips the Stiefel-Whitney class.

Let $B$ be a scheme, $ \Xcal:=  C\times B$, and $\pi: \Xcal\to B$ the projection. Let $\sigma: B\to \Xcal$ be a constant section of $\pi$. 
 A \emph{parabolic structure} on an orthogonal bundle $(\Ecal, q)\to \Xfrak$ at $\sigma$ is a choice of isotropic line subbundle of $\sigma^\ast\Ecal$.  If we let $\OG(\sigma^\ast\Ecal)\to B$ denote the bundle of  Grassmannians of isotropic lines of $\sigma^\ast\Ecal$, and $\tau\to \OG(\sigma^\ast\Ecal)$ the tautological line bundle, then
the data of an orthogonal bundle with parabolic structure on $\Xcal$  may be summarized in the following diagram:
\begin{equation*} \label{eqn:parabolic}
\begin{split}
\Scale[.9]{
\xymatrix{
(\Ecal, q) \ar[r] & \Xcal \ar[d]_{\pi} &&\tau\ar[r] & \OG(\sigma^\ast\Ecal)\ar[d] \\
\sigma^\ast\Ecal\ar[r]&B\ar@/_1pc/@{>}[u]_{\sigma} && s^\ast\tau \subset\sigma^\ast\Ecal\ar[r] & B\ar@/_2pc/@{>}[u]_{s \ } 
}}
\end{split}
\end{equation*}
Let $\tau^\perp\to \OG(\sigma^\ast\Ecal)$ be the bundle orthogonal to $\tau$ in the quadratic form $q$, and let $\tau_1=\sigma^\ast\Ecal/s^\ast\tau^\perp$ be the quotient line bundle on $B$. Then we may define the locally free sheaf $\Ecal^\flat$ by the \emph{elementary transformation} (cf.\ \cite{Maruyama:82}),
\begin{equation} \label{eqn:E-flat}
0\lra \Ecal^\flat\to \Ecal\to \sigma_\ast(\tau_1)\lra 0\ .
\end{equation}
Next, let $\Ecal^\sharp=(\Ecal^\flat)^*$.
Since the normal bundle to $\sigma(B)$ is trivial and the orthogonal structure gives an isomorphism $\Ecal^*\simeq\Ecal$, dualizing \eqref{eqn:E-flat} gives
\begin{equation} \label{eqn:E-flat-dual}
0\lra \Ecal\to \Ecal^\sharp\to \sigma_\ast(\tau_1^\ast)\lra 0\ .
\end{equation}
 Now $q$ induces  maps
\begin{equation}
\label{eqn:q-sharp}
q: \Ecal^\sharp\otimes \Ecal^\sharp\lra \Ocal_\Xcal(\sigma(B))\ ,\
q: \Ecal^\flat\otimes \Ecal^\sharp\lra \Ocal_\Xcal\ .
\end{equation}
Consider the subsheaf $\Ecal^\flat\hookrightarrow \Ecal^\sharp$ coming from \eqref{eqn:E-flat-dual}.
Then $\Ecal^\sharp/\Ecal^\flat$ is a torsion sheaf supported on $\sigma(B)$, and along $\sigma(B)$ it is locally free of rank $2$ with trivial determinant and an orthogonal structure.   Since $\Ecal/\Ecal^\flat$ is isotropic, $\Ecal^\sharp/\Ecal^\flat\simeq \Ecal/\Ecal^\flat\oplus \sigma_\ast(\tau_1^\ast)$.  Finally, we define  $\Ecal^\iota\subset \Ecal^\sharp$ to be the kernel of the map $\Ecal^\sharp\to \Ecal/\Ecal^\flat$. Equivalently, there is an exact sequence
\begin{equation} \label{eqn:E-iota}
0\lra \Ecal^\flat\to \Ecal^\iota\to \sigma_\ast(\tau_1^\ast)\lra 0\ .
\end{equation}
Then $\Ecal^\iota$ inherits an orthogonal structure $q^\iota$ from \eqref{eqn:q-sharp}. Moreover, the exact sequence \eqref{eqn:E-iota} determines an isotropic line $s^\iota\subset \sigma^\ast(\Ecal^\iota)$. Finally, from \eqref{eqn:E-flat} and \eqref{eqn:E-iota},  the trivialization of $\det\Ecal$ induces one for $\det\Ecal^\iota$.
\begin{definition} \label{def:iota}
The $\iota$-transform is the map: $(\Ecal, q, s )\mapsto (\Ecal^\iota, q^\iota, s^\iota)$.
\end{definition}

\begin{remark} \label{rem:functorial}
It is clear that the $\iota$-transform is functorial with respect to base change. Moreover T. Abe \cite{Abe:13} has shown that the $\iota$-transform is an involution. 
\end{remark}

\subsection{The $\iota$-transform on Clifford bundles}
We now show that the $\iota$-transform sends a bundle in one component of $\Mcal_{\SO(m)}$ to the other one.  
Fix a point $p\in C$, and recall that $C^\ast=C-\{p\}$.
 Let $\Mcal_{\Spin(m)}^{par}$ denote 
the moduli stack of pairs $({\mathcal S}, \Psf)$, 
where $\mathcal S$ is a $\Spin(m)$-bundle on $C$ and 
$\mathsf P$ is a maximal parabolic subgroup of the fiber $\sigma^\ast{\mathcal S}$ preserving an isotropic line $s $ in the fiber of the associated orthogonal bundle at $p$. 
Similarly, let $\Mcal_{\SO(m)}^{par}$ be the moduli stack of tuples $({\mathcal E}, q, s )$, where $({\mathcal E}, q)$ is a
 rank $m$ orthogonal bundle, and $s$ is an isotropic line in the fiber ${\mathcal E}_p$.
We then have a map
$
\Mcal_{\Spin(m)}^{par} \to \Mcal_{\SO(m)}^{par}$,
$({\mathcal S}, \Psf) \mapsto ({\mathcal E}, q, s )
$.
Forgetting the parabolic structure gives a morphism
$
\Mcal_{\Spin(m)}^{par}\rightarrow \Mcal_{\Spin(m)}\rightarrow \Mcal_{\SO(m)}^{+}
$.

We wish to define a 
morphism
$
\Mcal_{\Spin(m)}^{par}\rightarrow \Mcal_{\Spin(m)}^-\rightarrow \Mcal_{\SO(m)}^{-}
$.
Associated to $(\Ecal, q, s )$ we obtain 
a new orthogonal bundle with isotropic line $({\mathcal E}^\iota, q^\iota, s^\iota)$ defined in the previous section. By Remark \ref{rem:functorial}, this gives an involution of  stacks:
$
\iota:
\Mcal^{par}_{\SO(m)}\to \Mcal^{par}_{\SO(m)}
$.
This can be described explicitly in terms of transition functions  as follows.  First, since the result we wish to prove is topological it suffices to  work locally in the analytic topology, and in fact at a closed point of $B$.  We therefore let
 $\mathcal S$ be a spin bundle and $({\mathcal E},q)$ the associated
orthogonal bundle; $\Scal=\Spin(\Ecal, q)$.
Let $\Delta\subset C$ be a disk centered at $p$, and $\sigma: \Delta\to {\mathcal S}$ a section. 
This gives a trivialization of ${\mathcal S}$ and a local frame $e_1,\ldots, e_m$ for $\mathcal E$ on $\Delta$ with respect
to which the quadratic structure is, say, of the form $q_{ij}=\delta_{i+j-1,m}$.
Similarly, we may choose a section of $\left.{\mathcal S}\right|_{C^\ast}$. Set $\Delta^\ast=C^\ast\cap \Delta$.
Let $\hat\varphi: \Delta^\ast\to \Spin({\mathcal E},q)$ denote the transition function gluing the  bundles $\left.{\mathcal S}\right|_\Delta$ and $\left.{\mathcal S}\right|_{C^\ast}$, 
and let $\varphi: \Delta^\ast\to \SO(\left.({\mathcal E}, q)\right|_{\Delta^\ast})$ be the quotient transition function for 
 $({\mathcal E},q)$.
The transformed bundle ${\mathcal E}^\iota$ (cf.\ Section \ref{sec:parabolic}) is defined by modifying  $\varphi$  
 via
$\zeta: \Delta^\ast\to \SO(\left.({\mathcal E}, q)\right|_{\Delta^\ast})$, where $\zeta$ is as in \eqref{zeta}.
Write $z=\exp(2\pi i\xi)$. Then there is a well-defined lift 
$\widehat \zeta: \Delta^\ast\to \SC(\left.({\mathcal E}, q)\right|_{\Delta^\ast})$, 
given by 
\begin{equation} \label{eqn:bhat}
\widehat \zeta(z)= \exp(\pi i\xi)
\exp\left((\pi i\xi/2)(e_1e_m-e_me_1)
\right)\ .
\end{equation}
One checks that $\widehat \zeta$ is well-defined under $\xi\mapsto\xi+1$, and
the projection of $\widehat \zeta$ under the map $\SC({\mathcal E},q)\to \SO({\mathcal E},q)$
 recovers $\zeta$.
Gluing  the trivial $\SC$-bundles over $\Delta$ and $C^\ast$
via $\hat\varphi(z) \widehat \zeta(z)$, 
we define a new Clifford bundle ${\mathcal S}^\iota$.
The 
associated orthogonal bundle (with transition function $\varphi(z)\zeta(z)$) coincides with ${\mathcal E}^\iota$.
With this understood, the main observation is the following.
\begin{proposition}
The $\iota$-transform switches the Stiefel-Whitney class, i.e.\ if ${\mathcal E}
\in \Mcal_{\SO(m)}^+$ then
$ {\mathcal E}^{\iota}\in \Mcal^-_{\SO(m)}$, and vice-versa.
\end{proposition}

\begin{proof}
It suffices to check the spinor norm of ${\mathcal S}^\iota$. 
But from 
\eqref{eqn:bhat},  $\Nm({\mathcal S}^\iota)$ is a line bundle with transition function on $\Delta^\ast$ given by:
$$
\Nm(\hat\varphi\widehat \zeta)=\exp(2\pi i\xi)\Nm(\hat\varphi(z)) \Nm
\left(\exp\left((\pi i\xi/2)(e_1e_m-e_me_1)\right)\right) = z\ ,
$$
since $\hat\varphi(z)$ and $\exp\left((\pi i\xi/2)(e_1e_m-e_me_1)\right)\in \Spin({\mathcal
E},q)$.
Therefore,  $\Nm({\mathcal S}^\iota)\simeq{\mathcal O}_C(p)$.
\end{proof}

It will be useful to keep in mind the following diagram:

\begin{equation}
 \label{eqn:hecke}
 \begin{split}
 \Scale[.9]{
\xymatrix{
& \Mcal^{par}_{\Spin(m)}\ar[ld]_{\pr^+} \ar[rd]^{\pr^-} &\\
\Mcal^+_{\Spin(m)}\ar[d]&& \Mcal^-_{\Spin(m)}\ar[d] \\
\Mcal^+_{\SO(m)}&& \Mcal^-_{\SO(m)} 
}}
\end{split}
\end{equation}
Here, $\pr^+$ is the forgetful map that discards the parabolic structure, and $\pr^-$ is the $\iota$-transform described above.

\begin{remark} \label{rem:fiber}
As in the case of $\SO(m)$ bundles, the $\iota$-transform on $\SC(m)$ bundles  is reversible. In particular,  the fiber of $\pr^-$ is a copy of $\OG$, and so  it is connected and projective.
\end{remark}

\subsection{The $\iota$-transform and the Pfaffian bundle} \label{sec:det-pfaffian}
We first use the $\iota$-transform to prove the following.

\begin{proof}[Proof of Proposition \ref{prop:pfaffian-divisor}]
	For an orthogonal bundle, we have (cf.\ \cite[Prop.\ 4.6]{Oxbury:99})
	$
	w_2(\Ecal)\equiv h^0(C,\Ecal\otimes\kappa)+ mh^0(C,\kappa)$ mod $2$.
	Hence, for $\Ecal\in \Mcal^-_{\SO(m)}$, if either $m$ or $\kappa$ are even, then $h^0(C,\Ecal\otimes\kappa)$ is odd. If both $m$ and $\kappa$ are odd, then by 
	\cite[Prop.\ 4.6]{Oxbury:99}, $h^0(C,\Ecal\otimes\kappa)=0$   for generic $\Ecal$. On the other hand, choose any theta characteristic $\kappa_0$ with $h^0(C, \kappa_0)\neq 0$.
	Write $m=2r+1$, and let:
	$\Ecal_0= (\kappa_0\otimes \kappa^{-1})^{\oplus r}\oplus \Ocal_C\oplus (\kappa\otimes \kappa_0^{-1})^{\oplus r}$,
	with the obvious orthogonal structure. 
	Then 
	\begin{equation} \label{eqn:h0}
	h^0(C, \Ecal_0\otimes \kappa)=(m-1) h^0(C, \kappa_0) +h^0(C, \kappa)\geq m-1\geq 2\ .
	\end{equation}
	Pick an isotropic line of $\Ecal_0$ at a point, and perform the elementary transformation in \eqref{eqn:E-flat}. Then by \eqref{eqn:h0}, $h^0(C, \Ecal^\flat\otimes \kappa)\neq 0$, which by \eqref{eqn:E-iota} implies that the $\iota$-transform $\Ecal=\Ecal_0^\iota\in\Mcal^-_{\SO(m)} $ has $h^0(C, \Ecal\otimes \kappa)\neq 0$. This completes the proof.
\end{proof}
Recall the notation from Section \ref{sec:parabolic}.  Choose $\kappa\in\Th(C)$ and denote the pull-back to $\Xcal$ by $\pr^\ast\kappa$. Then we have the next result.
\begin{proposition} \label{prop:det}
For a family of orthogonal bundles $(\Ecal, q)\to \Xcal$, and $\Ecal^\iota$ the $\iota$-transform,
$$
\Det R\pi_\ast(\Ecal^\iota\otimes \pr^\ast\kappa)\simeq \Det R\pi_\ast(\Ecal\otimes \pr_1^\ast\kappa)\otimes (s^\ast\tau)^{\otimes 2}\ .
$$
\end{proposition}

\begin{proof}

First, notice that the quadratic form gives an isomorphism $\tau_1\simeq s^\ast(\tau^*)$.  Let
$\kappa_\sigma=\sigma^\ast\pr_1^\ast\kappa$.
Then using \eqref{eqn:E-flat} and \eqref{eqn:E-iota}, 
\begin{align*}
\Det R\pi_\ast (\Ecal\otimes \pr_1^\ast\kappa)
&
\simeq \Det R\pi_\ast (\Ecal^\flat\otimes \pr_1^\ast\kappa)\otimes s^\ast(\tau^*)\otimes \kappa_\sigma \notag\ , \\
\Det R\pi_\ast (\Ecal^\iota\otimes \pr_1^\ast\kappa)
&
\simeq \Det R\pi_\ast (\Ecal^\flat\otimes \pr_1^\ast\kappa)\otimes s^\ast \tau\otimes \kappa_\sigma  \notag\ .
\end{align*}
The result follows.
\end{proof}

\begin{corollary} \label{cor:determinant}
When pulled back to $\Mcal_{\Spin(m)}^{par}$, 
the Pfaffian bundles on $\Mcal_{\Spin(m)}^{\pm}$ for any theta characteristic $\kappa$ are related by
$
(\pr^-)^\ast \Pcal_\kappa \simeq (\pr^+)^\ast \Pcal_\kappa\otimes  s^\ast(\tau^*)$.
\end{corollary}

\begin{proof}
From  Propositions \ref{LSpfaff} and \ref{prop:det},
we see that 
$
\left[(\pr^-)^\ast \Pcal_\kappa\right]^{\otimes 2} \simeq \left[(\pr^+)^\ast \Pcal_\kappa\otimes s^\ast(\tau^*)\right]^{\otimes 2}
$.
The result follows from the fact that $\Pic(\Mcal_{\Spin(m)}^{par})$ is torsion-free (cf.\ \cite[Thm.\ 1.1]{LaszloSorger:97}).
\end{proof}

\subsection{Geometric version of Theorem \ref{identification}}
By Corollary \ref{cor:determinant} and Remark \ref{rem:fiber}, we have
$$
H^0(\Mcal^-_{\Spin(m)}, \Pcal^{\otimes \ell})= H^0(\Mcal^{par}_{\Spin(m)}, (\pr^-)^\ast \Pcal^{\otimes \ell}) = H^0(\Mcal^{par}_{\Spin(m)}, (\pr^+)^\ast \Pcal^{\otimes \ell}\otimes s^\ast(\tau^*)^{\otimes \ell})\ .
$$
By the Borel-Weil theorem, the highest weight representation $V_{\ell\omega_1}$ of $\Spin(m)$ is given by the global sections of 
$(\tau^*)^{\otimes \ell}\to \OG$. It then follows as in  \cite[Thm.\ 1.2]{LaszloSorger:97} or \cite[Props.\ 6.5 and 6.6]{Pauly:96} that the space of sections of $\Pcal^{\otimes \ell} \to \Mcal^-_{\Spin(m)}$ is isomorphic to the space of conformal blocks. This gives an alternative proof of Theorem \ref{identification}.

\section{Hitchin connection for twisted spin bundles}\label{twistedHiggs}

\subsection{Higgs bundles} \label{sec:higgs}

Let $M^\theta_{\SO(m)}$ denote the coarse quasi-projective moduli space of semistable \emph{$\SO(m)$-Higgs bundles} on $C$,
 and let $M^{\theta, reg}_{\SO(m)}\subset M^\theta_{\SO(m)}$ denote the regularly stable locus, i.e. pairs $(\Ecal,\theta)$, $\theta\in H^0(C,  \Ad(\Ecal)\otimes\omega_C)$, such that $\operatorname{Aut}(E,\theta)=Z(\SO(m))$, where $Z(\SO(m))$ denotes the center of $\SO(m)$.
We assume in this section that  $C$ has genus $g\geq 2$.
 Then $M^{\theta, reg}_{\SO(m)}$ is smooth (\cite[Proposition 2.14]{BradlowGpradagothen}), with complement of codimension $\geq 2$ (\cite[Theorem II.6]{Faltings:93}).
Let 
\begin{equation} \label{eqn:hitchin-base}
B(m)
=\begin{cases}
\bigoplus_{i=1}^r H^0(C, \omega_C^{\otimes 2r}) & m=2r+1 \\
\bigoplus_{i=1}^{r-1} H^0(C, \omega_C^{\otimes 2r})\oplus H^0(C,\omega_C^r) & m=2r\ .
\end{cases}
\end{equation}
 The Hitchin map $H: M^\theta_{\SO(m)}\to B(m)$ is a dominant, proper morphism.
 Away from the discriminant locus $\Delta\subset B_{\SO(m)}$, $H$ is a smooth fibration by abelian varieties, and $M^{\theta, ns}_{\SO(m)}:=M^\theta_{\SO(m)}\bigr|_{B(m)-\Delta}\subset M^{\theta, reg}_{\SO(m)}$ (see \cite{Hitchin:87b,Hitchin:07}, and also \cite[Lemma 4.2]{DonagiPantev:12}).
 There is an action of $\CBbb^\times$ on $M^\theta_{\SO(m)}$, and $H$ is equivariant with respect to multiplication on $H^0(C, \omega_C^{\otimes k})$ with weight $k$. 
We will need  the following.

\begin{proposition} \label{prop:flat-trivial}
Let $V\to M^{\theta, reg}_{\SO(m)}$ be a flat bundle. If $V$ restricted to a general fiber of $h$ is trivial, then $V$ is trivial.
\end{proposition}

The proposition is an immediate consequence of  the following.

\begin{lemma}
If $A$ is a general fiber of $h$, then the inclusion $A\hookrightarrow M^{\theta, reg}_{\SO(m)}$ induces a surjection 
$\pi_1(A)\rightarrow\pi_1(M^{\theta, reg}_{\SO(m)})$.
\end{lemma}

\begin{proof}
Since $\Delta$ is codimension $1$, inclusion induces a surjection
$
\imath_\ast: \pi_1(M^{\theta, ns}_{\SO(m)})\to \pi_1(M^{\theta, reg}_{\SO(m)})
$.
Hence, we have the diagram:
$$
\Scale[.9]{
\xymatrix{
\pi_1(A)\ar[r]^{a_\ast\qquad} \ar[dr]&\pi_1(M^{\theta, ns}_{\SO(m)}) \ar[r]^{h_\ast}\ar[d]^{\imath_\ast}  & \pi_1(B_(m)-\Delta)\ar[r] & \{1\} \\
& \pi_1(M^{\theta, reg}_{\SO(m)}) \ar[d] && \\
& \{1\} &&
}}
$$
Notice  that every element of $\pi_1(B(m)-\Delta)$ is represented by the boundary of a transverse disk. More precisely, for $\gamma\in \pi_1(B(m)-\Delta)$
 and $D\subset \CBbb$ the unit disk, there is  an embedding $D\hookrightarrow B(m)$, $D\cap\Delta=\{0\}$, and
  such that $\partial D$ represents $\gamma$. 
  
  Next, consider the fiber in $M^{\theta, reg}_{\SO(m)}$ of $h$ over $\{0\}$.  Since the image in $B(m)$ of fibers  contained in the critical locus of $h$ lies in a set of codimension $2$, we may assume without loss of generality that there is a regular point $x\in h^{-1}(0)$.  There are therefore local smooth coordinates about $x$ with respect to which $h$ is given by projection.
   By shrinking the disk if necessary, it follows that there is a local section $\sigma:D\to M^{\theta, reg}_{\SO(m)}$ of $h$. The image $\sigma(\partial D)$ is therefore a loop in $M^{\theta, ns}_{\SO(m)}$, contractible in $M^{\theta, reg}_{\SO(m)}$,  that projects to $\partial D$.  
   We conclude that for any $\gamma\in \pi_1(B(m)-\Delta)$ there is $\beta\in \pi_1(M^{\theta, ns}_{\SO(m)})$ such that 
$h_\ast(\beta)=\gamma$ and $\imath_\ast(\beta)=1$.

The lemma now follows easily. For if $\alpha\in \pi_1(M^{\theta, reg}_{\SO(m)})$, then by surjectivity of $\iota_\ast$ there is $\tilde\alpha\in \pi_1(M^{\theta, ns}_{\SO(m)})$ such that $\imath_\ast(\tilde\alpha)=\alpha$. By the discussion in the previous paragraph, we can find $\beta
\in \pi_1(M^{\theta, ns}_{\SO(m)})$ such that $h_\ast(\beta)=h_\ast(\tilde\alpha)$ and $\imath_\ast(\beta)=1$. But then $\beta^{-1}\tilde\alpha$ is in the kernel of $h_\ast$, and so is in the image of $a_\ast$, while at the same time it projects by $\imath_\ast$ to $\alpha$.  Therefore $\imath_\ast\circ a_\ast$ is surjective.
\end{proof}

The application of the previous result that we need is the following
\begin{corollary}  \label{cor:nontrivial}
Let $\chi: \pi_1(M^{reg}_{\SO(m)})\to \{\pm 1\}$ be a nontrivial character with associated line bundle $L_\chi$. Then the pullback of $L_\chi$ to $M^{\theta, reg}_{\SO(m)}$ is nontrivial on generic fibers of the Hitchin map.
\end{corollary}

\begin{proof} Let $M^{\theta,s}_{{\SO(m)}}$ be the locus of stable ${\SO(m)}$-Higgs bundles. By  \cite[Thm.\ II.6]{Faltings:93}, the set $M_{\SO(m)}^{\theta,hs}$ consisting of stable Higgs bundles $(E,\theta)$ where $E$ is not stable as a principal ${\SO(m)}$ bundle has codimension at least $2$ in $M^{\theta,s}_{\SO(m)}$. Clearly, $M^{\theta,reg}_{\SO(m)}\backslash T^{\ast}M^{reg}_{\SO(m)}$ is contained in $M^{\theta,hs}_{\SO(m)}$. Moreover by \cite[Theorem II.6]{Faltings:93}, the complement of  $M^{\theta,reg}_{\SO(m)}$ in $M^{\theta,s}_{\SO(m)}$ has codimension at least $4$. This implies that the complement of $T^\ast M^{reg}_{\SO(m)}\subset M^{\theta, reg}_{\SO(m)}$ has positive codimension. Hence, there is a surjection
$\pi_1(M^{\theta, reg}_{\SO(m)})\to  \pi_1(M^{reg}_{\SO(m)})$. The result then follows from Proposition \ref{prop:flat-trivial}.
\end{proof}

\subsection{A vanishing theorem}  We use the notation $Y=M_{\SO(m)}^{-,reg}$ and $\widetilde Y=p^{-1}(Y)\subset M^{-, reg}_{\Spin(m)}$ from Section \ref{sec:divisor}.
The following is a key assumption in the construction of the Hitchin connection.
\begin{proposition} \label{prop:vanishing}With the  notation above,
\begin{enumerate}
\item $H^0(\widetilde Y, T\widetilde Y)=\{0\}$;
\item $H^1(\widetilde Y, {\mathcal O}_{\widetilde Y})=\{0\}$.
\end{enumerate}
\end{proposition}

\begin{proof}
(1) 
Since $\widetilde Y\to Y$ is \'etale
it suffices to prove the vanishing of 
$$H^0(\widetilde Y, T\widetilde Y)=\bigoplus_{\chi\in J_2(C)} H^0(Y, TY\otimes L_\chi) \ .
$$
Fix a character $\chi$, and
suppose 
$\sigma\in H^0(Y, TY\otimes L_\chi)$ is nonzero.
Then pulling back  and contracting with the fibers gives  a section  $f$ of the flat line bundle (also denoted $L_\chi$) corresponding to $\chi$ on  $T^\ast Y$. 
 By Corollary \ref{cor:nontrivial},  it immediately follows that $\sigma$ vanishes if $\chi\neq 1$, for since $L_\chi$ is torsion it could have no sections on the generic fiber unless it is trivial there.
   For $\chi=1$, 
the argument is exactly as in \cite[p.\ 373]{Hitchin:90}.  Namely, 
by Hartog's theorem $f$ extends to a function on $M_{\SO(m)}^\theta$.
Since the Hitchin map  $H:M_{\SO(m)}^\theta\to B(m)$ is proper with connected fibers, it follows that  $f$ is the pull-back of a function $g$ on  $B(m)$. 
On the other hand, $f$ is homogeneous of degree 1, whereas the minimal degree of homogeneity for the action on $B(m)$ is $2$ (see the discussion following \eqref{eqn:hitchin-base}).  
Since $H$ is equivariant  we conclude that $f$, and hence also $\sigma$, must vanish.
%

For (2), we first note that provided $g\geq 2$,  the complement of $\widetilde Y\subset M^{-}_{\Spin(m)}$ is of codimension at least $3$. This follows from the fact that the complement of $Y$ in $M^{-}_{\SO(m)}$ of codimension at least $3$ if $g\geq 2$ and the map $p:M^{-}_{\Spin(m)} \rightarrow M^{-}_{\SO(m)}$ is finite and dominant (see \cite[Thm.\  II.6]{Faltings:93}, \cite[Appendix]{Laszlo} and recall $m\geq 5$). The proof of part (2) now follows from Scheja's theorem \cite{Scheja:61}.
 To prove the result, it then suffices  by Scheja's theorem  to prove that $H^1(M_{\Spin(m)}^-, {\mathcal O})=\{0\}$.
For this we closely follow the proof in \cite[Thm.\ 2.8]{KumarNarasimhan:97}. First, observe that by  \cite[Lemma 7.3]{BLS:98} (see also  \cite{Ramanathan:96})  the moduli space $M^{-}_{\Spin(m)}$ is a good quotient of a projective scheme $R$ by a reductive group $\Gamma$. From  \cite{Boutot:87} it then follows that $M^{-}_{\Spin(m)}$ is Cohen-Macaulay, normal and has rational singularities. Let $\omega_M$ denote the dualizing sheaf. 

Let $\widehat \Ecal\to C\times R$ be the universal bundle. By construction of the GIT quotient it follows that the adjoint vector bundle
$\operatorname{Ad}(\mathcal{\widetilde \Ecal})$ descends to a vector bundle on $C\times \widetilde{Y}$, which we also denote by $\operatorname{Ad}(\widetilde \Ecal)$. Since $\widetilde{Y}$ is an \'etale cover of $M_{\SO(m)}^{-,reg}$, deformation theory tells us that $T_{[\Ecal]}{\widetilde{Y}}$ can be identified with $H^1(C,\operatorname{Ad}(\Ecal))$. Moreover, since $\Ecal$ is regularly stable,  $H^0(C, \Ad (\Ecal))=\{0\}$. In particular, it follows from the definition of determinant of cohomology that 
$\Det(R\pi_\ast \Ad\widehat\Ecal)^*|_{\widetilde Y}=\left.\omega_M\right|_{\widetilde Y}$.
But $\Det(R\pi_\ast \Ad\widehat\Ecal)$ extends to an invertible sheaf  on the entire moduli space $M_{\Spin(m)}^{-}$, and hence in particular it is reflexive. Furthermore,  dualizing sheaves on Cohen-Macaulay normal varieties are also reflexive. Since the complement of $\widetilde Y$ is codimension $\geq 2$, then  by \cite[Lemma 2.7]{KumarNarasimhan:97} we get that $\omega_M$ is locally free, and hence by definition $M_{\Spin(m)}^{-}$ is Gorenstein. Now $\Det(R\pi_\ast \Ad\widehat\Ecal)$ is ample and the Picard group of $M_{\Spin(m)}^{-}$ is isomorphic to $\mathbb{Z}$ (\cite{BLS:98}). It follows from  Serre duality \cite[Cor.\ III.7.7]{Hartshorne:77} and the Grauert-Riemenschneider vanishing theorem  \cite{GrauertRiemenschneider:70} that $H^i(M_{\Spin(m)}^{-}, \mathcal{O})=\{0\}$  for $i>0$. 
\end{proof}

\subsection{The Hitchin connection}
Let $\Ccal\to B$ be a smooth family of genus $g$ curves, $B$ smooth.  Let $\pi: \underline M_{\Spin(m)}^{-, reg}\to B$ denote the universal moduli space of regularly stable twisted $\Spin(m)$ bundles on the fibers of $\Ccal$, with universal Pfaffian bundle $\underline\Pcal$.  Then the direct image sheaf $\pi_\ast\underline\Pcal^\ell$ is a holomorphic vector bundle over $B$ with fiber 
$H^0(M_{\Spin(m)}^{-, reg}, \Pcal^{\otimes \ell})$.  We wish to find a connection on the projective bundle
$\PBbb(\pi_\ast\underline\Pcal^{\otimes \ell})\to B$.
Following the method of Hitchin \cite{Hitchin:90}, the connection is constructed as a ``heat operator'' on the smooth sections of $\underline\Pcal^{\otimes \ell}\to\underline M_{\Spin(m)}^{-, reg}$. As noted in that reference (see also \cite{Andersen:12}), the procedure can be applied to the open moduli space of (regularly) stable points. 

 Given $[\Ecal]\in \widetilde Y\subset M_{\Spin(m)}^{-, reg}$, recall that
$T{\widetilde Y}|_{[\Ecal]} \simeq H^1(C, \Ad\Ecal)$.
By Serre duality, $$H^0(C, \Ad \Ecal\otimes\omega_C)\simeq H^1(C, \Ad\Ecal)^\ast \ .$$
Combine this with the cup product:
$$
H^1(C, TC)\otimes H^0(C, \Ad \Ecal\otimes\omega_C)\to H^1(C, \Ad \Ecal)
$$
and the identification above to obtain a map:
$
\tau: H^1(C, TC)\to H^0(\widetilde Y, S^2T\widetilde Y)
$.
Let $\Dcal^i(\Pcal^{\otimes \ell})$ denote the sheaf of differential operators of order $i$ on $\widetilde Y$.  
Given a nonzero section $s\in H^0(\widetilde Y, \Pcal^{\otimes \ell})$, evaluation on $s$ gives a length $2$ complex: $\Dcal^i(\Pcal^{\otimes \ell})\to \Pcal^{\otimes \ell}$, from which we obtain 
a hypercohomology group $\HBbb_s^1(\widetilde Y, \Dcal^1(\Pcal^{\otimes \ell}))$.
Let $\delta: H^0(\widetilde Y, S^2T{\widetilde Y})\to \HBbb_s^1(\widetilde Y, \Dcal^1(\Pcal^{\otimes \ell}))$ be given by the coboundary associated to
\begin{equation} \label{eqn:diff-exact}
0\lra \Dcal^1(\Pcal^{\otimes \ell})\lra \Dcal^2(\Pcal^{\otimes \ell})\lra S^2T{\widetilde Y}\lra 0\ .
\end{equation}
The main result is the following.
\begin{theorem} \label{thm:hitchin} 
\begin{enumerate}With the above notation
\item
Given a deformation $[\mu]\in H^1(C, TC)$, let $\dot I$ denote the variation of the almost complex structure on $\widetilde Y$.
The association of the class
$\displaystyle A(\dot I, s):= \frac{\delta\tau[\mu]}{2i(\ell+g^\vee)}\in \HBbb_s^1(\widetilde Y, \Dcal^1(\Pcal^{\otimes \ell}))$ to each section $s\in H^0(\widetilde Y, \Pcal^{\otimes \ell})$ defines a   connection on $\PBbb(\pi_\ast\underline\Pcal^{\otimes \ell})\to B$. 
 \item The connection commutes with the action of $J_2(C)$.
 \item The connection is projectively flat.
 \item Under the identification \eqref{eqn:stack-moduli}, the connection agrees with the TUY connection.
 \end{enumerate}
\end{theorem}

\begin{proof} The existence (1) of the connection follows if we show that the hypotheses (i) and (ii) of \cite[Thm.\ 1.20]{Hitchin:90} are satisfied. Condition (i) follows trivially from the vanishing result Proposition \ref{prop:vanishing} above.  For (ii), we must show that $-i A(\dot I, s)$ projects to the Kodaira-Spencer class under the map $\sigma : \HBbb_s^1(\widetilde Y, \Dcal^1(\Pcal^{\otimes \ell}))\to H^1(\widetilde Y, T\widetilde Y)$. This follows exactly as in \cite[p.\ 365]{Hitchin:90}, where in \cite[Lemma 2.13]{Hitchin:90} one replaces $\End_ 0 E$ with $\Ad E$, and in \cite[Prop.\ 2.16]{Hitchin:90} one uses the cup product defined above.
	
	 Item (2) follows
 by the same argument as in \cite[Cor.\ 4.2]{Belkale:09}.
 
  The flatness (3) proved in \cite[Thm.\ 4.9]{Hitchin:90} is a consequence of two other results of that paper: the fact that certain trace functions Poisson commute (\cite[Prop.\ 4.2]{Hitchin:90})  follows in our case from the integrable system structure for orthogonal bundles (cf.\ \cite{Hitchin:87b}); and by (2) above the version of \cite[Prop.\ 4.4]{Hitchin:90} that we need 
states that the map
$$
f: H^1(C,TC) \lra H^0( Y, S^2 T Y) \quad ,\quad 
f([\mu])(\alpha,\alpha)=\int_C(\alpha,\alpha)\, \mu
$$
is an isomorphism. This also follows as in \cite[bottom of p.\ 373]{Hitchin:90}. Indeed, for $(G^{ij})\in H^0( Y, S^2 T\widetilde Y)$, then by the same argument as in the proof of Proposition \ref{prop:vanishing} (1), the pull-back of $(G^{ij})$ via the Hitchin map extends as a function on $M_{\SO(m)}^\theta$, and thus descends to a function on $B(m)$, homogeneous of degree 2. The homogeneous degree $2$ functions on $B(m)$ are identified
 in \eqref{eqn:hitchin-base} with $H^0(C,\omega_C)^\ast\simeq H^1(C, TC)$.
 
  Finally, (4) follows from the proof in  \cite{Laszlo} where the only properties of the Hitchin connection used are the ones stated above and the uniformization theorem. For the case of (simply connected) $G$-bundles, Laszlo first considers the uniformization  $\pi:\mathcal{Q}_G=LG/L^{+}G\rightarrow \mathcal{M}_G$ of the moduli stack of $G$-bundles over a curve $C$ and restricts to the regularly stable locus $\mathcal{Q}^{0}_G$ of $\mathcal{Q}_G$. The ample generator $\mathcal{P}$ of the Picard group  of $\mathcal{M}_G$ pulls back to  give a line bundle $\mathcal{L}_{\chi}$, where $\chi$ is the character of the affine fundamental weight $\omega_0$. The Sugawara construction  then provides a second order differential operator $T$ (see  \cite[eq.\ (8.8)]{Laszlo}) on $\mathcal{L}_{\chi}^{\otimes\ell}$ (see Section \ref{sec:confuniform} for notation). Following \cite{DS}, 
    Laszlo takes  an \'etale quasi-section $r:N\twoheadrightarrow M_{G}^{reg}$ and $q:N\rightarrow \mathcal{Q}^{0}_G$  of the map $\pi: \mathcal{Q}_G^0\rightarrow M_{G}^{reg}$
    $$
    \xymatrix{
& \mathcal{Q}^{0}_{G}\ar[d]^{\pi} &\\
N\ar[ru]^{q}\ar@{->>}[r] & M^{reg}_{G}&    
}
    $$
     and defines an action of the differential operator $T$ on the line bundle $r^*\mathcal{P}^{\otimes \ell}$(also isomorphic to $q^*\mathcal{L}_{\chi}^{\otimes \ell}$). Once this is done, then the symbol of the ``heat operator" in the Hitchin connection is shown to be equal to the operator $T$ coming from the Sugawara construction (see \cite[Sec.\ 8.14]{Laszlo}). 
    
    In the present situation for $\mathcal{M}_{\operatorname{Spin}(m)}^{-}$, uniformization (Proposition \ref{uniformlyBLS}) and the existence of \'etale quasi-section follows from \cite{BLS:98, DS}. Further the pull back of the Pfaffian line bundle $\mathcal{P}$ on $\mathcal{Q}_{\operatorname{Spin}(m)}$ is also given by the same character $\chi$ of the affine fundamental weight $\omega_0$ (Section \ref{sec:confuniform}). Thus the set-up is exactly as in \cite{Laszlo}.

\end{proof}

We will call the projective connection constructed in Theorem \ref{thm:hitchin} the \emph{Hitchin connection}. The main consequence of the existence of a Hitchin connection is the following.

\begin{proposition}
The Pfaffian sections $s_{\kappa}$ are projectively flat with respect to the Hitchin connection.
\end{proposition}

\begin{proof}
The action of $J_2(C)$ commutes with the projective heat operator. Hence, the Hitchin connection preserves the spaces $H^0(M_{\SO(m)}^{reg}, \Pcal_\kappa)$. Now these one dimensional spaces are spanned by Pfaffian sections, making them projectively flat. 
\end{proof}
\begin{remark}
	If we consider the isomorphism of $\operatorname{SC}(3)\simeq \GL(3)$, then $\mathcal{M}^{-}_{\operatorname{Spin}(3)}$  is just the stack of rank two bundles on a curve with determinant $\mathcal{O}_C(p)$. The corresponding moduli space in this case is smooth (GIT stability and GIT semistability coincide) and the Hitchin connection is constructed in \cite{Hitchin:90}. The case $\mathcal{M}^{-}_{\Spin(4)}$ can be defined and handled similarly using the isomorphism of $\Spin(4)$ with $\SL(2)\times \SL(2)$.
\end{remark}

\section{Fock space realization of level one modules}
In this section, following work of I. Frenkel \cite{Frenkel} and Kac-Petersen \cite{KP}, we first recall the explicit construction of level one modules of $\widehat{\mathfrak{so}}(2d+1)$ using infinite dimensional Clifford algebras. We also give explicit expressions for the space of invariants.  

\subsection{Clifford algebra} Let $W$ be a vector space (not necessarily finite dimensional) with a nondegenerate bilinear form $\{ \, ,\, \}$. Let $T(W)$ denote the tensor algebra of $W$, and  define the {\em Clifford algebra} $Cl(W)$ to be the quotient of $T(W)$ by the ideal generated by elements of the form $v\otimes w+ w\otimes v-\{v,w\}$. 
Let $W=W^{+}\oplus W^{-}\oplus \mathbb{C}\cdot e^0$ be a quasi-isotropic decomposition of $W$ such that $e^0$ is either orthonormal with respect to $W^\pm$,  or zero. 
Then the Clifford algebra $Cl(W)$ acts on $\bigwedge W^{-}$ by setting $w^{+}\cdot1=0$ for all $w^{+}\in W^{+}$ and letting $W^{-}$ act by wedge product on the left. If $e^0\neq0$ and $v\in \bigwedge^p W^{-}$, then we set 
$\sqrt{2}e^0\cdot v=(-1)^pv$.

\subsection{Level one modules} \label{sec:levelone}
Now suppose $W=W_d$ is $(2d+1)$-dimensional. We choose an ordered basis $\phi^1,\ldots \phi^r,\phi_0=\phi^0,\phi^{-r},\ldots, \phi^{-1}$ of $W_d$ such that $\{\phi^{a},\phi^{b}\}=\delta_{a+b,0}$. Define operators $E^{i}_j(\phi^k):=  \delta_{j,k}\phi^i$, and set $B^i_j:=  E^i_j-E^{-j}_{-i}$. It follows that elements of $\mathfrak{so}(2d+1)$ are of the form $B^{i}_j$ (cf.\ \cite{hasegawa, FultonHarris:91}).

For $h\in \{0,\frac{1}{2}\}$, let $W_d^{\mathbb{Z}+h}:=  W_d\otimes t^{h}\mathbb{C}[t,t^{-1}]$. We extend the bilinear form on $W_d$ to $W_d^{\mathbb{Z}+h}$ by setting 
$\{w_1({a_1}),w_2({a_2})\}:=  \{w_1,w_2\}\delta_{a_1+a_2,0}$, where $w(a)=w\otimes t^{a}$. As above, choose a quasi-isotropic decomposition:

$$W_d^{\mathbb{Z}+h}=\begin{cases}
W_d^{\mathbb{Z}+h,+}\oplus W_d^{\mathbb{Z}+h,-}\oplus \mathbb{C}\cdot e^0 & \ ,\ \text{if }h=1/2 \ ,\\
W_d^{\mathbb{Z}+h,+}\oplus W_d^{\mathbb{Z}+h,-} &\ ,\ \text{if }h=0\ ,
\end{cases}
$$
where $e^0=\phi_0(0)$. Similarly, $W_d^{\mathbb{Z}+h,\pm}$ is given by the following:
\begin{itemize}
\item If $h=0$, then $W_d^{\mathbb{Z}+h,\pm}:=  W_d\otimes t^{\pm 1}\mathbb{C}[t^{\pm 1}]\oplus W_d^{\pm}\otimes t^0$;
\item If $h=\frac{1}{2}$, then $W_d^{\mathbb{Z}+h,\pm}:=  W_d\otimes t^{\pm \frac{1}{2}}\mathbb{C}[t^{\pm 1}]$.
\end{itemize}
We define the \emph{normal ordering} $\normalorder{\hbox{}}$ for products in $Cl(W_d^{\mathbb{Z}+h})$ as follows: For any element $w_1(a_1)$ and $w_2(a_2)$ in $W_d^{\mathbb{Z}+h}$, we define

\begin{equation} \label{eqn:normalorder}
  \normalorder{w_1(a_1)w_2(a_2)} =
\begin{cases}
      -w_2(a_2)w_1(a_1) & \mbox{if $a_1>0>a_2$} \\      
      \frac{1}{2}(w_1(a_1)w_2(a_2)-w_2(a_2)w_1(a_1)) & \mbox{ if $a_1=a_2=0$}\\
      w_1(a_1)w_2(a_2) & \mbox{otherwise.} 
\end{cases} 
\end{equation}

For $X\in \mathfrak{so}(2d+1)$, we denote $X(m):=  X\otimes t^m$. Now for any $i$ and $j$, we can define an action of $B^{i}_j(m)$ on $\bigwedge W_d^{\mathbb{Z}+h,-}$ by the following formula:$B^{i}_j(m)w:=  \sum_{a+b=m}\normalorder{\phi^i(a)\phi_j(b)}\ w\ ,$
where 
 the action on $w$ is given by Clifford multiplication. Then we have the following important result.
\begin{proposition}[Frenkel \cite{Frenkel}, Kac-Petersen \cite{KP}]
The above action $B^i_j(m)$ gives an isomorphism of the following $\widehat{\mathfrak{so}}(2d+1)$-modules at level one.
\begin{itemize}
\item $\displaystyle\mathcal{H}_{\omega_0}(\mathfrak{so}(2d+1),1)\oplus \mathcal{H}_{\omega_1}(\mathfrak{so}(2d+1),1) \simeq \bigwedge W_d^{\mathbb{Z}+\frac{1}{2},-}$;
\item $\displaystyle\mathcal{H}_{\omega_d}(\mathfrak{so}(2d+1),1)\simeq \bigwedge W_d^{\mathbb{Z},-}$.
\end{itemize}

\end{proposition}

\subsection{Clifford multiplication and the invariant form}\label{clifford}
In this section, we give explicit expressions for conformal blocks in $\mathcal{V}^*_{\omega_0,\omega_d,\omega_d}(\mathbb{P}^1, \mathfrak{so}(2d+1),1)$ and $\mathcal{V}^*_{\omega_1,\omega_d,\omega_d}(\mathbb{P}^1, \mathfrak{so}(2d+1),1)$ in terms of Clifford multiplication. In the following, let $W_d$ be as in Section \ref{sec:levelone}.

\subsubsection{The case $\vec{\Lambda}=(\omega_0,\omega_d,\omega_d)$}\label{case0dd} 
From  representation theory it follows that the conformal block $\mathcal{V}^*_{\omega_0,\omega_d,\omega_d}(\mathbb{P}^1, \mathfrak{so}(2d+1),1)$ is a subspace of $\operatorname{Hom}_{\mathfrak{so}(2d+1)}(V_{\omega_d}\otimes V_{\omega_d},\mathbb{C})$. But since both  spaces are $1$-dimensional they are isomorphic. 
Since we know that ${V}_{\omega_d}$ is isomorphic as an $\mathfrak{so}(2d+1)$-module to $\bigwedge W_d^{-}$, then taking the opposite Borel we can express it as $\bigwedge W_d^{+}$. Hence, the invariant bilinear form (unique up to constants) is given by 
\[
   B(\phi_{i_1}\wedge \cdots\wedge  \phi_{i_p},\phi^{j_1}\wedge \cdots \wedge \phi^{j_{q}} ):=   
\begin{cases}
    \prod_{a=1}^p \{ \phi_{i_a},\phi^{j_a}\}& \text{if $p=q$}\\
    0           & \text{otherwise.}
\end{cases}
\]
 
\subsubsection{The case $\vec{\Lambda}=(\omega_1,\omega_d,\omega_d)$}\label{case1dd}
As in the previous case, we know that the conformal block $\mathcal{V}^*_{\omega_1,\omega_d,\omega_d}(\mathbb{P}^1, \mathfrak{so}(2d+1),1)$ is a subspace of $\operatorname{Hom}_{\mathfrak{so}(2d+1)}(V_{\omega_1}\otimes V_{\omega_d}\otimes V_{\omega_d},\mathbb{C})$. The $\mathfrak{so}(2d+1)$-module $V_{\omega_1}$ appears with multiplicity one (see Kempf-Ness \cite{KempfNess:88}) in the tensor product of module $V_{\omega_d}\otimes V_{\omega_d}$, thus we have 
$$\operatorname{Hom}_{\mathfrak{so}(2d+1)}(W_d\otimes \bigwedge W_d^{-} \otimes \bigwedge W_d^{+}, \mathbb{C})\simeq \mathbb{C}\ .$$ 
We will denote a nonzero element of the left hand side above as $\langle \widetilde{\Psi}\mid $. 
Consider the Clifford multiplication map from $m:W_d\otimes \wedge W_d^{-} \rightarrow \wedge W_d^{-}$. Then we define $$\langle{\widetilde{\Psi}}\mid  a\otimes v\otimes w^*\rangle:=  B( m(a\otimes v),w^*)\ .$$ 
 We will show that $\langle \widetilde{\Psi}\mid $ is a nonzero element of $\mathcal{V}_{\omega_1,\omega_d,\omega_d}(\mathbb{P}^1,\mathfrak{so}(2d+1),1)$ and hence it is unique up to constants. First we prove that $\langle \widetilde{\Psi}\mid $ is $\mathfrak{so}(2d+1)$-invariant. 

There is an isomorphism $\wedge^2 W_d\simeq\mathfrak{so}(2d+1)$. 
Any $X \in \mathfrak{so}(2d+1)$ may be regarded as an element of the Clifford algebra as follows: for $a,b\in W_d$, we get an element of the Clifford algebra $a\cdot b -\frac{1}{2}\{a,b\}$. First we show that the Clifford multiplication map $m$ defined above is $\mathfrak{so}(2d+1)$-invariant; i.e $X\cdot m(a,w)=m(X\cdot a,w)+m(a,X\cdot w)$, where $a\in W_d$ and $w\in \bigwedge W_d^{-}$. Without loss of generality assume that $X$ is of the form $a\cdot b-\frac{1}{2}\{a,b\}$. By a direct calculation we see that 
\begin{align*}
m((X\cdot a)\otimes w) + m(a\otimes (X\cdot w))
&=(((ab-\tfrac{1}{2}\{a,b\})v\big)\cdot w+ v\cdot \big((ab-\tfrac{1}{2}\{a,b\})\cdot w\big)\\
&=(a\cdot b-\tfrac{1}{2})v\cdot w=X\cdot (v\cdot w)\ .\\
\end{align*}
Thus the Clifford multiplication map $m$ is $\mathfrak{so}(2d+1)$-invariant.  By a direct calculation, we get the following: 
\begin{align*}
 \langle{\widetilde{\Psi}}\cdot X\mid a\otimes v \otimes w^*\rangle
&:=   \langle \widetilde{\Psi}\mid  Xa \otimes v \otimes w^*\rangle +\langle \widetilde{\Psi}\mid  a \otimes Xv \otimes w^*\rangle +\langle \widetilde{\Psi}\mid  a \otimes v \otimes Xw^*\rangle\\
&=B( Xm(a\otimes v) \otimes w^*) + B( m(a\otimes v) \otimes Xw^*) \\
& =B\cdot  X ( m(a\otimes v) \otimes w^* )\\
&=0 \qquad\qquad \mbox{(since $B$ is $\mathfrak{so}(2d+1)$ invariant).}
\end{align*}
This shows that $\langle \widetilde{\Psi}\mid $ is $\mathfrak{so}(2d+1)$-invariant.

 It will actually be more convenient to express $\langle{\widetilde{\Psi}} \mid $ in terms of an $\mathfrak{so}(2d+1)$-equivariant map 
$f: \bigwedge W_d^{-} \otimes \bigwedge W_d^{+} \rightarrow W_d^*$,
that will be  unique up to constants;  the relationship  is
$
\langle\widetilde \Psi \mid  a\otimes v\otimes w^\ast\rangle= f(v\otimes w^\ast)(a)
$.
We want to write  $f$ explicitly with respect to the given choice of basis. 
Let $\mathcal{I}_p=(1\leq i_1<\cdots < i_p\leq d )$  be a set of $p$ tuples of distinct ordered integers from the set $\{1,\ldots,d\}$ and similarly let $\mathcal{J}_q=((1\leq j_1<\cdots < j_q\leq d )$  be a set of $q$ tuples of distinct ordered integers. We are now ready to define the function $f$. This will be defined in several steps. First of all $f(v,w)=0$ if $v \in 
\wedge^{p} W_d^{-}$ and $w\in \wedge^{q}W_d^{+}$ and $|p-q|>1$.

\subsubsection{Case I, p=q}This is divided into the following subcases. If $\mathcal{I}_p\neq \mathcal{J}_p$, then we declare $f(\phi_{i_1}\wedge \cdots\wedge  \phi_{i_p},\phi^{j_1}\wedge \cdots \wedge \phi^{j_p}):=  0$. If $\mathcal{I}_p=\mathcal{J}_p$, then we define 
$f(\phi_{i_1}\wedge \cdots\wedge  \phi_{i_p},\phi^{i_1}\wedge \cdots \wedge \phi^{i_p})$ is up to a constant equal to $\{\phi_0,-\}$.

\subsubsection{Case II, q=p-1}Then $f(\phi_{i_1}\wedge \cdots\wedge  \phi_{i_p},\phi^{j_1}\wedge \cdots \wedge \phi^{j_{p-1}} )$ is up to a constant equal to
\[
\begin{cases}
    \{\phi_{i_k},-\}& \text{if } \mathcal{J}_{p-1}\cup \{i_k\}=\mathcal{I}_p\\
    0,              & \text{otherwise.}
\end{cases}
\]

\subsubsection{Case III, q=p+1}Then $f(\phi_{i_1}\wedge \cdots\wedge  \phi_{i_p},\phi^{j_1}\wedge \cdots \wedge \phi^{j_{p+1}} )$ is up to a constant equal to
\[
 \begin{cases}
    \{\phi^{j_k},-\}& \text{if } \mathcal{I}_{p}\cup \{j_k\}=\mathcal{J}_{p+1}\\
    0,              & \text{otherwise.}
\end{cases}
\]
This shows that $f$, and hence also $\langle \widetilde{\Psi}\mid $, is nonzero and $\mathfrak{so}(2d+1)$-invariant.

\section{Highest weight vectors for branching of basic modules}
We  give an explicit description of highest weight vectors in the branching rule in ``Kac-Moody'' form i.e.\ as product of operators from the affine Lie algebra acting on the level one representations. Our guideline is the paper of Hasegawa \cite{hasegawa}.

\subsection{Tensor products}
Let $W_s$ be a $(2s+1)$-dimensional $\CBbb$-vector space  with a nondegenerate bilinear form $\{,\}$, and let $\{e_{p}\}_{p=-s}^s$ be an orthonormal basis of $W_s$. Let $\phi^1,\ldots, \phi^s,\phi^0,\phi^{-s},\ldots,\phi^{-1}$ be an ordered quasi-isotropic basis of $W_s$. The tensor product of $W_d=W_r\otimes W_s$ carries a nondegenerate symmetric bilinear form $\{,\}$ given by the product of the forms on $W_r$ and $W_s$. Clearly the elements $\{e_{j,p}:=  e_j\otimes e_p | -r\leq j\leq r \ \mbox{and}\  -s\leq p\leq s \}$ form an orthonormal basis of $W_d$. By $(j,p)>0$, we mean $j>0$ or $j=0, p>0$. Set 
$$\phi^{j,p}=\frac{1}{\sqrt{2}}(e_{j,p}-\sqrt{-1}e_{-j,-p})\ , \hspace{1cm} \phi^{-j,-p}=\frac{1}{\sqrt{2}}(e_{j,p}+\sqrt{-1}e_{-j,-p})\ ,$$ 
 for $(j,p)>0$. The form $\{,\}$ on $W_d$ is given by the formula 
$\{ \phi^{j,p},\phi^{-k,-q}\}=\delta_{j,k}\delta_{p,q}$, for $-r\leq j,k \leq r$,  $-s\leq p,q \leq s$. Let as before $W_d^{\pm}=\bigoplus_{(j,p)>0}\mathbb{C}\cdot\phi^{\pm j,\pm p}$ and $\phi^{0,0}=e_{0,0}$. The quasi-isotropic decomposition of $W_d$ is given by
$W_d=W_d^{+}\oplus W_d^{-}\oplus \mathbb{C}\cdot\phi^{0,0}$.

Define the operators $E^{j,p}_{k,q}$ by the formula $E^{j,p}_{k,q}(\phi^{i,l})=\delta_{i,k}\delta_{l,q}\phi^{j,p}$. We get a matrix in $\mathfrak{so}(2d+1)$ by the formula 
$B^{j,p}_{k,q}=E^{j,p}_{k,q}-E^{-k,-q}_{-j,-p}$. Clearly the Cartan subalgebra $\mathfrak{h}$ of $\mathfrak{so}(2d+1)$ is generated by the diagonal matrices $B^{j,p}_{j,p}$ for $(j,p)>0$. The dual basis is denoted by $L_{j,p}$. Hence 
$\mathfrak{h}^*=\oplus_{(j,p)>0}\mathbb{C}\cdot L_{j,p}$.

The tensor product $W_d=W_r\otimes W_s$ gives the embedding \eqref{eqn:so-embedding}.
If $B^{i}_j$ be an element of $\mathfrak{so}(2r+1)$, then the action of $B^{i}_j(m)$ on $W_d^{\ZBbb+h}$ is given by 
$$L(B^{i}_j(m))= \sum_{q=-s}^s \sum_{a+b=m} \normalorder{\phi^{i,q}(a)\phi_{j,q}(b)}\ .$$
Similarly for $B^{i}_j$ is an element of $\mathfrak{so}(2s+1)$, then the action of $B^{i}_j(m)$ is given by 
$$R(B^{i}_j(m))= \sum_{p=-r}^r \sum_{a+b=m} \normalorder{\phi^{p,i}(a)\phi_{p,j}(b)}\ .$$

\subsection{Notation for weights}The Cartan algebra $\mathfrak{h}$ of $\mathfrak{so}(2r+1)$ is generated by elements of the form $B^i_i$ for $1\leq i \leq r$. Let $L_i$ denote the dual of $B^i_i$. The fundamental weights of $\mathfrak{so}(2r+1)$ are given by $\omega_i=\sum_{k=1}^i L_i$ for $1\leq i \leq r-1$ and $\omega_r=\frac{1}{2}(L_1+\cdots+L_r)$. 
Let us denote by $\mathcal{Y}_{r}$ the set of Young diagrams with at most $r$ rows. Any integral dominant weight $\lambda$ of $\mathfrak{so}(2r+1)$ is of the form 
$\lambda=\sum_{i=1}a_i\omega_i$, $a_i\geq 0$ for all $i$.
\begin{enumerate}
\item If $a_r$ is even, then the representation $\lambda$ induces a representation of the group $\SO(2r+1)$. By using the expression of $\omega_i$ in terms of $L_i$'s, we get  $\lambda=\sum_{i}b_iL_i$, and $b_1\geq \cdots \geq b_r$, give a Young  diagram in $\mathcal{Y}_r$.

\item If $a_r$ is odd, then we can rewrite $\lambda=\lambda'+\omega_r$. Then the coefficient of $\omega_r$ in $\lambda'$ is even and by repeating the same process for $\lambda'$, we can write $\lambda=Y+\omega_r$, where $Y$ is an element of $\mathcal{Y}_{r}$.  
\end{enumerate}

The group of affine Dynkin-diagram automorphisms $\mathbb{Z}/2$ acts on the set of level $2s+1$ weights $P_{2s+1}(\mathfrak{so}(2r+1))$ by interchanging the affine fundamental weight $\omega_0$ with $\omega_1$. Let $\lambda=\sum_{i=1}^ra_i\omega_i$ and $\sigma$ be the generator of $\mathbb{Z}/2$, then 

\begin{equation}\label{daiso}
\sigma(\lambda):=(2s+1-(a_1+2(a_2+\dots+a_{r-1})+a_r))\omega_1+a_2\omega_2+\dots+a_r\omega_r\ .
\end{equation}
Let $\mathcal{Y}_{r,s}$ denote the set of Young diagrams with at most $r$ rows and $s$ columns. Then
the orbits of the group action with the cardinality are given below \cite{OxburyWilson:96}:
\begin{itemize}
\item $Y\in \mathcal{Y}_{r,s}$ and the orbit length is $2$. 
\item $Y+\omega_r$, where $Y\in \mathcal{Y}_{r,{s-1}}$ and the orbit length is $2$.
\item $Y+\omega_r$, where $Y\in \mathcal{Y}_{r,s}\backslash \mathcal{Y}_{r,s-1}$ and the orbit length is $1$.
\end{itemize}
For a Young diagram $Y$, we denote by $Y^T\in \mathcal{Y}_{s,r}$, the diagram obtained by interchanging the rows and columns of $Y$,   by $Y^c\in Y_{r,s}$  the complement of $Y$ in a box of size $r\times s$, and by $Y^*\in \mathcal{Y}_{s,r}$ the Young diagram $(Y^{T})^c$ obtained by first taking the transpose and then taking the complement in  a box of size $(s\times r)$.

\subsection{Branching rules}\label{branching}
For reference, we
  state here some of the components that appear in the branching rule  for the  embedding \eqref{eqn:so-embedding} (recall Section \ref{sec:conformal-embedding}). Let $\sigma$ be as in the previous section.
   Let $\lambda$, $\mu$, and $\Lambda$ be integrable highest weights for 
$\sofrak(2r+1)$ at level $2s+1$, $\sofrak(2s+1)$ at level $2r+1$, and $\sofrak(2d+1)$ at level $1$, respectively. 
We say that $(\lambda,\mu) \in B(\Lambda)$ if 
$\mathcal{H}_{\lambda}(\sofrak(2r+1))\otimes \mathcal{H}_{\mu}(\sofrak(2s+1))$
 appears in the branching of $\mathcal{H}_{\Lambda}(\sofrak(2d+1))$. 
Note that here (and for the rest of the paper) unless specified otherwise the levels $2s+1, 2r+1$, and $1$,  have been (wiil be) suppressed from the notation of highest weight modules.
Then the branching rules we need are the following:
\begin{itemize}
\item $(Y,Y^T)\in B(\omega_0)$ if $|Y|$ is even;
\item $(\sigma(Y),Y^T)$ and $(Y,\sigma(Y^T))\in B(\omega_0)$ if $|Y|$ is odd;
\item $(Y,Y^T)\in B(\omega_1)$ if $|Y|$ is odd;
\item $(\sigma(Y),Y^T)$ and $(Y,\sigma(Y^T))\in B(\omega_1)$ if $|Y|$ is even;
\item for $Y\in \Ycal_{r,s-1}$, $(Y+\omega_r,Y^\ast+\omega_s)$ and  $(\sigma(Y+\omega_r),Y^\ast+\omega_s)$ are in  $B(\omega_d)$;
\item for $Y\in \Ycal_{r,s}\setminus\Ycal_{r,s-1}$, $(Y+\omega_r,Y^\ast+\omega_s)$ and $(Y+\omega_r,\sigma(Y^\ast+\omega_s))$ are in  $B(\omega_d)$.
\end{itemize}
 We  refer the reader to \cite{hasegawa} for a proof. 

\subsection{Highest weight vectors of branching } \label{sec:highest-weight}
An explicit description of the highest weight vectors for the components of the branching can be found in \cite{hasegawa}. In this section,  in those cases that  will be convenient for our applications,
we express them  as   products of operators in $\widehat{\mathfrak{so}}(2d+1)$ acting on the level one representations of $\widehat{\mathfrak{so}}(2d+1)$.
  Recall the following  from \cite{Mukhopadhyay:12}.

\begin{proposition}\label{kacmoody1}
Let $\lambda' \in \mathcal{Y}_{r,s}$ be obtained from $\lambda$ by removing two boxes with coordinates $(a,b)$ and $(c,d)$ and $\epsilon \in \{0,1\}$. Assume that $(a,b)<(c,d)$ under the lexicographic ordering. If $v_{\lambda'} \in \operatorname{End}(\mathcal{H}_{\omega_{\epsilon}}(\mathfrak{so}(2d+1)))$ is the highest weight vector of the component $\mathcal{H}_{\lambda'}\otimes\mathcal{H}_{\lambda'^T}$, then the highest weight vector $v_{\lambda}$ of the component $\mathcal{H}_{\lambda}(\sofrak(2r+1))\otimes \mathcal{H}_{\lambda^T}(\sofrak(2s+1))$ is given by
$v_{\lambda}=B^{a,b}_{-c,-d}(-1)v_{\lambda'}$.

\end{proposition}

From \cite{hasegawa} we have:
\begin{proposition}\label{kacmoody2}
The element $\bigwedge_{j=-r}^r \phi^{j,1}(-\frac{1}{2})\in\bigwedge W_d^{\mathbb{Z}+\frac{1}{2},-}$ gives the highest weight vector for the component $\mathcal{H}_{\omega_0}(\sofrak(2r+1))\otimes \mathcal{H}_{(2r+1)\omega_1}(\sofrak(2s+1))$. 
\end{proposition}

Next, we describe the highest weight vectors for the branching of the Spin module at level one. First, we need some notation. 
Given  $Y\in\mathcal{Y}_{r,s}$, we view it pictorially  as an $r\times s$ box  with the white boxes carving out the Young diagram (see below). We associate to $Y$
another diagram as follows:
\[
\widetilde{Y}_{j,p}=\left\{
\begin{array}{l l}
\blacksquare &\quad \text{if $Y$ has an empty box in the $(j,p+s+1)$-th position,} \\
\square &\quad \text{otherwise.}\\
\end{array}\right.
\]
Here in the matrix $\widetilde Y$, 
$j=0,1,\ldots, r$, $p=1,\ldots, s,-s, \ldots, -1$.
%
This is illustrated as:

\begin{equation*}
  Y=\Scale[.8]{
  \tikz[baseline=(M.west)]{%
    \node[matrix of math nodes,matrix anchor=west,ampersand replacement=\&] (M) {%
      \square \& \square \& \blacksquare \& \blacksquare  \\
      \square \& \square \& \blacksquare \& \blacksquare   \\
       \square \& \square \& \blacksquare \& \blacksquare  \\
       \square \&  \blacksquare \&   \blacksquare  \&\blacksquare \\
      };
    \node[draw,fit=(M-1-1)(M-4-4),thick,  inner sep=-1pt] {};
  }}
\qquad\longleftrightarrow\qquad
  \widetilde Y_{j,p}=
  \Scale[.8]{
  \tikz[baseline=(M.west)]{%
    \node[matrix of math nodes,matrix anchor=west,ampersand replacement=\&] (M) {%
   j\backslash p \&   1 \& .. \& .. \& s \&-s\& ..\& ..\& -1 \\
	-1  \&  \square \& .. \& .. \& \square\& \square\& .. \& .. \& \square \\
  0  \&  \square \& .. \& .. \& \square\& \square\& .. \& .. \& \square \\
   1  \&  \square \& .. \& .. \& \square \& \square \& \square \& \blacksquare \&\blacksquare \\
  : \& :   \&  \& \& :  \&    \square  \& \square \& \blacksquare\&\blacksquare  \\
  : \& :   \&  \& \& :  \&    \square  \& \square \& \blacksquare \& \blacksquare  \\
    r \&\square  \& ..   \& ..  \&\square  \&\square\& \blacksquare \&  \blacksquare \&  \blacksquare \\
    };
     \node[draw,fit=(M-2-2)(M-7-9),thick,  inner sep=0pt] {};
    \node[draw,fit=(M-7-6)(M-4-9),thick, dashed, inner sep=-2.5pt] {};
  }}
\end{equation*}
With this notation, we can state the branching rules.
If $Y \in \mathcal{Y}_{r,s}$, then $\mathcal{H}_{{Y}+\omega_r}(\sofrak(2r+1))\otimes \mathcal{H}_{Y^*+\omega_s}(\sofrak(2s+1))$ appears in the decomposition of $\mathcal{H}_{\omega_d}(\sofrak(2d+1))$. For a proof of the following, we refer the reader to \cite{hasegawa}.
\begin{proposition}\label{kacmoody3}
 The highest weight vector $v_Y$ of the component $\mathcal{H}_{{Y}+\omega_r}(\sofrak(2r+1))\otimes \mathcal{H}_{Y^*+\omega_s}(\sofrak(2s+1))$ is given by $\bigwedge_{\widetilde{Y}_{j,p}=\,\blacksquare}\phi_{j,p}$.

\end{proposition}
From Section \ref{branching}, if
$Y\in \mathcal{Y}_{r,s-1}$,  the component $\mathcal{H}_{\sigma(Y+\omega_r)}(\sofrak(2r+1))\otimes \mathcal{H}_{Y^*+\omega_s}(\sofrak(2s+1))$ appears in the decomposition of $\mathcal{H}_{\omega_d}(\sofrak(2d+1))$. 
 We  describe the highest weight vectors. We define a new diagram $\sigma(\widetilde{Y})$, obtained by  first considering $\widetilde{Y}$ and then interchanging the black boxes in the $1$-st row by the corresponding white boxes in the $(-1)$-st row, and keeping the columns invariant.
\begin{equation*}
  Y=\Scale[.8]{
  \tikz[baseline=(M.west)]{%
    \node[matrix of math nodes,matrix anchor=west,ampersand replacement=\&] (M) {%
      \square \& \square \& \blacksquare \& \blacksquare  \\
      \square \& \square \& \blacksquare \& \blacksquare   \\
       \square \& \square \& \blacksquare \& \blacksquare  \\
       \square \&  \blacksquare \&   \blacksquare  \&\blacksquare \\
      };
    \node[draw,fit=(M-1-1)(M-4-4),thick,  inner sep=-1pt] {};
  }}
\qquad\longleftrightarrow\qquad
 \sigma( \widetilde Y)_{j,p}=
 \Scale[.8]{
  \tikz[baseline=(M.west)]{%
    \node[matrix of math nodes,matrix anchor=west,ampersand replacement=\&] (M) {%
   j\backslash p \&   1 \& .. \& .. \& s \&-s\& ..\& ..\& -1 \\
	-1  \&  \square \& .. \& .. \& \square\& \square\& .. \& \blacksquare \& \blacksquare \\
  0  \&  \square \& .. \& .. \& \square\& \square\& .. \& .. \& \square \\
   1  \&  \square \& .. \& .. \& \square \& \square \& \square \& \square \&\square \\
  : \& :   \&  \& \& :  \&    \square  \& \square \& \blacksquare\&\blacksquare  \\
  : \& :   \&  \& \& :  \&    \square  \& \square \& \blacksquare \& \blacksquare  \\
    r \&\square  \& ..   \& ..  \&\square  \&\square\& \blacksquare \&  \blacksquare \&  \blacksquare \\
    };
     \node[draw,fit=(M-2-2)(M-7-9),thick,  inner sep=0pt] {};
    \node[draw,fit=(M-7-6)(M-4-9),thick, dashed, inner sep=-2.5pt] {};
  }}
\end{equation*}
With the above notation, we have the following from \cite{hasegawa}.
\begin{proposition}\label{sigmatwist}
The highest weight vector  of the component $\mathcal{H}_{\sigma({Y}+\omega_r)}(\sofrak(2r+1))\otimes \mathcal{H}_{Y^*+\omega_s}(\sofrak(2s+1))$ is given by $\bigwedge_{\sigma(\widetilde{Y})_{j,p}=\,\blacksquare}\phi_{j,p}(\epsilon)$, where $\epsilon=-1$ if $j=-1$, and $\epsilon=0$  otherwise.  
\end{proposition}
We can rewrite the result above  in the ``Kac-Moody" form. 
\begin{corollary}\label{kacmoody4}
Let $Y'$ be the Young diagram obtained from  $Y$  by changing the black boxes in the first row to white. Let $v_{Y'}$ be the highest weight vector of the component $\mathcal{H}_{Y'+\omega_r}(\sofrak(2r+1))\otimes \mathcal{H}_{(Y')^*+\omega_s}(\sofrak(2s+1))$. Then a highest weight vector of the component  $\mathcal{H}_{\sigma({Y}+\omega_r)}(\sofrak(2r+1))\otimes \mathcal{H}_{Y^*+\omega_s}(\sofrak(2s+1))$ is given by:
$\displaystyle\prod_{\sigma(\widetilde{Y})_{-1,p}=\,\blacksquare}B^{1,-p}_{0,0}(-1)v_{Y'}$.
\end{corollary}
\begin{proof}
Use Proposition \ref{sigmatwist} and the definition of the action of $B^{1,-p}_{0,0}(-1)$.
\end{proof}

\section{Rank-level duality in genus zero}

\subsection{General context of rank-level duality}\label{ranklevelinjectivity}  Let $\phi:\frg_1\oplus \frg_2\rightarrow \frg$ be a conformal embedding with Dynkin multi-index
 $d_\phi=(\ell_1,\ell_2)$  (see Section \ref{sec:conformal-embedding}).
Then $\phi$ extends  to a homomorphism of affine Lie algebras $\widehat{\phi}:\widehat{\frg}_1\oplus \widehat{\frg}_2 \rightarrow \widehat{\frg}$. Let $\vec{\lambda}=(\lambda_1,\ldots, \lambda_n)$ (resp.\  $\vec{\mu}=(\mu_1,\ldots, \mu_n))$ and $\vec{\Lambda}=(\Lambda_1,\ldots, \Lambda_n)$ be $n$-tuples of level $\ell_1$ (resp.\  $\ell_2$) and level one integrable highest weights such that for each $1\leq i \leq n$, $(\lambda_i, \mu_i)\in B(\Lambda_i)$.
Taking the $n$-fold tensor product, we get a map: 
$\bigotimes_{i=1}^n \mathcal{H}_{\lambda_i}(\gfrak_1)\otimes \mathcal{H}_{\mu_i}(\gfrak_2) \rightarrow \bigotimes_{i=1}^n\mathcal{H}_{\Lambda_i}(\gfrak)$. Let $\mathfrak{X}$ be the data associated to curve $C$ with $n$ marked points and chosen coordinates. Taking coinvariants, we get  the following map of dual conformal blocks:$\alpha:\mathcal{V}_{\vec{\lambda}}(\mathfrak{X},\frg_1,\ell_1)\otimes \mathcal{V}_{\vec{\mu}}(\mathfrak{X},\frg_2,\ell_2)\rightarrow \mathcal{V}_{\vec{\Lambda}}(\mathfrak{X},\frg,1)\ .$
We call a triple $(\vec{\lambda},\vec{\mu},\vec{\Lambda})\in P^n_{\ell_1}(\frg_1)\times P^n_{\ell_2}(\frg_2) \times P^n_{1}(\frg) $ {\em admissible}, if they are connected by a map as above by the branching of level one modules.  If $\mathcal{V}_{\vec{\Lambda}}(\mathfrak{X},\frg,1) $ is one dimensional we get a map:
$\alpha^*:\mathcal{V}_{\vec{\lambda}}(\mathfrak{X},\frg_1,\ell_1)\rightarrow \mathcal{V}^*_{\vec{\mu}}(\mathfrak{X},\frg_2,\ell_2)$, which is determined up to a nonzero multiplicative constant.
 Then $\alpha^*$ is known as the {\em rank-level duality map}.  We say that \emph{rank-level duality holds} if $\alpha^\ast$ is an isomorphism.

Let $\mathcal{F}=(\pi:\mathcal{C}\rightarrow B; \sigma_1,\ldots, \sigma_n; \xi_1,\ldots, \xi_n)$ be a family of nodal curves on a base $B$ with sections $\sigma_i$ and local coordinates $\xi_i$. The map $\alpha^\ast$ can be extended to a map of sheaves 
$
\alpha(\mathcal{F}): \mathcal{V}_{\vec{\lambda}}(\mathcal{F}, \frg_1, \ell_1 )\otimes \mathcal{V}_{\vec{\mu}}(\mathcal{F},\frg_2, \ell_2 )\rightarrow \mathcal{V}_{\vec{\Lambda}}(\mathcal{F}, \frg, 1)\ .
$
Furthermore, if the embedding is $\frg_1\oplus \frg_2 \rightarrow \frg$ is conformal \cite{KacWakimoto:88}, as in the case of the odd orthogonal groups considered here,
 it follows that the rank-level duality map is flat with respect to the TUY connection. We refer the reader to \cite{Belkale:09} for a proof.

\subsection{Failure of rank-level duality over $\mathbb{P}^1$ with spin weights} \label{failure}
Rank-level duality isomorphisms for odd orthogonal groups on $\PBbb^1$ with $\SO$-weights was proved in \cite{Mukhopadhyay:12}. In this section, we give explicit examples where the rank-level duality map over $\mathbb{P}^1$ with four marked points is well-defined, but fails to be an isomorphism. 

\subsubsection{Example 1}
Consider the embedding  $\mathfrak{so}(5)\oplus \mathfrak{so}(7)\rightarrow \mathfrak{so}(35)$. From the branching rules in Section \ref{branching}, we know that $\mathcal{H}_{2\omega_1+\omega_2}(\sofrak(5))\otimes \mathcal{H}_{\omega_1+3\omega_3}(\sofrak(7))$ appears in the branching of the spin module $\mathcal{H}_{\omega_{17}}(\sofrak(35))$. By functoriality we get the following map of conformal blocks:
$\mathcal{V}_{\vec{\lambda}}(\mathbb{P}^1,\mathfrak{so}(5),7)\otimes \mathcal{V}_{\vec{\mu}}(\mathbb{P}^1,\mathfrak{so}(7),5)\rightarrow \mathcal{V}_{\omega_{17},\omega_{17},\omega_1,\omega_1}(\mathbb{P}^1,\mathfrak{so}(35),1)$, where $\vec{\lambda}=({2\omega_1+\omega_2, 2\omega_1+\omega_2,\omega_1,\omega_1})$ and $\vec{\mu}=({\omega_1+3\omega_3, \omega_1+3\omega_3,\omega_1,\omega_1})$. One checks (e.g.\ by \cite{Swinarski:10}) that  $\dim_\CBbb\mathcal{V}_{\omega_{17},\omega_{17},\omega_1,\omega_1}(\mathbb{P}^1,\mathfrak{so}(35),1)=1$. Hence, we get a rank-level duality map between $\mathcal{V}_{\vec{\lambda}}(\mathbb{P}^1,\mathfrak{so}(5),7)^*$ and $\mathcal{V}_{\vec{\mu}}(\mathbb{P}^1,\mathfrak{so}(7),5)$. But this map cannot be an isomorphism since  $\dim_\CBbb\mathcal{V}_{\vec{\mu}}(\mathbb{P}^1,\mathfrak{so}(7),5)=5$, whereas  $\dim_\CBbb\mathcal{V}_{\vec{\lambda}}(\mathbb{P}^1,\mathfrak{so}(5),7)=4$.



\subsubsection{Example 2}
Consider the embedding  $\mathfrak{so}(7)\oplus \mathfrak{so}(9)\rightarrow \mathfrak{so}(63)$. Then  
$$\dim_\CBbb \mathcal{V}_{\omega_{31},\omega_{31},\omega_1}(\mathbb{P}^1,\mathfrak{so}(63),1)=1\ ,$$  and following the branching rules in Section \ref{branching}, there is a well-defined rank-level duality map
$\mathcal{V}_{\vec{\lambda}}(\mathbb{P}^1,\mathfrak{so}(7),9)\otimes \mathcal{V}_{\vec{\mu}}(\mathbb{P}^1,\mathfrak{so}(9),7)\rightarrow \mathcal{V}_{\omega_{31},\omega_{31},\omega_1}(\mathbb{P}^1,\mathfrak{so}(63),1)$, where 
$\vec{\lambda}=(\omega_2+3\omega_3,\omega_2+3\omega_3, \omega_1+\omega_2)$, $\vec{\mu}=(\omega_1+2\omega_3+\omega_4,\omega_1+2\omega_3+\omega_4, \omega_1+\omega_2)$. But this map is not an isomorphism, since the dimensions of $\mathcal{V}_{\vec{\lambda}}(\mathbb{P}^1,\mathfrak{so}(7),9)$ and $\mathcal{V}_{\vec{\mu}}(\mathbb{P}^1,\mathfrak{so}(9),7)$ are $3$ and $4$,  respectively.

\subsubsection{Example 3}
Consider the embedding  $\mathfrak{so}(9)\oplus \mathfrak{so}(7)\rightarrow \mathfrak{so}(63)$. Then  
$$\dim_\CBbb \mathcal{V}_{\omega_{31},\omega_{31},\omega_1}(\mathbb{P}^1,\mathfrak{so}(63),1)=1\ ,$$
and following the branching rules in Section \ref{branching}, there is a well-defined rank-level duality map
$\mathcal{V}_{\vec{\lambda}}(\mathbb{P}^1,\mathfrak{so}(9),7)\otimes \mathcal{V}_{\vec{\mu}}(\mathbb{P}^1,\mathfrak{so}(7),9)\rightarrow \mathcal{V}_{\omega_{31},\omega_{31},\omega_1}(\mathbb{P}^1,\mathfrak{so}(63),1)$, where 
$\vec{\lambda}=(\omega_2+3\omega_4,\omega_2+3\omega_4, 2\omega_1+2\omega_4)$, $\vec{\mu}=(2\omega_1+2\omega_2+\omega_3,2\omega_1+2\omega_2+\omega_3, 3\omega_1+2\omega_3)$. But this map is not an isomorphism, since the dimensions of $\mathcal{V}_{\vec{\lambda}}(\mathbb{P}^1,\mathfrak{so}(9),7)$ and $\mathcal{V}_{\vec{\mu}}(\mathbb{P}^1,\mathfrak{so}(7),9)$ are $8$ and $14$,  respectively. 

\begin{remark}The above examples show that even if the rank-level duality map is well-defined it may not be an isomorphism due to the inequality of  dimensions of the source and the target spaces. However, we can still ask if  $\alpha^*$ is  injective?  The next section gives a positive answer to that question.   This is distinct from the issue in higher genus, where we will see that there is an equality of dimensions but still the rank-level duality map is not an isomorphism (see Section \ref{sec:higher-genus}).
\end{remark}

\subsection{Rank-level duality for 3-pointed $\mathbb{P}^1$ with spin weights}
 Consider the embedding \eqref{eqn:so-embedding}.
 The only interesting cases for 3-points with spin weights are the tuples $\vec{\Lambda}=(\omega_0,\omega_d,\omega_d)$ and $\vec{\Lambda}=(\omega_1,\omega_d,\omega_d)$. We also observe that the action of the automorphism of the affine Dynkin diagram fixes $\omega_d$ and interchanges  $\omega_0$ with $\omega_1$. We first fix some notation.
Let $Y_1\in \mathcal{Y}_{r,s}$ and $Y_2, Y_3  \in \mathcal{Y}_{r,s-1}$ and we consider $\vec{\lambda}=(Y_1, Y_2+\omega_r,Y_3+\omega_r)$. Let $\vec{\mu}=(Y_1^T,Y_2^*+\omega_s, Y_3^*+\omega_s)$ and $\vec{\Lambda}=(\omega_{\epsilon},\omega_d,\omega_d)$, where $\epsilon$ is zero or one depending on the even or odd parity of the number of boxes of the Young diagram of $Y_1$. From the branching rules of Section \ref{branching}, we get the following map of conformal blocks
\begin{equation}\label{injectivity}
\mathcal{V}_{\vec{\lambda}}(\mathfrak{X},\mathfrak{so}(2r+1),2s+1))\rightarrow \mathcal{V}^{*}_{\vec{\mu}}(\mathfrak{X},\mathfrak{so}(2s+1),2r+1)\otimes \mathcal{V}_{\vec{\Lambda}}(\mathfrak{X},\mathfrak{so}(2d+1),1)\ ,
\end{equation}
Here, $\mathfrak{X}$ denotes the data associated to $\mathbb{P}^1$ with three marked points and chosen coordinates.
The following is the main statement of this section.
\begin{theorem}\label{braidreducibility}
The rank-level duality map defined  in \eqref{injectivity} is injective.
\end{theorem}
The proof of this theorem is broken up into several steps and can be reduced to the case when both $Y_2$ and $Y_3$ are empty and $Y_1$ is just a Young diagram with one column, in which case the corresponding conformal blocks are one dimensional for $\mathfrak{so}(2r+1)$. We now describe the steps in the reduction.

\subsection{Reduction to the one dimensional case}The main tools used here are factorization/sewing of conformal blocks  (cf.\ Sections \ref{sec:conformal-block} and \ref{factorizationlemma}), and the fact that certain Littlewood Richardson coefficients are one.

\subsubsection{Step I}
Clearly, we may assume that the rank of the conformal block in the source of \eqref{injectivity} is nonzero. 
Now  consider a new tuple $\vec{\lambda}'=(\vec{\nu}_2,\omega_r,Y_1,\vec{\nu}_3,\omega_r)$, where $\vec{\nu}_2$ (resp.\  $\vec{\nu}_3$) is a tuple of $\omega_1$ of cardinality $|Y_2|$ (resp.\  $|Y_3|$). Similarly let $\vec{\mu}'=(\vec{\nu}_2^T,(2r+1)\omega_s,Y_1^T,\vec{\nu}_3^T,(2r+1)\omega_s)$ and $\vec{\Lambda}'=(\vec{\omega}_1, \omega_d,{\omega}_{{\epsilon}},\vec{\omega}_1,\omega_d)$, where $\omega_{{\epsilon}}$ is  $\omega_1$ or $\omega_0$'s, depending on the number of boxes of $Y_1$. It is easy to see that the triple $(\vec{\lambda}',\vec{\mu}',\vec{\Lambda}')$ is admissible.

\subsubsection{Step II} Let $\mathfrak{X}$ denote the data associated to $\mathbb{P}^1$ with $|Y_2|+|Y_3|+3$ marked points with chosen coordinates. The rank of the conformal block $\mathcal{V}_{\vec{\Lambda}'}(\mathfrak{X},\mathfrak{so}(2d+1),1)$ is one and the rank of the conformal block $\mathcal{V}_{\vec{\lambda}'}(\mathfrak{X},\mathfrak{so}(2r+1),2s+1)$ is nonzero. The first assertion can be easily checked via factorization (cf.\ Section \ref{factorizationlemma}) since the only nontrivial three point cases with spin weights up to permutation are $(\omega_0,\omega_d,\omega_d)$ and $(\omega_1,\omega_d,\omega_d)$ both of which are rank one.
For the second assertion, we get by the factorization theorem that the dimension of the conformal block $\mathcal{V}_{\vec{\lambda}'}(\mathfrak{X},\mathfrak{so}(2r+1),2s+1)$ is greater than equal to the dimension of the following product of conformal blocks:
\begin{eqnarray*}
&&  \mathcal{V}_{\vec{\nu}_2,\omega_r,Y_2+\omega_r}(\mathfrak{X}_1,\mathfrak{so}(2r+1),2s+1)\otimes \mathcal{V}_{Y_1,Y_2+\omega_r,Y_3+\omega_r}(\mathfrak{X}_2,\mathfrak{so}(2r+1),2s+1)\\
  && \hspace{2cm}\otimes \mathcal{V}_{\vec{\nu}_3,\omega_r,Y_3+\omega_r}(\mathfrak{X}_3,\mathfrak{so}(2r+1),2s+1) \ .
	\end{eqnarray*}
	Here, $\mathfrak{X}_1$ (resp.\  $\mathfrak{X}_3$) denote the data associated to a $\mathbb{P}^1$ with $|Y_2|+2$ (resp.\  $|Y_3|+ 2$) marked points and $\mathfrak{X}_2$ denote the data associated to a $\mathbb{P}^1$ with three marked points and chosen coordinates. 
The nonvanishing  of the dimensions on the first and third factors in the above expression follows from  Proposition \ref{basiccaseLRS}.

\subsubsection{Step III}Assume that the injectivity of the rank-level duality map for the admissible pairs $(\vec{\lambda}',\vec{\mu}',\vec{\Lambda}')$ holds, then Theorem \ref{braidreducibility} holds, where $\vec{\lambda}'$, $\vec{\mu}'$ and ${\vec{\Lambda}'}$ be as in Step II. The basic idea is that we split up the rank-level duality map into a direct sum of several rank-level duality maps. Now the injectivity of the rank-level duality map for the bigger space implies the injectivity of the rank-level duality map for the components, and vice-versa. 
 The key geometric input is Lemma \ref{geometricinput} in Section \ref{factorizationlemma}. The conditions in Lemma \ref{geometricinput} are guaranteed by the fact that the dimensions of the two conformal blocks on $\mathbb{P}^1$ with weights $\vec{\nu}_2,\omega_r,\sigma(Y_2+\omega_r)$ and $\vec{\nu}_3,\omega_r,\sigma(Y_3+\omega_r)$ are zero (cf.\ Proposition \ref{basiccaseLRS}). This is where we use that $Y_2,Y_3\in\mathcal{Y}_{r,s-1}$. 

\subsubsection{Step IV}By the previous discussion, it enough to prove the injectivity of the following rank-level duality map:
$$\mathcal{V}_{\vec{\lambda}'}(\mathfrak{X},\mathfrak{so}(2r+1),2s+1)\longrightarrow \mathcal{V}_{\vec{\mu}'}^{*}(\mathfrak{X},\mathfrak{so}(2s+1),2r+1)\otimes\mathcal{V}_{\vec{\Lambda}'}(\mathfrak{X},\mathfrak{so}(2d+1),1)\ .$$
We now consider a degeneration of $\mathbb{P}^1$ into nodal curve $C=C_1\cup C_2$, where $C_1$ is a copy of $\mathbb{P}^1$ with two smooth marked points and the weights $\omega_r$, $\omega_r$ are the weights attached to the markings. The other component $C_2$ is $\mathbb{P}^1$ with rest of the marked points. The components $C_1$ and $C_2$ meet at a point $p$.
The normalization $\widetilde{C}$ of $C$ is a disjoint union of $C_1$ and $C_2$ with one extra marked point on each component. Since the two marked points of $C_1$ have spin weights, it follows that the weight associated to the new marked point on $C_1$, considered as a component of the normalization of $\widetilde{C}$, is marked by an $\SO(2r+1)$ weight.
Hence, by repeating the process discussed in Section \ref{factorizationlemma}, we are reduced to the case where $\vec{\lambda}=(\omega_r,\omega_r, Y)$,  $Y\in P_{2s+1}(\SO(2r+1))$, and the case where $\vec{\lambda}=(Y, Y_1, \vec{\omega}_1)$. The rank-level duality in the latter case is a Theorem in \cite{Mukhopadhyay:12}. Hence, we are only left with the admissible triples of the form $(\vec{\lambda}, \vec{\mu}, \vec{\Lambda})$, when $\vec{\lambda}=(Y,\omega_r,\omega_r)$, $\vec{\mu}=(Y^T, (2r+1)\omega_s,(2r+1)\omega_s)$ and $\vec{\Lambda}=(\omega_{\epsilon},\omega_d,\omega_d)$. We now determine which $Y$ are possible. 

\subsubsection{Step V} As in Section \ref{sec:affine}, let  $V_\lambda$
denote the finite dimensional irreducible representation of $\Spin(2r+1)$ with highest weight $\lambda$. By a theorem of Kempf-Ness \cite{KempfNess:88}, we know that  $V_\lambda$ appears in the tensor product decomposition of $V_{\omega_r}\otimes V_{\omega_r}$ if and only if $\lambda \in \{\omega_0,\omega_1,\ldots,\omega_{r-1}, 2\omega_{r}\}$. It follows from  \cite[Prop.\ 4.3]{Beauville:93} that the conformal blocks $\mathcal{V}_{Y,\omega_r,\omega_r}(\mathbb{P}^1,\mathfrak{so}(2r+1),2s+1)$ are one dimensional, where $Y \in \{\omega_0,\ldots \omega_{r-1}, 2\omega_r\}$, and trivial otherwise. We refer to these  as the {\em minimal cases}.

\subsection{Rank-level duality for the minimal cases}The minimal case can be further subdivided into the case when $|Y|$ is odd or even. After dividing into these cases, we will approach the minimal cases by induction. The basic strategy is similar to the strategy of the minimal cases in \cite{Mukhopadhyay:12}.  We will refer to the Appendix for some of the details of the formulas.

To show that the rank-level duality map is injective, it is enough to find vectors 
$$v_1\otimes v_2\otimes v_3 \in \mathcal{H}_{Y}\otimes \mathcal{H}_{Y^T}\otimes \mathcal{H}_{\omega_r}\otimes \mathcal{H}_{(2r+1)\omega_s}\otimes \mathcal{H}_{\omega_r}\otimes \mathcal{H}_{(2r+1)\omega_s}$$
 such that $\langle \Psi\mid v_1\otimes v_2\otimes v_3\rangle \neq 0$, where $\langle{\Psi}\mid $ is  the (up to scalars) unique nonzero element of $\mathcal{V}^*_{\omega_{\epsilon},\omega_d,\omega_d}(\mathbb{P}^1, \mathfrak{so}(2d+1),1)$, and $\epsilon$ is either $0$ or $1$, depending on the parity of $|Y|$.

\subsubsection{The case $|Y|=0$} In this case we choose $v_1=1$, $v_2=\bigwedge_{1\leq i\leq r,-1\leq j\leq -s}\phi_{i,j}$, and $v_3=\bigwedge_{1\leq i\leq r,-1\leq j\leq -s}\phi^{i,j}$.
It is then clear (cf.\ Section \ref{case0dd}) that $\langle \Psi\mid  v_1\otimes v_2\otimes v_2 \rangle \neq 0$. 
%

\subsubsection{The case $Y=\omega_1$}
 Choose $v_1=\phi^{1,0}(-\frac{1}{2})=R(B^{0}_1)\phi^{1,1}(-\frac{1}{2})$ (cf.\ Lemma \ref{lemmak=1} of the Appendix),
$ v_2= L(B^{0}_1)\bigwedge_{1\leq i\leq r,-1\leq j\leq -s}\phi_{i,j}$, and  $v_3=\bigwedge_{1\leq i\leq r,-1\leq j\leq -s}\phi^{i,j}$. Now by a  direct computation (cf.\ Proposition \ref{leftaction}), we get $ v_2=\phi_{1,0}\wedge \bigwedge_{1\leq i\leq r,-1\leq j\leq -s}\phi_{i,j}$. We are now left to evaluate $\langle \Psi \mid  v_1\otimes v_2\otimes v_3\rangle$, and from the discussion in Section \ref{case1dd}, it is nonzero.

\subsubsection{The case $Y=\omega_2$}We need to choose $v_1$, $v_2$ and $v_3$ as before. 
In this case, take $v_1=R^2(B^0_1)\phi^{1,1}(-\frac{1}{2})\phi^{2,1}(-\frac{1}{2})$ and $v_2=L(B^{-1}_2)v$, where as before $v=\bigwedge_{1\leq i\leq r,-1\leq j\leq -s}\phi_{i,j}$. By Proposition \ref{leftaction}, we get $v_2=\phi_{1,0}\wedge \phi_{2,0}\wedge v$ and $v_3=v^{opp}=\bigwedge_{1\leq i\leq r,-1\leq j\leq -s}\phi^{i,j}$. Now by Proposition \ref{k=2}, we get 
$$R^2(B^{0}_1)\phi^{1,1}(-\tfrac{1}{2})\phi^{2,2}(-\tfrac{1}{2})\cdot1=\bigg(2B^{1,0}_{-2,0}(-1)+B^{2,1}_{-1,1}(-1)+B^{2,-1}_{-1,-1}(-1)\bigg)\cdot 1\ .$$

Let the three points be $p_1=0$, $p_2=1$ and $p_3=\infty$ and let $z$ be the local coordinate at the point $0$. Consider $f$ defined by the equation $1/z$. Around $P_1$, the functions $f$ has a pole of order one and hence a zero of order one  around $p_3$. Let $\xi=z-1$ be a coordinate at the point $p_2=1$ and around $p_2$, the function $f$ has the following form $f_1(\xi)=1-\xi+\xi^2-\xi^3+\cdots$. This follows by formally expanding 
$$f(z)=\frac{1}{1+(z-1)}=1-(z-1) + (z-2)^2-(z-3)^3 + \cdots $$

We now use gauge symmetry (cf.\ Section \ref{propertiesofconformalblocks}) to finish the argument 
\begin{align*}
\langle \Psi \mid  R^2(B^0_1)&\phi^{1,1}(-\tfrac{1}{2})\phi^{2,1}(-\tfrac{1}{2})\otimes L(B^{-1}_2)v \otimes v^{opp}\rangle\\
&=\langle \Psi \mid  (2 B^{1,0}_{-2,0}(-1) + B^{2,1}_{-1,1}(-1)+ B^{2,-1}_{-1,-1}(-1) ).1\otimes\phi_{1,0}\wedge \phi_{2,0}\wedge v\otimes v^{opp}\rangle\\
&=2\langle \Psi\mid  1\otimes (-B^{1,0}_{-2,0}+B^{1,0}_{-2,0}(1)-\cdots)(\phi_{1,0}\wedge \phi_{2,0}\wedge v)\otimes v^{opp}\rangle \\
&\qquad + \langle \Psi\mid  1\otimes (-B^{2,1}_{-1,1}+B^{2,1}_{-1,1}(1)-\cdots)(\phi_{1,0}\wedge \phi_{2,0}\wedge v)\otimes v^{opp}\rangle \\
&\qquad + \langle \Psi\mid  1\otimes (-B^{2,-1}_{-1,-1}+B^{2,-1}_{-1,-1}(1)-\cdots)(\phi_{1,0}\wedge \phi_{2,0}\wedge v)\otimes v^{opp}\rangle\\
&=-2\langle\Psi\mid 1\otimes B^{1,0}_{-2,0}\phi_{1,0}\wedge \phi_{2,0}\wedge v\otimes v^{opp}\rangle\ .
\end{align*}
In the above calculation, we use the fact $B^{2,-1}_{-1,-1}(\phi_{1,0}\wedge \phi_{2,0}\wedge v)=B^{2,1}_{-1,1}(\phi_{1,0}\wedge \phi_{2,0}\wedge v)=0$. This is justified by Lemma \ref{khayal}. But now $B^{1,0}_{-2,0}(\phi_{1,0}\wedge \phi_{2,0}\wedge v)=-v$. Hence $\langle \Psi\mid  1\otimes B^{1,0}_{-2,0}(\phi_{1,0}\wedge \phi_{2,0}\wedge v\otimes v^{opp})\rangle\neq 0$ (cf.\ Section \ref{case0dd}). Thus we are done in this case.

\subsubsection{The case $Y=\omega_3$} The highest weight vector for the component $\mathcal{H}_{\omega_3}\otimes \mathcal{H}_{3\omega_1}$ is given by $\phi^{1,1}(-\frac{1}{2})\phi^{2,1}(-\frac{1}{2})\phi^{3,1}(-\frac{1}{2})$. We choose $v_1=R^{3}(B^0_1)\phi^{1,1}(-\frac{1}{2})\phi^{2,1}(-\frac{1}{2})\phi^{3,1}(-\frac{1}{2})$,  $v_2=L(B^0_1)L(B^{-2}_3)v$, and $v_3=v^{opp}$. 
Now by Proposition \ref{leftaction}, we get $L(B^0_1)L(B^{-2}_3)v=\phi_{1,0}\wedge \phi_{2,0}\wedge \phi_{3,0}\wedge v$, and by Proposition \ref{thumri} we get 
\begin{eqnarray*}
&&R^3(B^{0}_{1})\phi^{1,1}(-\tfrac{1}{2})\phi^{2,1}(-\tfrac{1}{2})\phi^{3,1}(-\tfrac{1}{2})= 3!\big[\phi^{1,0}(-\tfrac{1}{2})\phi^{2,0}(-\tfrac{1}{2})\phi^{3,0}(-\tfrac{1}{2})\big]\\
&&-3\big[ \phi^{1,-1}(-\tfrac{1}{2})\phi^{2,0}(-\tfrac{1}{2})\phi^{3,1}(-\tfrac{1}{2}) + \phi^{1,0}(-\tfrac{1}{2})\phi^{2,-1}(-\tfrac{1}{2})\phi^{3,1}(-\tfrac{1}{2})\\
&& + \phi^{1,-1}(-\tfrac{1}{2})\phi^{2,1}(-\tfrac{1}{2})\phi^{3,0}(-\tfrac{1}{2})+ \phi^{1,0}(-\tfrac{1}{2})\phi^{2,1}(-\tfrac{1}{2})\phi^{3,-1}(-\tfrac{1}{2}) \\
&&+ \phi^{1,1}(-\tfrac{1}{2})\phi^{2,-1}(-\tfrac{1}{2})\phi^{3,0}(-\tfrac{1}{2}) + \phi^{1,1}(-\tfrac{1}{2})\phi^{2,0}(-\tfrac{1}{2})\phi^{3,-1}(-\tfrac{1}{2})\big]\ .
\end{eqnarray*}
We can rewrite the above expression in the following ``Kac-Moody" form. 
\begin{eqnarray*}
&&R^3(B^{0}_{1})\phi^{1,1}(-\tfrac{1}{2})\phi^{2,1}(-\tfrac{1}{2})\phi^{3,1}(-\tfrac{1}{2})= 3!B^{2,0}_{-3,0}(-1)\phi^{1,0}(-\tfrac{1}{2})\\
&&-3\big[- B^{1,-1}_{-3,-1}(-1)\phi^{2,0}(-\tfrac{1}{2}) + B^{2,-1}_{-3,-1}(-1) \phi^{1,0}(-\tfrac{1}{2}) + B^{1,-1}_{-2,-1}(-1)\phi^{3,0}(-\tfrac{1}{2})\\ 
&&- B^{3,-1}_{-2,-1}(-1) \phi^{1,0}(-\tfrac{1}{2})  -B^{2,-1}_{-1,-1}(-1)\phi^{3,0}(-\tfrac{1}{2}) + B^{3,-1}_{-1,-1}(-1)\phi^{2,0}(-\tfrac{1}{2})\big]\ .
\end{eqnarray*}

We now evaluate  $\langle \Psi \mid  v_1\otimes v_2\otimes v_3\rangle $ using gauge symmetry (cf.\ Section \ref{propertiesofconformalblocks}) as before. Choose $P_1$, $P_2$ and $P_3$ to be $(0,1,\infty)$ with the obvious coordinates. Expanding $f(z)=1/z$ around $1$ and $\infty$ and applying gauge symmetry, we get that $\langle \Psi\mid  v_1\otimes v_2\otimes v_3\rangle$ is (up to a sign)  equal to $3!\langle \Psi \mid  \phi^{1,0}(-\frac{1}{2}) \otimes B^{2,0}_{-3,0}(\phi_{1,0}\wedge \phi_{2,0}\wedge \phi_{3,0}\wedge v)\otimes v^{opp}\rangle$, which is nonzero by the discussion in Section \ref{case1dd}. Explicitly,
\begin{lemma}\label{extravanishing}
Let $a$ and $b$ be both nonzero integers and $i\neq j$ are both positive integers, then
$B^{2,0}_{-3,0}(\phi_{1,0}\wedge \phi_{2,0}\wedge \phi_{3,0}\wedge v)=-v$, and
$B^{i,a}_{-j,b}(\phi_{1,0}\wedge \phi_{2,0}\wedge \phi_{3,0}\wedge v)=0$.
\end{lemma}

\subsubsection{The general case:  $Y=\omega_{k}$ or $2\omega_r$}
The strategy for the general case is the same as for the previous special case. We choose the  points $(p_1, p_2, p_3)=(0,1,\infty)$. We choose $v_1=R^{k}(B^0_1)\phi^{1,1}(-\frac{1}{2})\wedge\phi^{2,1}(-\frac{1}{2})\wedge\cdots\wedge \phi^{k,1}(-\frac{1}{2})$, $v_2=\phi_{1,0}\wedge\cdots\wedge \phi_{k,0}\wedge v$ and $v_3=v^{opp}$. 
 Using gauge symmetry (cf.\ Section \ref{propertiesofconformalblocks}),
the expression $\langle \Psi\mid  v_1\otimes v_2\otimes v_3 \rangle$ is equal  (up to a sign) to, 
\begin{eqnarray}\label{generalgauge}
k!\langle \Psi \mid  \phi^{1,0}(-\tfrac{1}{2})\wedge \cdots \wedge \phi^{k,0}(-\tfrac{1}{2})\otimes v_2\otimes v_3\rangle\ .
\end{eqnarray}
The above step uses Proposition \ref{extraterms} and a calculation similar to Lemma \ref{extravanishing}. We can rewrite the right hand side of \eqref{generalgauge} as follows:
\begin{itemize}
\item If $k$ is odd,
\begin{eqnarray*}\langle \Psi \mid  \phi^{1,0}(-\tfrac{1}{2}) \wedge \cdots \wedge  \phi^{k,0}(-\tfrac{1}{2})\otimes v_2\otimes v_3\rangle=\langle B^{2,0}_{-3,0}(-1)\cdots B^{k-1,0}_{-k,0}(-1)\phi^{1,0}(-\tfrac{1}{2})\otimes v_2\otimes v_3\rangle\ ;
\end{eqnarray*}
\item If $k$ is even,
\begin{eqnarray*}\langle \Psi \mid  \phi^{1,0}(-\tfrac{1}{2})\wedge  \cdots \wedge \phi^{k,0}(-\tfrac{1}{2})\otimes v_2\otimes v_3\rangle=\langle B^{1,0}_{-2,0}(-1)\cdots B^{k-1,0}_{-k,0}(-1)\otimes v_2\otimes v_3\rangle\ . 
\end{eqnarray*}
\end{itemize}

By using gauge symmetry (cf.\ Section \ref{propertiesofconformalblocks}) we get that up to a sign $\langle \Psi \mid  v_1\otimes v_2\otimes v_3\rangle$ is the following:
\begin{itemize}
\item If $k$ is odd, 
\begin{eqnarray*}
k!\langle \Psi\mid  \phi^{1,0}(-\tfrac{1}{2})\otimes B^{2,0}_{-3,0}\cdots B^{k-1,0}_{-k,0}v_2\otimes v_3=k!\langle \Psi \mid  \phi^{1,0}(-\tfrac{1}{2})\otimes \phi_{1,0}\wedge v \otimes v^{opp}\rangle=k!\ ;
\end{eqnarray*}
\item If $k$ is even, 
\begin{eqnarray*}
k!\langle \Psi\mid  1\otimes \otimes B^{1,0}_{-2,0}\cdots B^{k-1,0}_{-k,0}v_2\otimes v_3=k!\langle \Psi \mid  1\otimes  v \otimes v^{opp}\rangle=k!.
\end{eqnarray*}
\end{itemize}
This completes the proof in the general case. 

\subsection{Key Littlewood-Richardson coefficients}In this section, we prove some basic facts on dimensions of conformal blocks and apply this to reduce the general case of rank-level duality to the  minimal cases.

\begin{proposition}\label{basiccaseLRS}
Let $\lambda \in P_{\ell}(\mathfrak{so}(2r+1))$ and  let $\lambda=Y+\omega_r$ and $Y \in \mathcal{Y}_{r,s}$, then
$\dim_{\mathbb{C}}\mathcal{V}^{*}_{\lambda,\vec{\omega}_1,\omega_r}(\mathfrak{X}, \mathfrak{so}(2r+1),\ell)\neq 0\ ,$ where 
$\vec{\omega}_1$ is a $|Y|$-tuple of $\omega_1$'s at level $\ell$. 
\end{proposition}
\begin{proof}
The proof is by induction on $|Y|$. If $Y$ is zero or one, then it is easy to see that $\mathcal{V}_{\omega_r,\omega_r,\omega_0}(\mathfrak{X},\mathfrak{so}(2r+1),\ell)$ and $\mathcal{V}_{\omega_r,\omega_r,\omega_1}(\mathfrak{X},\mathfrak{so}(2r+1),\ell)$ are both one dimensional. 
Now the inductive step follows by factorization (cf.\ Section \ref{factorizationlemma}). By factorization and Lemma \ref{LRS}, we know that the dimension of 
$ \mathcal{V}_{\lambda,\omega_r, \vec{\omega}_1}(\mathfrak{X},\mathfrak{so}(2r+1),\ell)$ is greater than equal to the dimensions  $\mathcal{V}_{\lambda,\omega_1,\lambda'}(\mathfrak{X},\mathfrak{so}(2r+1),\ell)\otimes \mathcal{V}_{\lambda',\omega_r, \vec{\omega}'_1}(\mathfrak{X},\mathfrak{so}(2r+1),\ell)$. Here $\lambda'=Y'+\omega_r$ and $|Y'|=|Y|-1$ and $\vec{\omega}_1'$ is an $|Y|$-tuple of $\omega_1$. 

\end{proof}

We now determine  which  three point $\mathfrak{so}(2r+1)$, level $2s+1$, conformal blocks  with weights $\vec{\lambda}$ are nonzero. First, we compute the  Littlewood-Richardson numbers following Littlemann \cite{Littelmann:90}.
\begin{lemma}\label{1dimS}
Let $\lambda \in P_{+}(\mathfrak{so}(2r+1))$ and assume that $\lambda$ is of the form $Y +\omega_r$, where $Y \in \mathcal{Y}_{r,s}$. Then the dimension of the space $\operatorname{Hom}_{\mathfrak{so}(2r+1)}(V_{\omega_1}\otimes V_{\lambda}\otimes V_{\mu}, \mathbb{C})$ is nonzero if $\mu$ is either $\lambda$, or is of the form $Y'+\omega_r$, where $Y'$ is obtained by adding or deleting a box  of $Y$. 
\end{lemma}
We use the above proposition to calculate the following dimensions. 
\begin{lemma}\label{LRS} Let $\lambda=\sum_{i=1}^r a_i\omega_i +\omega_r \in P_{\ell}(\mathfrak{so}(2r+1))$. Then,  
$$\dim_\CBbb\mathcal{V}^{*}_{\lambda,\mu, \omega_1}(\mathfrak{X}, \mathfrak{so}(2r+1),\ell)=
\begin{cases}
1& \text{ if }  \mu\in P_{\ell}(\mathfrak{so}(2r+1))\ ,\ \mu \text{ as in Lemma \ref{1dimS};} \\
0& \text{ otherwise.}
\end{cases}
$$
\end{lemma}
\begin{proof}
The proof follows directly from the explicit description of three pointed description of conformal blocks on $\mathbb{P}^1$ as the space of invariants (cf.\ \cite[Prop.\ 4.3]{Beauville:93}). 
\end{proof}

\subsection{Sewing and injectivity}\label{factorizationlemma}
In this section, we discuss the key induction steps in the proof of rank-level duality maps. This strategy has already been used in \cite{BoysalPauly:10, Mukhopadhyay:15}. We recall the details here for completeness. We  begin with an important lemma.

Let $B=\operatorname{Spec}\mathbb{C}\llrrbracket{t}$. Suppose $\mathcal{V}$ and $\mathcal{W}$ are two coherent sheaves such that $\operatorname{rank}\mathcal{V}\leq \operatorname{rank}\mathcal{W}$ and $\mathcal{L}$  be a line bundle on $B$. Suppose $f: \mathcal{V} \rightarrow \mathcal{W}\otimes \mathcal{L}$ be a morphism of vector bundles over $B$. Assume that over $B$ there are isomorphisms:
$\oplus s_i:\mathcal{V}\isorightarrow \oplus_{i\in I} \mathcal{V}_i$, and $\oplus t_j:\oplus_{j\in I}\mathcal{W}_j\isorightarrow \mathcal{W}$, so that if  $f_{i,j}:\mathcal{V}_i\rightarrow \mathcal{W}_j\otimes \mathcal{L}$, then
\begin{itemize}
\item For each $i \in I$, $f_{i,j}=0$ unless $i=j$. 
\item The map $f=\sum_{i} t^{m_i}t_i\circ f_{i,i}\circ s_{i},$ where $m_i$ are nonnegative integers.
\end{itemize}
With the above notation and hypotheses, we have the following easy lemma.
\begin{lemma}\label{geometricinput}
The map $f$ is injective on $B^*=B \backslash\{t=0\}$ if and only if the maps $f_{i,i}$'s are injective for all $i \in I$.
\end{lemma}

\begin{remark}We will sometimes need to use a slightly generalized version of Lemma \ref{geometricinput}. Suppose in the above situation there is an isomorphism $\oplus t_j:\oplus_{j\in J}\mathcal{W}_j\isorightarrow \mathcal{W}$, and an injective map $\delta: I \rightarrow J$  such that $f_{i,j}=0$  unless $j=\delta({i})$. 
Then $f$ is injective if and only if for each $i \in I$, the maps $f_{i,\delta(i)}$'s are injective.
For our applications, the role of $I$ will often be played by the set $\mathcal{Y}_{r,s}$ or $\mathcal{Y}_{r,s-1}$ and the role of $J$  by $\mathcal{Y}_{s,r}$.
\end{remark}

Consider a conformal embedding $\sfrak\rightarrow \frg$. Assume that all level one highest weight integrable modules of $\widehat{\frg}$ decompose with multiplicity one as $\widehat{\sfrak}$-modules. Let $\vec{\Lambda}=(\Lambda_1, \ldots, \Lambda_n)$ be an $n$-tuple of level one highest weights of $\frg$ and $\vec{\lambda}=(\lambda_1,\ldots,\lambda_n)$ be an $n$-tuple of level $\ell$ weights that appear in the branching of $\vec{\Lambda}$. By functoriality of the embedding of $\sfrak \rightarrow \frg$, we get a $\mathbb{C}\llrrbracket{t}$-linear map 
 $\alpha(t) :\mathcal{V}_{\vec{\Lambda}}^{*}(\mathcal{X}, \frg,1) \rightarrow \mathcal{V}_{\vec{\lambda}}^{*}(\mathcal{X}, \sfrak,\ell).$ For ${\lambda}$ appearing the branching of $\Lambda$, we denote by $\alpha_{\Lambda, \lambda}(t)$ the rank-level duality map for the smooth curve $X_0$. 
as follows:
 $$\alpha_{\Lambda, \lambda}(t): \mathcal{V}^{*}_{\vec{\Lambda},\Lambda, \Lambda^{\dagger}}(\widetilde{\mathfrak{X}}_0, \frg,1) \otimes \mathbb{C}\llrrbracket{t}\rightarrow \mathcal{V}^{\ast}_{ \vec{\lambda},\lambda, \lambda^{\dagger}}(\widetilde{\mathfrak{X}}_0, \sfrak,\ell)\otimes \mathbb{C}\llrrbracket{t}\ .$$
We recall the following proposition from \cite{BoysalPauly:10}.
\begin{proposition}\label{keydegen} On $B$, the map $\alpha(t)$ decomposes under factorization/sewing as follows
$$\alpha(t) \circ s_{\Lambda}(t)=\sum_{{\lambda} \in B(\Lambda)}t^{m_{\lambda}}\cdot s_{\lambda}(t)\circ\alpha_{\Lambda,\lambda}(t)\ ,$$ where $m_{\lambda}$ are positive integers given by the formula:
 $ m_{\lambda}=\Delta_{\lambda}(\sfrak,\ell)-\Delta_{\Lambda}(\mathfrak{g},1)$ $($see \eqref{eqn:trace-anomaly}$)$.
\end{proposition}

\begin{remark}
For the Lie algebra $\mathfrak{so}(2r+1)$, it is easy to see that $V_{\lambda}$ is isomorphic to its dual as an $\mathfrak{so}(2r+1)$-module. Hence $\lambda^{\dagger}=\lambda$.
\end{remark}

\subsection{Proof of
Theorem \ref{mainjune16}}
The proof now follows from 
factorization as in the previous section, Lemma \ref{geometricinput}, the rank-level duality for $\SO$-weights in \cite{Mukhopadhyay:12}, and Theorem \ref{braidreducibility}.

\section{Strange duality maps in higher genus} \label{sec:higher-genus}

\subsection{Formulation of the problem} As mentioned in the introduction, the natural map between the special Clifford groups obtained by the tensor product of vector spaces induces one between the moduli stacks 
$p:\mathcal{M}_{2r+1}\times \mathcal{M}_{2s+1}\rightarrow \mathcal{M}_{2d+1}$.
 By a direct calculation we can check that $p^*(\mathcal{P})\simeq \mathcal{P}^{\otimes 2s+1}\boxtimes \mathcal{P}^{\otimes 2r+1}$. Hence, we obtain the  map $SD$ defined in \eqref{eqn:SD}.
Recall that $\dim_\CBbb H^0(\mathcal{M}_{2d+1},\mathcal{P})=2^{2g}$, so by Corollary \ref{cor:spin-dimension} the map $SD$ cannot be an isomorphism. However,
it is natural to ask the following:
\begin{question}  \label{spinconjecture}
For $r,s\geq 1$, is the map $SD$ injective?
\end{question}
\noindent
We shall show that the answer to this question is actually negative for all $r,s$, and in all genus.


\subsection{Action of $J_2(C)$ and the strange duality map}\label{strangedualityj2C}
If $W_r$ and $W_s$ are vector spaces each with a nondegenerate symmetric bilinear form,  then the tensor product $W_d=W_r\otimes W_s$ inherits one as well. This gives an embedding
$\SO(W_r)\times \SO(W_s) \lra \SO(W_d)$.
If $\dim_\CBbb W_r=2r+1$, $\dim_\CBbb W_s=2s+1$, the map above in turn induces one between the corresponding moduli stacks,
\begin{equation} \label{eqn:m}
m:
\mathcal{M}_{\SO(2r+1)}\times \mathcal{M}_{\SO(2s+1)}\longrightarrow \mathcal{M}_{\SO(2d+1)}\ .
\end{equation} 

The embedding of orthogonal groups lifts to one on spin groups.
Then we have a commutative diagram:
\begin{equation}
\begin{split} \label{spin-diagram}
\Scale[.9]{
\xymatrix{
\ZBbb/2\times \ZBbb/2 \ar[d]\ar[r] & \ZBbb/2\ar[d] \\
\Spin(2r+1)\times \Spin(2s+1) \ar[d] \ar[r] & \Spin(2d+1)\ar[d] \\
\SO(2r+1)\times \SO(2s+1)  \ar[r] & \SO(2d+1)
}}
\end{split}
\end{equation}
where the map $\ZBbb/2\times \ZBbb/2 \to \ZBbb/2$ is multiplication.
By results of \cite{BLS:98}, we know that $\mathcal{M}_{2r+1}$ forms a $J_2(C)$ torsor over $\mathcal{M}_{\SO(2r+1)}$.
Hence, from \eqref{spin-diagram}  we get the following commutative diagram of moduli stacks:
\begin{equation}
\begin{split}\label{diagram}
\Scale[.9]{
\xymatrix{
\mathcal{M}_{2r+1}\times \mathcal{M}_{2s+1}\ar[d]^{J_2(C)\times J_2(C)} \ar[r] &\mathcal{M}_{2d+1}\ar[d]^{J_2(C)}\\
\mathcal{M}_{\SO(2r+1)}\times \mathcal{M}_{\SO(2s+1)}\ar[r]^{\qquad m} &\mathcal{M}_{\SO(2d+1)}}}
\end{split}
\end{equation}
 which is equivariant with respect to  the action of $J_2(C)\times J_2(C)$ under the multiplication map $J_2(C)\times J_2(C) \rightarrow J_2(C)$.

The natural inclusion of $\SO(2r+1) \subset \SL(2r+1)$ gives the following commutative diagram of moduli stacks:
$$\Scale[.9]{
\xymatrix{
\mathcal{M}_{\SO(2r+1)}\times \mathcal{M}_{\SO(2s+1)} \ar[d]^{f_1\otimes f_2}\ar[r]^{\qquad m} &\mathcal{M}_{\SO(2d+1)}\ar[d]^{f}\\
\mathcal{M}_{\SL(2r+1)}\times \mathcal{M}_{\SL(2s+1)}\ar[r]^{\qquad p} &\mathcal{M}_{\SL(2d+1)}}}
$$
Let $\mathcal{D}$ be the determinant of cohomology on $\mathcal{M}_{\SL(2d+1)}$ (cf.\ Proposition \ref{LSpfaff}). Also, denote by $\widetilde\Dcal$ the pull-back of $\mathcal{D}$ under $f$. Since we know that $p^{*}{\mathcal{D}}=\mathcal{D}^{\otimes (2s+1)}\boxtimes \mathcal{D}^{\otimes(2r+1)}$, it follows that $m^*{\widetilde{\Dcal}}=\widetilde{\Dcal}_1^{\otimes(2s+1)}\boxtimes \widetilde{\Dcal}_2^{\otimes(2r+1)}$, where $\mathcal{D}_i$ and (resp.\  $\widetilde{\Dcal}_i$) denote the determinants of cohomology and their respective pull-backs.
From Proposition \ref{LSpfaff}, it follows that if we fix a theta characteristic $\kappa$, the pull-back of $\mathcal{P}_{\kappa}$ under the map $m$ in \eqref{eqn:m} is $\mathcal{P}_{\kappa}^{\otimes(2s+1)}\boxtimes \mathcal{P}_{\kappa}^{\otimes(2r+1)}$.

As mentioned before, given $\kappa\in\Th(C)$ we get an action of $J_2(C)$ on the space of global sections $H^0(\mathcal{M}_{2r+1},\mathcal{P}^{\otimes(2s+1)})$, and the above diagram  of moduli stacks commutes and is equivariant with respect to  $J_2(C)\times J_2(C)$. We have the following.

\begin{lemma}\label{equivariantsection}
Let $X$ and $Y$ be two spaces with an action of a group $G$ actions and $f:X\rightarrow Y$ be a $G$-equivariant map. Suppose $\mathcal{L}$ is a line bundle on $Y$ and suppose both $\mathcal{L}$ and $f^*\mathcal{L}$ are $G$-linearized. Then the map of global sections is $G$-equivariant
$$f^*:H^0(Y,\mathcal{L})\longrightarrow H^0(X, f^*(\mathcal{L}))\ .$$
\end{lemma}

\begin{proof}
Let $s$ be a global section of $X$ and we consider $gf^*(s)$. Let $x$ be any element of $X$. Then we get $gf^*(s)(x)=f^*(s)(g^{-1}x)=s(f(g^{-1}x)=s(g^{-1}(f(x))$. On the other hand $f^*(g(s))(x)=(g(s))(f(x)=s(g^{-1}f(x))$. Thus we have the equality $gf^*(s)=f^{*}(gs)$. 
\end{proof}

The commutativity of the diagram \ref{diagram} and Lemma \ref{equivariantsection} implies that the following map of global sections is $J_2(C)\times J_2(C)$ equivariant.
\begin{equation} \label{eqn:rsd}
H^0(\mathcal{M}_{2r+1},\mathcal{P}^{\otimes(2s+1)})^*\otimes H^0(\mathcal{M}_{2s+1},\mathcal{P}^{\otimes(2r+1)})^*\longrightarrow H^0(\mathcal{M}_{2d+1},\mathcal{P})^*\ .
\end{equation}

\begin{lemma}\label{algebra}
Let $V_1$, $V_2$ and $W$ be three vector spaces endowed with an action of a finite abelian group $A$. Let $f:V_1\otimes V_2 \rightarrow W$ be a $A\times A$ equivariant map, where the action of $A\times A$ on $W$ is via multiplication map $A\times A\rightarrow A$. Then $f:V_1^{\chi_1}\otimes V_2^{\chi_2}\rightarrow W^{\chi_3}$ is zero unless $\chi_1=\chi_2=\chi_3$, where $V^{\chi}$ denotes the $\chi$-character spaces of a vector space $V$ with respect to $A$ and $\chi \in \widehat{A}$.
\end{lemma}

\begin{proof}
The proof follows directly by taking character subspaces of $f$ with respect to $A\times A$. This implies that $\chi_1(a_1)\chi_2(a_2)=\chi_3(a_1.a_2)$ for any $a_1$ and $a_2$ in $A$. This is only possible when $\chi_1=\chi_2=\chi_3$.
\end{proof}

Let $V$ be a finite dimensional vector space and $A$  a finite abelian group acting on $V$. The invariant subspace $V^A$ is a $A$-submodule of $V$, and since $A$ is finite, $V$ admits an orthogonal splitting of the form $V=\oplus_{\chi \in \widehat{A}}V^{\chi}$, i.e.\ there is a symmetric nondegenerate bilinear $A$-invariant bilinear form $\{,\}$ such that $\{V^{\chi_1},V^{\chi_2}\}=0$ unless $\chi_1=\chi_2$. Hence $s:(V^{\chi})^*\simeq (V^*)^{\chi}$ canonically. The map is given as if $f\in (V^{\chi})^*$, then extend $f$ to $\tilde{f}$ by the obvious  rule $\tilde{f}(v)=0$ unless $v \in V^{\chi}$. Clearly $\tilde{f} \in (V^*)^{\chi}$.
\begin{lemma}\label{basic1}
Let $V_1$, $V_2$ and $W$ be as in Lemma \ref{algebra} and further assume that $f^*:V_1\rightarrow V_2^*\otimes W$ is injective. Then the induced map between the $\chi$-character spaces $g:V_1^{\chi} \rightarrow (V_2^{\chi})^*\otimes W^{\chi}$ is also injective.
\end{lemma}

\begin{proof} Since the map $f$ is $A\times A$ equivariant under the multiplication map, then it is also equivariant with respect to the subgroup $B=A\times {\id}$. Taking invariants with respect to $B$ implies that the map $f_{\chi}^*:V_1^{\chi}\rightarrow V_2^*\otimes W^{\chi}$ is also injective. 
Now recall that there is a canonical isomorphism $s:(V^{\chi})^*\isorightarrow (V^*)^{\chi}$. Consider the map $V_1^{\chi}\rightarrow V_2^*\otimes W^\chi$ given by the following: 
$V_1^{\chi}
\stackrel{g}{\lra}
 (V_2^{\chi})^{*}\otimes W^{\chi}\stackrel{s}{\lra} (V_2^*)^{\chi}\otimes W^{\chi}\hookrightarrow V_2^*\otimes W^{\chi}$.
 It follows from the definition that the composition of the above map coincides with $f_{\chi}^\ast$, and since the latter is injective so must be $g$.
\end{proof}

By taking invariants in \eqref{eqn:rsd} with respect to the $J_2(C)\times J_2(C)$ action, Lemma \ref{basic1} and ``yes'' to Question \ref{spinconjecture} would imply that the following map is injective:
$$\left[ H^0(\mathcal{M}_{2r+1},\mathcal{P}^{\otimes(2s+1)})^*\right]^{J_2(C)} \longrightarrow \left[ H^0(\mathcal{M}_{2s+1},\mathcal{P}^{\otimes(2r+1)})\right]^{J_2(C)}\otimes \mathbb{C}\cdot s_{\kappa}^{\ast}\ .$$
 Now since $r$ and $s$ are arbitrary, and 
 $\left[H^0(\mathcal{M}_{2r+1},\mathcal{P}^{\otimes(2s+1)})\right]^{J_2(C)}=H^0(\mathcal{M}_{\SO(2r+1)},\mathcal{P}_{\kappa}^{2s+1})$, we get the following : 
\begin{question} \label{soconjecture}
Let $\kappa\in \Th(C)$ and consider the Pfaffian section $s_{\kappa}$ in $H^0(\mathcal{M}_{\SO(2d+1)},\mathcal{P}_{\kappa})$. The pull-back of the Pfaffian divisor induces a strange duality map:
$$s_\kappa^*: H^0(\mathcal{M}_{\SO(2r+1)},\mathcal{P}_{\kappa}^{\otimes(2s+1)})^*\longrightarrow H^0(\mathcal{M}_{\SO(2s+1)},\mathcal{P}_{\kappa}^{\otimes(2r+1)})\ .$$
Is $s_\kappa^\ast$ an isomorphism for all $\kappa$?
\end{question}
The discussion above tells us that an affirmative answer to Question \ref{spinconjecture} implies an affirmative answer to Question \ref{soconjecture}. We now show the converse. 
\begin{proposition} \label{prop:equiv}
Questions \ref{spinconjecture} and \ref{soconjecture} are equivalent.
\end{proposition}
\begin{proof}
Since the $H^{0}(\mathcal{M}_{2d+1},\mathcal{P})^{J_2(C)}$ is one dimensional and generated by $s_{\kappa}$, it follows from Lemma \ref{algebra} that the map:$s_{\kappa}^*: H^0(\mathcal{M}_{\SO(2r+1)},\mathcal{P}_{\kappa_1}^{\otimes(2s+1)})^*\longrightarrow H^0(\mathcal{M}_{\SO(2s+1)},\mathcal{P}_{\kappa_2}^{\otimes(2r+1)})$
 is zero unless $\kappa=\kappa_1=\kappa_2$. Now if the answer to Question \ref{soconjecture} is yes, this implies that for $\kappa=\kappa_1=\kappa_2$, then $s_{\kappa}^*$ is  injective. Since the level is odd, we have $L_{\chi}^{2r+1}=L_{\chi}$, where $L_{\chi}$ is the two torsion line bundle corresponding to the character $\chi$, and so
$$H^0(\mathcal{M}_{2r+1},\mathcal{P}^{\otimes(2r+1)})\simeq \bigoplus_{\kappa \in \Th(C)} H^0(\mathcal{M}_{\SO(2r+1)},\mathcal{P}_{\kappa}^{\otimes(2r+1)})\ .$$ 
It follows that the map $s_\Delta^\ast$ from \eqref{eqn:SD1} is also injective. This implies an affirmative answer to Question \ref{spinconjecture}. In the above decomposition, it is crucial that we are working with odd levels. 
\end{proof}
\subsection{Comparison of  dualities and reduction to genus one}In this section, we reformulate Question  \ref{spinconjecture} in terms of conformal blocks. We use the factorization/sewing  theorem  of conformal blocks to reduce Question  \ref{spinconjecture} to a rank-level duality of conformal blocks on elliptic curves with one marked point (cf.\ Sections \ref{sec:conformal-block} and
 \ref{factorizationlemma}). 

Let $C$ be a stable curve of genus $g$ with one marked point, and let $\mathfrak{X}$ be the data associated to the additional choice of a formal neighborhood around the point. Introduce:
$$
\widetilde{\mathcal{V}}_{\omega_0}(\mathfrak{X},\mathfrak{so}(2r+1),2s+1):=  \mathcal{V}_{\omega_0}(\mathfrak{X}, \mathfrak{so}(2r+1),2s+1)\oplus \mathcal{V}_{(2s+1)\omega_1}(\mathfrak{X}, \mathfrak{so}(2r+1),2s+1)\ .$$
We have the following diagram:
\begin{equation}
\begin{split}\label{diagramcomm}
\Scale[.9]{
\xymatrix{
\widetilde{\mathcal{V}}_{\omega_0}(\mathfrak{X}, \mathfrak{so}(2r+1),2s+1)\otimes \widetilde{\mathcal{V}}_{\omega_0}(\mathfrak{X},\mathfrak{so}(2s+1),2r+1) \ar[r]\ar[d] & \widetilde{\mathcal{V}}_{\omega_0}(\mathfrak{X},\mathfrak{so}(2d+1),1)\ar[d]\\
H^0(\mathcal{M}_{2r+1},\mathcal{P}^{\otimes(2s+1)})^*\otimes H^0(\mathcal{M}_{2s+1},\mathcal{P}^{\otimes(2r+1)})^* \ar[r] & H^0(\mathcal{M}_{2d+1},\mathcal{P})^*
}}
\end{split}
\end{equation}
Here, the vertical arrows are given by Theorem \ref{identification},  and the horizontal arrow on the top is given by the rank-level duality map induced by the branching rule in Section \ref{branching}. The other horizontal arrow is the strange duality map. With the above notation, we have the following.

\begin{proposition}\label{equality}
The rank-level duality and strange duality maps are the same under uniformization, i.e.\ the  diagram \eqref{diagramcomm} commutes. 
\end{proposition}

\begin{proof}
The proof follows from the uniformization theorem of the moduli stacks and is similar to the proof of  \cite[Prop.\ 5.2]{Belkale:09}. We omit the details.
\end{proof}
Recall the notation $B(\Lambda)$ from Section \ref{branching}.

\begin{question}\label{elliptic}Let $E$ be any elliptic curve and $\mathfrak{X}$  associated to $E$ with a formal neighborhood at one marked point. Let $\lambda \in P_{2s+1}(\SO(2r+1))$ (resp.\  $\mu \in P_{2r+1}(\SO(2s+1))$ and $\Lambda \in P_{2d+1}(\SO(2r+1))$ such that $\Lambda$ is either $\omega_0$ or $\omega_1$, and $(\lambda,\mu)\in B(\Lambda)$. Is the following map of conformal blocks  injective:
\begin{eqnarray*}
\mathcal{V}_{ \lambda}(\Xfrak, \mathfrak{so}(2r+1),2s+1)\rightarrow \mathcal{V}^*_{\mu}(\Xfrak, \mathfrak{so}(2s+1),2r+1)\otimes \mathcal{V}_{\Lambda}(\Xfrak, \mathfrak{so}(2d+1),1)\\
\hspace{1cm} \oplus  \mathcal{V}^*_{\sigma(\mu)}(\Xfrak, \mathfrak{so}(2s+1),2r+1)\otimes \mathcal{V}_{\sigma(\Lambda)}(\Xfrak, \mathfrak{so}(2d+1),1)\ ?
\end{eqnarray*}
\end{question}
\begin{proposition} \label{prop:equivalence}
An affirmative answer to
Question \ref{elliptic} implies one for Question \ref{spinconjecture}.
\end{proposition}
\begin{proof}By Proposition \ref{equality}, it is enough to prove it for conformal blocks. Let $C_0$ be a nodal curve with $g$ elliptic tails attached to a $\mathbb{P}^1$. Consider a one parameter family of $\mathcal{C} \rightarrow \operatorname{Spec}(\mathbb{C}\llrrbracket{t})$ such that the generic fiber is smooth and the special fiber is $C_0$. The normalization of $C_0$ is a  $\mathbb{P}^1$ with $g$ marked points, and $g$-elliptic curves each with one marked point. 
By factorization (cf.\ Section \ref{factorizationlemma}), it follows that $\mathcal{V}_{\omega_0}(C_0,\mathfrak{so}(2r+1),2s+1)$ splits up as a direct sum where each component looks like $$\big(\bigotimes_{i=1}^g \mathcal{V}_{\lambda_i}(E,\mathfrak{so}(2r+1),2s+1)\big)\otimes \mathcal{V}_{\vec{\lambda}}(\mathbb{P}^1,\mathfrak{so}(2d+1),1)\ ,$$
and the direct sum is indexed by $g$-tuples $\vec{\lambda}=(\lambda_1,\ldots,\lambda_g)$, 
$\lambda_i\in P_{2s+1}(\SO(2r+1)$. These weights  
 have the special property  that given $\mu$ and $\Lambda$, there exists at most one $\lambda$ such that $(\lambda,\mu)\in B(\Lambda)$ (cf.\ Section \ref{branching}). This guarantees that the conditions in Lemma \ref{geometricinput} are satisfied, and the map in Question \ref{spinconjecture} splits as a direct sum (up to a nonnegative power of the parameter $t$) indexed by the set of $g$-tuple of points in $P_{2s+1}(\SO(2r+1))$.

Now  \cite[Prop.\ 9.11]{Mukhopadhyay:12} and the compatibility of factorization/sewing with rank-level duality (cf.\ Sections \ref{sec:conformal-block} and \ref{factorizationlemma}, and also \cite{TUY:89}), imply that Question \ref{spinconjecture} holds in the affirmative if this is true for Question \ref{elliptic}  and if there is a rank-level duality isomorphism on $\mathbb{P}^1$ with $g$ marked points and weights coming from $P_{2s+1}(\SO(2r+1))$. The rank-level duality isomorphism on $\mathbb{P}^1$ with $n$-marked points and weights coming from $P_{2s+1}(\SO(2r+1))$ has been proved in \cite{Mukhopadhyay:12}. Hence, ``yes'' in  Question \ref{elliptic}  implies ``yes'' in  Question \ref{spinconjecture}.
\end{proof}




\begin{remark} By the same strategy, it is natural to continue by degenerating the elliptic curves and applying factorization to reduce to the case of $\mathbb{P}^1$. However, serious issues occur due to the appearance of spin weights, which were avoided before. The problems are twofold.  They are:
\begin{enumerate}
\item The property that given $\Lambda$ and $\mu$, there exists an unique weights $\lambda$ such that $(\lambda,\mu) \in B(\Lambda)$ fails. The failure of this property means the factorization of rank-level duality maps falls into nondiagonal blocks, so we can not do induction.
\item Rank-level duality isomorphisms on $\mathbb{P}^1$ with spin weights fail to hold. This was explained in Section \ref{failure}. 
\end{enumerate}

\end{remark}

\section{The case of elliptic curves}
In this section, we study the following maps and investigate whether they are injective. An affirmative answer to both would give an affirmative answer to the strange duality question for elliptic curves. However, we shall see that this is in fact not the case (Proposition \ref{sde}). Let $E$ be an elliptic curve with one marked point and a choice of formal coordinate.  The maps are:

\begin{align}
\begin{split} \label{statement1} 
&\mathcal{V}_{\omega_0}(E,\mathfrak{so}(2r+1),2s+1)\lra 
\big( \mathcal{V}^*_{\omega_0}(E,\mathfrak{so}(2s+1),2r+1)\otimes \mathcal{V}_{\omega_0}(E,\mathfrak{so}(2d+1),1)   \\
&\qquad \qquad \quad  \quad \quad  \quad\oplus \mathcal{V}^*_{(2r+1)\omega_1}(E,\mathfrak{so}(2s+1),2r+1)\otimes \mathcal{V}_{\omega_1}(E,\mathfrak{so}(2d+1),1)\big) \ ;
\end{split}
 \end{align}
 
\begin{align}
\begin{split} \label{statement2}
&\mathcal{V}_{(2s+1)\omega_1}(E,\mathfrak{so}(2r+1),2s+1)\lra \big(  \mathcal{V}^*_{\omega_0}(E,\mathfrak{so}(2s+1),2r+1)\otimes \mathcal{V}_{\omega_1}(E,\mathfrak{so}(2d+1),1) \\
&\qquad\qquad \qquad \qquad \quad  \quad \quad  \quad\oplus \mathcal{V}_{(2r+1)\omega_1}^*(E,\mathfrak{so}(2s+1),2r+1)\otimes \mathcal{V}_{\omega_0}(E,\mathfrak{so}(2d+1),1) \big)\ .
\end{split}
\end{align}

\subsection{Factorization for elliptic curves}
We will use factorization to further reduce to the case of $\mathbb{P}^1$ with three marked points. Let us first focus on \eqref{statement2}. By definition of the diagram automorphism $\sigma$, we know that $(2r+1)\omega_1=\sigma(\omega_0)$. Hence, by factorization (cf.\ Section \ref{factorizationlemma}):
\begin{align*}
\dim_{\mathbb{C}}\mathcal{V}_{(2r+1)\omega_1}&(E,\mathfrak{so}(2s+1),2r+1)\\
&=\sum_{\lambda \in P_{2s+1}(\mathfrak{so}(2r+1))}\dim_{\mathbb{C}}\mathcal{V}_{(2r+1)\omega_1,\lambda,\lambda}(\mathbb{P}^1,\mathfrak{so}(2r+1),2s+1) \\ 
&=\sum_{\lambda \in P_{2s+1}(\mathfrak{so}(2r+1))}\dim_{\mathbb{C}}\mathcal{V}_{\omega_0,\sigma(\lambda),\lambda}({\mathbb{P}^1}, \mathfrak{so}(2r+1),2s+1) \qquad \mbox{(cf.\ \cite{FS})}\\
&= | \{ \lambda \in P_{2s+1}(\mathfrak{so}(2r+1))\mid \sigma(\lambda)=\lambda\}|\\
&= |\mathcal{Y}_{r,s}\backslash \mathcal{Y}_{r,s-1}|\ .
\end{align*}
By the above calculation, the next result  proves injectivity of \eqref{statement2}:
\begin{proposition}
The rank-level duality map between the following one dimensional conformal blocks is an isomorphism. 
\begin{align*}
\mathcal{V}_{(2s+1)\omega_1, \lambda, \lambda}(\PBbb^1,\sofrak(2r+1),2s+1)&\longrightarrow \\
 \mathcal{V}^*_{\omega_0,\lambda^*,\lambda^*}&(\PBbb^1,\sofrak(2s+1),2r+1)\otimes \mathcal{V}_{\omega_1,\omega_d,\omega_d}(\PBbb^1,\sofrak(2d+1),1)\ ,
\end{align*}
 where $\lambda= Y+\omega_r$ and $Y\in \mathcal{Y}_{r,s}\backslash \mathcal{Y}_{r,s-1}$ and $\lambda^*=Y^*+\omega_s$ and $Y^*\in\mathcal{Y}_{s,r}$ obtained by taking the transpose of $Y$ and taking the complement in an $r\times s$ box.
\end{proposition}
\begin{proof}
This is a consequence of Theorem \ref{braidreducibility}.
\end{proof}
It remains to investigate the injectivity of \eqref{statement1}. The following appear in the factorizations of the conformal block $\mathcal{V}_{\omega_0}(E, \mathfrak{so}(2r+1),2s+1)$:
\begin{enumerate}
\item $\mathcal{V}_{\lambda,\lambda}(\mathbb{P}^1,\mathfrak{so}(2r+1),2s+1)$ for $\lambda \in \mathcal{Y}_{r,s}$.
\item  $\mathcal{V}_{\lambda,\lambda}(\mathbb{P}^1,\mathfrak{so}(2r+1),2s+1)$ and $\mathcal{V}_{\sigma(\lambda),\sigma(\lambda)}(\mathbb{P}^1,\mathfrak{so}(2r+1),2s+1)$ for $\lambda \in \mathcal{Y}_{r,s-1}+\omega_r$.
\item $\mathcal{V}_{\lambda,\lambda}(\mathbb{P}^1,\mathfrak{so}(2r+1),2s+1)$ for $\lambda \in \mathcal{Y}_{r,s}\backslash \mathcal{Y}_{r,s-1}$.
\end{enumerate}
Thus we need for check injectivity for each of the factors. For the factors of the form in (1), $\lambda$ is a weight of $\SO(2r+1)$, and we only need that the rank-level duality map is an isomorphism for $\lambda \in \mathcal{Y}_{r,s}$:
$$\mathcal{V}_{\lambda,\lambda}(\mathbb{P}^1,\mathfrak{so}(2r+1),2s+1)\rightarrow \mathcal{V}_{\lambda^T,\lambda^T}(\mathbb{P}^1,\mathfrak{so}(2r+1),2s+1)^*\otimes \mathcal{V}_{\omega_{\epsilon}, \omega_{\epsilon}}(\mathbb{P}^1,\mathfrak{so}(2d+1),1)\ ,$$ where $\epsilon$ is zero or one depending on the parity of $|\lambda|$. This is done in \cite{Mukhopadhyay:12}. The argument for $\lambda \in \mathcal{Y}_{r,s}\backslash \mathcal{Y}_{r,s-1}$ follows from Theorem \ref{mainjune16}. Thus, we are only left with the case when $\lambda \in \mathcal{Y}_{r,s-1}+\omega_r$.
For every $\lambda=Y+\omega_r$,  $Y \in \mathcal{Y}_{r,s-1}$, consider the map:
\begin{align}
\begin{split}\label{sdelliptic}
&\mathcal{V}_{\omega_0, \lambda, \lambda}(\mathbb{P}^1,\mathfrak{so}(2r+1),2s+1)\oplus \mathcal{V}_{\omega_0, \sigma(\lambda), \sigma(\lambda)}(\mathbb{P}^1,\mathfrak{so}(2r+1),2s+1) \\
 &\qquad  \longrightarrow  \mathcal{V}^*_{\omega_0,\lambda^*,\lambda^*}(\mathbb{P}^1,\mathfrak{so}(2s+1),2r+1)\otimes \mathcal{V}_{\omega_0,\omega_d,\omega_d}(\mathbb{P}^1,\mathfrak{so}(2d+1),1)\\ 
 &\qquad\qquad\oplus\ \mathcal{V}^*_{(2r+1)\omega_1,\lambda^*,\lambda^*}(\mathbb{P}^1,\mathfrak{so}(2s+1),2r+1)\otimes \mathcal{V}_{\omega_1,\omega_d,\omega_d}(\mathbb{P}^1, \mathfrak{so}(2d+1),1)\ .
  \end{split}
\end{align}
The following is the main result of this section.
\begin{proposition}\label{sde}
The rank-level duality map  in \eqref{sdelliptic} is not injective. 
\end{proposition}

\begin{proof}[Proof of Theorem \ref{thm:SD-false}]
By Propositions \ref{prop:equivalence} and \ref{sde}, it follows that the answer to Question \ref{spinconjecture} is negative. Then Proposition \ref{prop:equiv} completes the proof.
\end{proof}

\begin{remark}
It is easy to see that the dimensions of both the source and  target of \eqref{sdelliptic} is two. Since there are maps between all  components that appear in map Question \ref{elliptic}, it follows from Theorem \ref{braidreducibility} that all entries of the $(2\times 2)$-matrix are nonzero. The proof of Proposition \ref{sde} is broken up to into several steps. We will fix an explicit basis to compute the matrix of the map  \eqref{sdelliptic}  and show that the determinant vanishes. 
\end{remark}

\subsection{Tensor decompositions}
We digress to give an explicit expression for highest weight vectors that is compatible with branching and tensor products associated to the marked points.  
As above, let $\Xfrak$ refer to the data of a curve $C$ with marked points and a choice of local coordinates.
To make the notation more transparent, let us 
denote the inclusion of an abstract highest weight $\widehat\sofrak(2r+1)\oplus \widehat\sofrak(2s+1)$ module  appearing in the branching of 
a highest weight
 $\widehat\sofrak(2d+1)$ module  by
$$
\beta^{\lambda \mu}_\Lambda: \Hcal_\lambda(\sofrak(2r+1))\otimes \Hcal_{\mu}(\sofrak(2s+1))\hookrightarrow\Hcal_{\Lambda}(\sofrak(2d+1))\ .
$$
For an element $v_1\otimes \cdots\otimes v_n \in \Hcal_{\lambda_1}(\gfrak)\otimes\cdots\otimes\Hcal_{\lambda_n}(\gfrak)$, let $[v_1\otimes \cdots \otimes v_n]$ denote the vector in the quotient space of dual conformal blocks, e.g.\
$$[v_1\otimes \cdots \otimes v_n]\in \Vcal_{\vec\lambda}(\mathfrak{X},\sofrak(2r+1), 2s+1) =\Hcal_{\vec\lambda}(\sofrak(2s+1))/\gfrak(\Xfrak)\Hcal_{\vec\lambda}(\sofrak(2s+1))\ ,
$$ where $\vec{\lambda}=(\lambda_1,\dots,\lambda_n)$. Since $\mathfrak{X}$ is fixed, we will drop the notation for the curve $\mathfrak{X}$ from the notation of conformal blocks.
It is easy to check (and we have already used!) the fact that the map 
\begin{align*}
\Vcal_{\lambda_1,\lambda_2,\lambda_3}(\sofrak(2r+1), 2s+1)
&\otimes \Vcal_{\mu_1,\mu_2,\mu_3}(\sofrak(2s+1), 2r+1)
\to \Vcal_{\Lambda_1,\Lambda_2,\Lambda_3}(\sofrak(2d+1), 1)\ , \\
[v_1\otimes v_2\otimes v_3]\otimes [w_1\otimes w_2\otimes w_3]&\mapsto \left[\beta_{\Lambda_1}^{\lambda_1\mu_1}(v_1\otimes w_1)
\otimes \beta_{\Lambda_2}^{\lambda_2\mu_2}(v_2\otimes w_2)\otimes \beta_{\Lambda_3}^{\lambda_3\mu_3}(v_3\otimes w_3)\right]\ ,
\end{align*}
is well-defined.

Let $v_\lambda\in \Hcal_{\omega_d}(\sofrak(2d+1))$ be the highest weight vector of the component $\mathcal{H}_{\lambda}(\sofrak(2r+1))\otimes \mathcal{H}_{\lambda^*}(\sofrak(2s+1)) $, and $\bar{v}_{\lambda}\in \Hcal_{\omega_d}(\sofrak(2d+1))$  the highest weight vector of the component $\mathcal{H}_{\sigma({\lambda})}(\sofrak(2r+1))\otimes \mathcal{H}_{\lambda^*}(\sofrak(2s+1))$ as expressed explicitly as an element of $\mathcal{H}_{\omega_d}(\sofrak(2d+1))$ as in Section \ref{sec:highest-weight}. We denote $v^{\lambda}$ and $\bar{v}^{\lambda}$ to be the corresponding opposite highest weights again expressed explicitly. 

Choose highest weight vectors $v_1$ and $v_2$ of $\mathcal{H}_{\lambda}(\sofrak(2r+1))$ and $\mathcal{H}_{\lambda^*}(\sofrak(2s+1))$ such that $\beta_{\omega_d}^{\lambda \lambda^\ast}(v_1 \otimes v_2) =v_{\lambda}$. Similarly choose $v^1$ and $v^2$ for the opposite highest weight such that $\beta_{\omega_d}^{\lambda\lambda^\ast}(v^1 \otimes v^2)=v^{\lambda}$. Let $\bar{v}_1$ be such that $\beta_{\omega_d}^{\sigma(\lambda)\lambda^\ast}(\bar{v}_1\otimes v_2)=\bar{v}_{\lambda}$ and similarly choose $\bar{v}^1$ for the corresponding $\bar{v}^{\lambda}$. 

Let $\tilde{v}$ be a vector in $\Hcal_{\omega_1}(\sofrak(2d+1))$ which is equal  (up to a scalar)  to  $R^{2r+1}(B^0_1)$ acting on the highest weight vector of the component $\mathcal{H}_{\omega_0}(\sofrak(2r+1))\otimes \mathcal{H}_{(2r+1)\omega_1}(\sofrak(2s+1))$ expressed explicitly in Clifford algebra terms. Since we get the vector $\tilde{v}$ by acting only on the right component
in the tensor decomposition, it follows that
 $\tilde{v}$ is a pure tensor in $\mathcal{H}_{\omega_0}(\sofrak(2r+1))\otimes \mathcal{H}_{(2r+1)\omega_1}(\sofrak(2s+1))$. Hence, we can choose $x \in \mathcal{H}_{(2r+1)\omega_1}(\sofrak(2s+1))$ such that $\beta_{\omega_1}^{0, (2r+1)\omega_1}(1 \otimes x)=\tilde{v}$. 

We are interested in the following classes.

\par\noindent
$\bullet\ \mathcal{V}_{\omega_0,\lambda,\lambda}(\mathfrak{so}(2r+1), 2s+1)\otimes \mathcal{V}_{\omega_0,\lambda^*,\lambda^*}(\mathfrak{so}(2s+1), 2r+1) \rightarrow \mathcal{V}_{\omega_0,\omega_d,\omega_d}(\mathfrak{so}(2d+1), 1)$
\begin{align*}
[1\otimes v_1 \otimes v^1]\otimes [1\otimes v_2 \otimes v^2]
&\mapsto [\beta_{0}^{0,0}(1\otimes 1)\otimes \beta_{\omega_d}^{\lambda\lambda^\ast}(v_1\otimes v_2)\otimes \beta_{\omega_d}^{\lambda\lambda^\ast}(v^1\otimes v^2)] \\
&=[1\otimes v_{\lambda}\otimes v^{\lambda}]\ .
\end{align*}

\par\noindent
$\bullet\ \mathcal{V}_{\omega_0,\sigma(\lambda),\sigma(\lambda)}(\mathfrak{so}(2r+1), 2s+1)\otimes \mathcal{V}_{\omega_0,\lambda^*,\lambda^*}(\mathfrak{so}(2s+1), 2r+1) \rightarrow \mathcal{V}_{\omega_0,\omega_d,\omega_d}(\mathfrak{so}(2d+1), 1)$

\begin{align*}
[1\otimes \bar v_1 \otimes \bar v^1]\otimes [1\otimes v_2 \otimes v^2]
&\mapsto [\beta_{0}^{0,0}(1\otimes 1)\otimes \beta_{\omega_d}^{\sigma(\lambda)\lambda^\ast}(\bar v_1\otimes v_2)\otimes \beta_{\omega_d}^{\sigma(\lambda)\lambda^\ast}(\bar v^1\otimes v^2)] \\
&=[1\otimes \bar{v}_{\lambda}\otimes \bar{v}^{\lambda}]\ .
\end{align*}
%

\par\noindent
$\bullet\ \mathcal{V}_{\omega_0,\lambda,\lambda}(\mathfrak{so}(2r+1), 2s+1)\otimes \mathcal{V}_{(2r+1)\omega_1,\lambda^*,\lambda^*}(\mathfrak{so}(2s+1), 2r+1) \rightarrow \mathcal{V}_{\omega_1,\omega_d,\omega_d}(\mathfrak{so}(2d+1), 1)$

\begin{align*}
[1\otimes v_1 \otimes v^1]\otimes [x\otimes v_2 \otimes v^2]
&\mapsto [\beta_{\omega_1}^{0,(2r+1)\omega_1}(1\otimes x)\otimes \beta_{\omega_d}^{\lambda\lambda^\ast}(v_1\otimes v_2)\otimes \beta_{\omega_d}^{\lambda\lambda^\ast}(v^1\otimes v^2)] \\
&=[\tilde v\otimes v_{\lambda}\otimes v^{\lambda}]\ .
\end{align*}
%

\par\noindent
$\bullet\ \mathcal{V}_{\omega_0,\sigma(\lambda),\sigma(\lambda)}(\mathfrak{so}(2r+1), 2s+1)\otimes \mathcal{V}_{(2r+1)\omega_1,\lambda^*,\lambda^*}(\mathfrak{so}(2s+1), 2r+1) \rightarrow \mathcal{V}_{\omega_1,\omega_d,\omega_d}(\mathfrak{so}(2d+1), 1)$
\begin{align*}
[1\otimes \bar v_1 \otimes \bar v^1]\otimes [x\otimes v_2 \otimes v^2]
&\mapsto [\beta_{\omega_1}^{0,(2r+1)\omega_1}(1\otimes x)\otimes \beta_{\omega_d}^{\sigma(\lambda)\lambda^\ast}(\bar v_1\otimes v_2)\otimes \beta_{\omega_d}^{\sigma(\lambda)\lambda^\ast}(\bar v^1\otimes v^2)] \\
&=[\tilde v\otimes \bar{v}_{\lambda}\otimes \bar{v}^{\lambda}]\ .
\end{align*}
%

\subsection{Case by case analysis}
\subsubsection{The case $(\omega_0,\lambda,\lambda)\times (\omega_0,\lambda^*,\lambda^*)\rightarrow (\omega_0,\omega_d,\omega_d)$}
Let $\lambda=Y+\omega_r$, $Y\in \Ycal_{r,s}$.
 Let $v_{\lambda}=\bigwedge_{\widetilde{Y}_{i,j}=\,\blacksquare}\phi_{i,j}$ as in Proposition \ref{kacmoody3}. Similarly let $v^{\lambda}=\bigwedge_{\widetilde{Y}_{i,j}=\,\blacksquare}\phi^{i,j}$. Let $\langle \Psi|$ denote the unique up to constants nonzero element of $\mathcal{V}^*_{\omega_0,\omega_d,\omega_d}(\mathbb{P}^1,\mathfrak{so}(2d+1),1)$. This was discussed in Section \ref{case0dd}. 
Let $B(\ ,\ )$ be the nondegenerate bilinear form on $W_d$. The choice of $v_{\lambda}$ (resp.\  $v^{\lambda}$) implies that $\langle \Psi| 1\otimes v_{\lambda}\otimes v^{\lambda}\rangle$ is up to a sign equal to $\prod_{\widetilde{Y}_{i,j}=\, \blacksquare} B(\phi_{i,j},\phi^{i,j})$ which is nonzero.



\subsubsection{The case $(\omega_0,\lambda,\lambda)\times ((2r+1)\omega_1,\lambda^*,\lambda^*)\rightarrow (\omega_1,\omega_d,\omega_d)$} Let $\lambda=Y+\omega_r$ and further assume that $\lambda \in \mathcal{Y}_{r,s-1}$. We choose $v_{\lambda}$ and $v^{\lambda}$ as above in the previous case. We need to choose a vector in $\mathcal{H}_{\omega_0}(\sofrak(2r+1))\otimes \mathcal{H}_{(2r+1)\omega_1}(\sofrak(2s+1))$ as an explicit element in $\mathcal{H}_{\omega_1}(\sofrak(2d+1))$. We choose the vector: 
$\tilde{v}:=  B^{2,0}_{2,0}(-1)\cdots B^{r,0}_{r,0}(-1)B^{0,0}_{1,0}(-1)\phi^{1,0}(-\tfrac{1}{2} )$.
Let $\langle \widetilde{\Psi}\mid $ be the unique nonzero vector of $\mathcal{V}^*_{\omega_1,\omega_d,\omega_d}(\mathbb{P}^1,\mathfrak{so}(2d+1),1)$ normalized such that it is equal to the one induced from Clifford multiplication (cf.\ Section \ref{case1dd}). We now evaluate the following using gauge symmetry (cf.\ Section \ref{propertiesofconformalblocks}) and choosing the points to be $(1,0,\infty)$ with the obvious local coordinates.
\begin{eqnarray*}
\langle \widetilde{\Psi} \mid  \tilde{v}\otimes v_{\lambda}\otimes v^{\lambda}\rangle &=& \langle \widetilde{\Psi} \mid B^{2,0}_{2,0}(-1)\cdots B^{r,0}_{r,0}(-1)B^{0,0}_{1,0}(-1)\phi^{1,0}(-\tfrac{1}{2} )\otimes v_{\lambda}\otimes v^{\lambda}\rangle\\
&=&(-1)\langle \widetilde{\Psi}\mid B^{r,0}_{r,0}(-1)B^{0,0}_{1,0}(-1)\phi^{1,0}(-\tfrac{1}{2} )\otimes B^{2,0}_{2,0}v_{\lambda}\otimes v^{\lambda}\rangle\\
&=&\frac{1}{2}\langle \widetilde{\Psi}\mid B^{r,0}_{r,0}(-1)B^{0,0}_{1,0}(-1)\phi^{1,0}(-\tfrac{1}{2} )\otimes v_{\lambda}\otimes v^{\lambda}\rangle\\
&=&\frac{1}{2^{r-1}}\langle \widetilde{\Psi}\mid  B^{0,0}_{1,0}(-1)\phi^{1,0}(-\tfrac{1}{2} )\otimes v_{\lambda}\otimes v^{\lambda}\rangle\\
&=&\frac{1}{2^{r-1}}\langle \widetilde{\Psi}\mid  \phi^{1,0}(-\tfrac{1}{2} )\otimes B^{0,0}_{1,0}v_{\lambda}\otimes v^{\lambda}\rangle\\
&=& \frac{1}{2^{r-1}}\frac{(-1)^{rs-|\lambda|+1}}{\sqrt{2}}\langle \widetilde{\Psi}\mid  \phi^{1,0}(-\tfrac{1}{2} )\otimes \phi_{1,0}\wedge v_{\lambda}\otimes v^{\lambda}\rangle\ .
\end{eqnarray*}
We used the following in the  calculation above.
\begin{lemma} With the same notation,
\begin{enumerate}
\item for $i>0$, we get $B^{i,0}_{i,0}v_{\lambda}=\frac{1}{2}v_{\lambda}\ ;$ 
\item $ B^{0,0}_{1,0}v_{\lambda}=\frac{(-1)^{rs-|\lambda|+1}}{\sqrt{2}}\phi_{1,0} \wedge v_{\lambda}$.
\end{enumerate}
\end{lemma}
\begin{proof} The proof is as usual by a direct calculation. The most important observation is that $\sqrt{2}\phi^{0,0}v=(-1)^p v$, where $p$ is the degree of $v$ in $\bigwedge W_d^{-}$.
\end{proof}

\subsubsection{The case $(\omega_0,\sigma(\lambda),\sigma(\lambda))\times(\omega_0,\lambda^*,\lambda^*)\rightarrow (\omega_0,\omega_d,\omega_d)$}Let $\lambda=Y+\omega_r$ and further assume that the number of boxes in the first row of $Y$ is $s-1$. Assume that $\overline{Y}=Y+ y_{1,-1}$, so that the number of boxes of $\overline{Y}$ is $s$. Let $\bar{\lambda}=\overline{Y}+\omega_r$ and $v_{\bar{\lambda}}=\bigwedge_{\widetilde{\overline{Y}}_{i,j}=\,\blacksquare}\phi_{i,j}$ (for notation, see  Section \ref{sec:highest-weight}).

 We need to choose a highest weight vector of $\mathcal{H}_{\sigma(\lambda)}(\sofrak(2r+1))\otimes \mathcal{H}_{\lambda^\ast}(\sofrak(2s+1))$ as an explicit element in $\mathcal{H}_{\omega_d}(\sofrak(2d+1))$. Applying Corollary \ref{kacmoody4}, we choose $\bar{v}_{\lambda}:=  B^{1,1}_{0,0}(-1)v_{\bar{\lambda}}.$ Let $\langle \Psi\mid $ be the unique element of the $\mathcal{V}^*_{\omega_0,\omega_d,\omega_d}(\mathbb{P}^1,\mathfrak{so}(2d+1),1)$ We want to evaluate the following:

\begin{align*}
\langle \Psi \mid  1\otimes \bar{v}_{\lambda}\otimes \bar{v}^{\lambda}\rangle&=\langle \Psi \mid  1\otimes B^{1,1}_{0,0}(-1)v_{\bar{\lambda}}\otimes B^{0,0}_{1,1}(-1)v^{\bar{\lambda}}\rangle \\
&=-\langle \Psi \mid  1 \otimes  B^{0,0}_{1,1}(1) B^{1,1}_{0,0}(-1)v_{\bar{\lambda}}\otimes v^{\bar{\lambda}}\rangle\\
&=-\langle \Psi \mid  1 \otimes \big( [B^{0,0}_{1,1}, B^{1,1}_{0,0}]+ (B^{0,0}_{1,1},B^{1,1}_{0,0})c\big)v_{\bar{\lambda}}\otimes v^{\bar{\lambda}}\rangle\\
&=-\langle \Psi \mid  1 \otimes [B^{0,0}_{1,1},B^{1,1}_{0,0}]v_{\bar{\lambda}}\otimes v^{\bar{\lambda}}\rangle-\langle \Psi \mid  1 \otimes v_{\bar{\lambda}}\otimes v^{\bar{\lambda}}\rangle\\
&=\langle \Psi\mid  1\otimes B^{1,1}_{1,1}v_{\bar{\lambda}}\otimes v^{\bar{\lambda}}\rangle -\langle \Psi \mid  1 \otimes v_{\bar{\lambda}}\otimes v^{\bar{\lambda}}\rangle\\
&=\frac{1}{2}\langle \Psi\mid  1\otimes v_{\bar{\lambda}}\otimes v^{\bar{\lambda}}\rangle-\langle \Psi \mid  1 \otimes v_{\bar{\lambda}}\otimes v^{\bar{\lambda}}\rangle\\
&=-\frac{1}{2}\langle \Psi \mid  1 \otimes v_{\bar{\lambda}}\otimes v^{\bar{\lambda}}\rangle\ .
\end{align*}


\subsubsection{The case $(\omega_0,\sigma(\lambda),\sigma(\lambda))\times((2r+1)\omega_1,\lambda^*,\lambda^*)\rightarrow (\omega_1,\omega_d,\omega_d)$}Let $\lambda$ be such that the Young diagram associated to the weight $\lambda$ has exactly $s-1$ boxes in the first row. Let $\langle \widetilde{\Psi}\mid $ be the unique nonzero element of $\mathcal{V}^*_{\omega_1,\omega_d,\omega_d}(\mathbb{P}^1,\mathfrak{so}(2d+1),1)$ normalized such that it is equal to the Clifford multiplication. We want to evaluate 
\begin{align*}
\langle \widetilde{\Psi} \mid  \tilde{v}\otimes \bar{v}_{\lambda}\otimes \bar{v}^{\lambda}\rangle&=\frac{1}{2}\langle \Psi \mid  \tilde{v} \otimes v_{\bar{\lambda}} \otimes v^{\bar{\lambda}}\rangle
=\frac{1}{2^{r}}\langle \widetilde{\Psi} \mid  B^{0,0}_{1,0}(-1)\phi^{1,0}(-\tfrac{1}{2} )\otimes v_{\bar{\lambda}}\otimes v^{\bar{\lambda}}\rangle\\
&=\frac{1}{2^{r}}\langle \widetilde{\Psi}\mid  \phi^{1,0}(-\tfrac{1}{2} )\otimes B^{0,0}_{1,0}v_{\bar{\lambda}}\otimes v^{\bar{\lambda}}\rangle\\
&=\frac{1}{2^{r}}\langle \widetilde{\Psi}\mid  \phi^{1,0}(-\tfrac{1}{2} )\otimes \normalorder{\phi^{0,0}\phi_{1,0}}\, v_{\bar{\lambda}}\otimes v^{\bar{\lambda}}\rangle\\
&=\frac{1}{2^{r}}\langle \widetilde{\Psi}\mid  \phi^{1,0}(-\tfrac{1}{2} )\otimes \phi^{0,0}(\phi_{1,0}\wedge v_{\bar{\lambda}})\otimes v^{\bar{\lambda}}\rangle\\
&=\frac{(-1)^{(rs-|\bar{\lambda}|+1)}}{(\sqrt{2})^{2r+1}}\langle \widetilde{\Psi}\mid \phi^{1,0}(-\tfrac{1}{2} )\otimes \phi_{1,0}\wedge v_{\bar{\lambda}}\otimes v^{\bar{\lambda}}\rangle\\
&=\frac{(-1)^{(rs-|\lambda|)}}{(\sqrt{2})^{2r+1}}\langle \widetilde{\Psi}\mid \phi^{1,0}(-\tfrac{1}{2} )\otimes \phi_{1,0}\wedge v_{\bar{\lambda}}\otimes v^{\bar{\lambda}}\rangle\ .\\
\end{align*}

\subsection{Proof of Proposition \ref{sde}} Let $\lambda=Y+\omega_r$,  $Y\in\Ycal_{r,s-1}$, and we further assume that the number of boxes in the first row of $Y$ is exactly $s-1$. The previous calculations tell us that the matrix of the map \eqref{sdelliptic} is the following: $A_{\lambda}:= $
$$
 \left[
\begin{matrix}
    \langle \Psi\mid 1\otimes v_{\lambda}\otimes v^{\lambda}\rangle  & -\frac{1}{a^2}\langle \Psi\mid 1\otimes v_{\bar{\lambda}}\otimes v^{\bar{\lambda}}\rangle\\[6pt]
   \frac{(-1)^{rs+1-|\lambda|}}{a^{2r-1}}\langle \widetilde{\Psi}\mid \phi^{1,0}(-\tfrac{1}{2} )\otimes \phi_{1,0}\wedge v_{\lambda}\otimes v^{\lambda}\rangle   & \frac{(-1)^{(rs-|\lambda|)}}{a^{2r+1}}\langle \widetilde{\Psi}\mid \phi^{1,0}(-\tfrac{1}{2} )\otimes \phi_{1,0}\wedge v_{\bar{\lambda}}\otimes v^{\bar{\lambda}}\rangle \\
	\end{matrix}
\right].
$$
Then the determinant is (up to a constant): $\det{A_{\lambda}}\simeq$
$$\langle \Psi\mid 1\otimes v_{\lambda}\otimes v^{\lambda}\rangle \langle \widetilde{\Psi}\mid \phi^{1,0}(-\tfrac{1}{2} )\otimes \phi_{1,0}\wedge v_{\bar{\lambda}}\otimes v^{\bar{\lambda}}\rangle 
-\langle \Psi\mid 1\otimes v_{\bar{\lambda}}\otimes v^{\bar{\lambda}}\rangle  \langle \widetilde{\Psi}\mid \phi^{1,0}(-\tfrac{1}{2} )\otimes \phi_{1,0}\wedge v_{\lambda}\otimes v^{\lambda}\rangle \ .
$$
But now by the construction of $\langle \widetilde{\Psi}\mid $ and $\langle \Psi\mid $ in Sections \ref{case1dd} and \ref{case0dd}, we get 
\begin{itemize}
\item $\langle \widetilde{\Psi } \mid  \phi^{1,0}(-\tfrac{1}{2} )\otimes \phi_{1,0}\wedge v_{\lambda}\otimes v^{\lambda}\rangle=\langle \Psi\mid  1\otimes v_{\lambda}\otimes v^{\lambda}\rangle$;
\item $\langle \widetilde{\Psi } \mid  \phi^{1,0}(-\tfrac{1}{2} )\otimes \phi_{1,0}\wedge v_{\overline{\lambda}}\otimes v^{\overline{\lambda}}\rangle=\langle \Psi\mid  1\otimes v_{\overline{\lambda}}\otimes v^{\overline{\lambda}}\rangle$.
\end{itemize}
It follows that $\det A_\lambda=0$. This completes the proof.





\appendix
\section{Computations in the Clifford algebra}
In this section, we compute some vectors  in the highest weight modules as explicit elements in the infinite dimensional Clifford algebra. 
\subsection{Action of $L(B^{i}_j)$}
Consider the  rectangle $r\times s$ as a Young diagram $Y$ where the rows are indexed by integer in $\{1,\ldots, r\}$ and the columns by $\{-s,\ldots,-1\}$. Let $({i,j})$ be the coordinates of $Y$ and let $v:=  \bigwedge_{\tilde{Y}_{i,j}=\,\blacksquare}\phi_{i,j}$ (cf.\ Section \ref{sec:highest-weight}).

\begin{proposition}\label{leftaction}Let $v$ as before be the highest weight vector of the component with highest weight $(\omega_r, (2r+1)\omega_s)$. Let $0\leq k\leq m\leq r$, then
$$L(B^{-k}_{(k+1)})L(B^{-(k+2)}_{k+3})\cdots L(B^{-(m-3)}_{m-2})L(B^{-(m-1)}_{m})\cdot v=\phi_{k,0}\wedge \phi_{k+1,0}\wedge\cdots \wedge \phi_{m,0}\wedge v\ .$$

\end{proposition}

\subsection{Action of $R^k(B^0_1)$}Let $v_k=\phi^{1,1}(-\tfrac{1}{2} )\phi^{2,1}(-\tfrac{1}{2} )\cdots \phi^{k,1}(-\tfrac{1}{2} )\cdot 1$. In this section, we want to give explicit expressions for $R^k(B^i_j)v_k$. First, consider the case when $k=1$.

\begin{lemma}\label{lemma1}\label{lemmak=1} Consider the highest weight vector $\phi^{1,1}(-\tfrac{1}{2} )\cdot 1$ of the component with highest weight $(\omega_1,\omega_1)$. Then
$R(B^0_1)\phi^{1,1}(-\tfrac{1}{2} )\cdot 1=\phi^{1,0}(-\tfrac{1}{2} )$.
\end{lemma}

Now we want to compute $R^2(B^0_1)v_2$. We first have the following lemma
\begin{lemma} \label{lemmak=2} 
 For the component $(\omega_2,2\omega_1)$,  $\phi^{1,1}(-\tfrac{1}{2} )\phi^{2,2}(-\tfrac{1}{2} )\cdot 1$, is a highest weight vector.
Moreover,
\begin{equation} \label{eqn:rv2}
R(B^0_1)v_2= \phi^{1,0}(-\tfrac{1}{2} )\phi^{2,1}(-\tfrac{1}{2} )+\phi^{1,1}(-\tfrac{1}{2} )\phi^{2,0}(-\tfrac{1}{2} )\ .
\end{equation}
\end{lemma}
\begin{proof}
The proof is by a direct computation. As before, we know that $R(B^0_1)$ acts as $\sum_{-r\leq q\leq r} \normalorder{\phi^{q,0}\phi_{q,1}}$ on the infinite dimensional Clifford algebra for $\widehat{\mathfrak{so}}(2d+1)$. Hence,
\begin{align*}
R(B^{0}_{1})\phi^{1,1}(-\tfrac{1}{2} )\phi^{2,1}(-\tfrac{1}{2} ) &
= \sum_{-r \leq q\leq r}\bigl[\phi^{q,0}(-\tfrac{1}{2} )\phi_{q,1}(\tfrac{1}{2} )-\phi_{q,1}(-\tfrac{1}{2} )\phi^{q,0}(\tfrac{1}{2} )\bigr]\phi^{1,1}(-\tfrac{1}{2} )\phi^{2,1}(-\tfrac{1}{2} )\\
&= \sum_{-r \leq q\leq r}\bigl[\phi^{q,0}(-\tfrac{1}{2} )\phi_{q,1}(\tfrac{1}{2} )\bigr] \phi^{1,1}(-\tfrac{1}{2} )\phi^{2,1}(-\tfrac{1}{2} )\\
&= \phi^{1,0}(-\tfrac{1}{2} )\phi^{2,1}(-\tfrac{1}{2} )-\phi^{2,0}(-\tfrac{1}{2} )\phi^{1,1}(-\tfrac{1}{2} )\ .
\end{align*}
\end{proof}

\begin{proposition}\label{k=2} We have: 
$$
R^2(B^0_1)v_2=2\left[\phi^{1,0}(-\tfrac{1}{2} )\phi^{2,0}(-\tfrac{1}{2} )\right)-\left(\phi^{1,-1}(-\tfrac{1}{2} )\phi^{2,1}(-\tfrac{1}{2} )+\phi^{1,1}(-\tfrac{1}{2} )\phi^{2,-1}(-\tfrac{1}{2} )\right]\ .
$$
\end{proposition}

\begin{proof} Compute using \eqref{eqn:rv2},
\begin{align*}
R(B^{0}_{1})\phi^{1,0}(-\tfrac{1}{2} )\phi^{2,1}(-\tfrac{1}{2} )
&= \sum_{-r \leq q\leq r}\bigl[\phi^{q,0}(-\tfrac{1}{2} )\phi_{q,1}(\tfrac{1}{2} )-\phi_{q,1}(-\tfrac{1}{2} )\phi^{q,0}(\tfrac{1}{2} )\bigr]\phi^{1,0}(-\tfrac{1}{2} )\phi^{2,1}(-\tfrac{1}{2} )\\
&=-\phi^{2,0}(-\tfrac{1}{2} )\phi^{1,0}(-\tfrac{1}{2} )-\phi_{-1,1}(-\tfrac{1}{2} )\phi^{2,1}(-\tfrac{1}{2} )\\
&= \phi^{1,0}(-\tfrac{1}{2} )\phi_{-2,0}(-\tfrac{1}{2} )+\phi^{2,1}(-\tfrac{1}{2} )\phi_{-1,1}(-\tfrac{1}{2} )\\
R(B^{0}_{1})\phi^{2,0}(-\tfrac{1}{2} )\phi^{1,1}(-\tfrac{1}{2} )
&= \sum_{-r \leq q\leq r}\bigl[\phi^{q,0}(-\tfrac{1}{2} )\phi_{q,1}(\tfrac{1}{2} )-\phi_{q,1}(-\tfrac{1}{2} )\phi^{q,0}(\tfrac{1}{2} )\bigr]\phi^{2,0}(-\tfrac{1}{2} )\phi^{1,1}(-\tfrac{1}{2} )\\
&=-\phi^{1,0}(-\tfrac{1}{2} )\phi^{2,0}(-\tfrac{1}{2} )-\phi_{-2,1}(-\tfrac{1}{2} )\phi^{1,1}(-\tfrac{1}{2} )\\
&= -\phi^{1,0}(-\tfrac{1}{2} )\phi_{-2,0}(-\tfrac{1}{2} )-\phi^{2,-1}(-\tfrac{1}{2} )\phi_{-1,-1}(-\tfrac{1}{2} )\ .\\
\end{align*}
\end{proof}

We use the following calculation in the proof of strange duality for the pair $(\omega_2,\omega_r,\omega_r)$ and $(2\omega_1,(2r+1)\omega_s, (2r+1)\omega_s)$.
\begin{lemma}\label{khayal}
Let $ w=\phi_{1,0}\wedge \phi_{2,0}\wedge \bigwedge_{ 1\leq i \leq r, -s\leq j\leq -1}\phi_{i,j}$. Then the following hold in $\mathcal{H}_{\omega_r}(\sofrak(2r+1))\otimes \mathcal{H}_{(2s+1)\omega_s}(\sofrak(2s+1))$:
$$
B^{2,1}_{-1,1}w=B^{2,-1}_{-1,-1}w=0\ ;\ 
 B^{1,0}_{-2,0}w=\bigwedge_{ 1\leq i \leq r, -s\leq j\leq -1}\phi_{i,j}\ .
 $$
\end{lemma}

Next we   compute $R^3(B^0_1)v_3$. Our strategy is same as the previous steps.

\begin{proposition}\label{thumri} We have:
 \begin{align*}
R^3(B^0_1)v_3
 &=6\left[\phi^{1,0}(-\tfrac{1}{2} )\phi^{2,0}(-\tfrac{1}{2} )\phi^{3,0}(-\tfrac{1}{2} )\right]-3\bigl[ \phi^{1,-1}(-\tfrac{1}{2} )\phi^{2,0}(-\tfrac{1}{2} )\phi^{3,1}(-\tfrac{1}{2} )\\
&\qquad+ \phi^{1,0}(-\tfrac{1}{2} )\phi^{2,-1}(-\tfrac{1}{2} )\phi^{3,1}(-\tfrac{1}{2} )+ \phi^{1,-1}(-\tfrac{1}{2} )\phi^{2,1}(-\tfrac{1}{2} )\phi^{3,0}(-\tfrac{1}{2} )\\
&\qquad + \phi^{1,0}(-\tfrac{1}{2} )\phi^{2,1}(-\tfrac{1}{2} )\phi^{3,-1}(-\tfrac{1}{2} ) + \phi^{1,1}(-\tfrac{1}{2} )\phi^{2,-1}(-\tfrac{1}{2} )\phi^{3,0}(-\tfrac{1}{2} )\\
&\qquad + \phi^{1,1}(-\tfrac{1}{2} )\phi^{2,0}(-\tfrac{1}{2} )\phi^{3,-1}(-\tfrac{1}{2} )\bigr]\ .
\end{align*}

\end{proposition}

\begin{proof}
The proof  follows by applying the expression for $R(B^0_1)$ successively:
\begin{align*}
R(B^0_1)v_3
 &=\phi^{1,0}(-\tfrac{1}{2} )\phi^{2,1}(-\tfrac{1}{2} )\phi^{3,1}(-\tfrac{1}{2} )+ \phi^{1,1}(-\tfrac{1}{2} )\phi^{2,0}(-\tfrac{1}{2} )\phi^{3,1}(-\tfrac{1}{2} )\\
&\qquad +  \phi^{1,1}(-\tfrac{1}{2} )\phi^{2,1}(-\tfrac{1}{2} )\phi^{3,0}(-\tfrac{1}{2} )\ . \\
R^2(B^0_1)v_3
 &=
2\bigl(\phi^{1,0}(-\tfrac{1}{2} )\phi^{2,0}(-\tfrac{1}{2} )\phi^{3,1}(-\tfrac{1}{2} )+\phi^{1,0}(-\tfrac{1}{2} \phi^{2,1}(-\tfrac{1}{2} )\phi^{3,0}(-\tfrac{1}{2} )\\
& \qquad +\phi^{1,1}(-\tfrac{1}{2} \phi^{2,0}(-\tfrac{1}{2} )\phi^{3,0}(-\tfrac{1}{2} )\bigr)\\
& \qquad -\bigl(\phi^{1,1}(-\tfrac{1}{2} \phi^{2,1}(-\tfrac{1}{2} )\phi^{3,-1}(-\tfrac{1}{2} )+\phi^{1,1}(-\tfrac{1}{2} )\phi^{2,-1}(-\tfrac{1}{2} )\phi^{3,1}(-\tfrac{1}{2} )\\
& \qquad +\phi^{1,-1}(-\tfrac{1}{2} \phi^{2,1}(-\tfrac{1}{2} )\phi^{3,1}(-\tfrac{1}{2} )\bigr).
\end{align*}
and acting once more by  $R(B^0_1)$.
\end{proof}

%
%

We now gather these calculations into the following algorithm:
\begin{itemize}
\item If $v_k=\phi^{1,1}(-\tfrac{1}{2} )\cdots \phi^{k,1}(-\tfrac{1}{2} )$, then the $\mathfrak{h}_2$-weight of $R^k(B^0_1)v_k$ is zero, where $\mathfrak{h}_2$ is the Cartan subalgebra of $\mathfrak{so}(2s+1)$. 
\item The expression for $R(B^0_1)$, viewed as an operator on the Clifford module for $\widehat{\mathfrak{so}}(2d+1)$, implies that
\begin{itemize}
\item if $v =\phi^{1,a_1}(-\tfrac{1}{2} )\cdots \phi^{k,a_k}(-\tfrac{1}{2} )$, where $0\leq a_1+\cdots+a_k\leq k$, and each $a_i\in \{-1,0,1\}$, then the action of $R({B^0_1})$ on $v$ is a sum of expressions of the form $\phi^{1,b_1}(-\tfrac{1}{2} )\cdots \phi^{k,b_k}(-\tfrac{1}{2} )$, where exactly one of the $b_i$'s is different from $a_i$;
\item the operator $R(B^0_1)$ can change an $a_i=1$ to $b_i=0$, or $a_i=0$ to $b_i=-1$. In the latter case, this introduces a minus sign in front of the new expression. In particular for each expression $\phi^{1,b_1}(-\tfrac{1}{2} )\cdots \phi^{k,b_k}(-\tfrac{1}{2} )$ appearing in $R(B^0_1)v$, we get $b_1+\cdots + b_k+1=a_1+\cdots+a_k$. For examples, see the previous lemmas. 
\end{itemize}
\item Thus, applying  the operator $R(B^0_1)$ to $v_k$, $k$-times, we get an expression which is a sum of terms of the form $(-1)^m\phi^{1,c_1}(-\tfrac{1}{2} )\cdots \phi^{k,c_k}(-\tfrac{1}{2} )$, with multiplicities, where $c_1+\cdots+c_k=0$, and each $-1\leq c_i\leq 1$, and $m$ is the number of $(-1)$'s appearing among the $c_i$'s. 
\item The multiplicity of the expression $\phi^{1,0}(-\tfrac{1}{2} )\cdots \phi^{k,0}(-\tfrac{1}{2} )$ is $k!$. 
\end{itemize}
To summarize, we have the following.
\begin{proposition}\label{extraterms}
As an element of $\mathcal{H}_{\omega_k}(\sofrak(2r+1))\otimes \mathcal{H}_{k\omega_1}(\sofrak(2s+1))$,
the vector $R^k(B^0_1)v_k$  is of the  form $k!\phi^{1,0}(-\tfrac{1}{2} )\cdots \phi^{k,0}(-\tfrac{1}{2} )$, plus a sum of terms of the form $B^{i,a}_{-j,b}(-1)w$, where $i\neq j$ are positive integers and $a$, $b$ are nonzero.
\end{proposition}



\bibliographystyle{amsplain}
\bibliography{./papers}

\end{document}